\documentclass[10pt]{amsart}
\usepackage{color}
\usepackage{graphicx}
\usepackage{mathtools,mathrsfs}
\usepackage{tikz}
\usepackage{hyperref}
\usepackage{arydshln}
\usetikzlibrary{calc,chains,cd}

\usepackage{amsmath,amsthm,amsfonts,amssymb,latexsym,verbatim,bbm}
\usepackage{subcaption}
\usepackage[shortlabels]{enumitem}
\usepackage{afterpage}
\usepackage{mathtools}
\usepackage[mathscr]{euscript}

\usepackage[T1]{fontenc}
\usepackage{blkarray}

\numberwithin{equation}{section}
\theoremstyle{definition}
\newtheorem{theorem}{Theorem}[section]
\newtheorem{thm}[theorem]{Theorem}
\newtheorem{lemma}[theorem]{Lemma}
\newtheorem{proposition}[theorem]{Proposition}
\newtheorem{prop}[theorem]{Proposition}
\newtheorem{example}[theorem]{Example}
\newtheorem{definition}[theorem]{Definition}
\newtheorem{defn}[theorem]{Definition}

\newtheorem{corollary}[theorem]{Corollary}
\newtheorem{cor}[theorem]{Corollary}

\newtheorem{problem}[theorem]{Problem}
\newtheorem{rmk}[theorem]{Remark}

\newcommand{\iso}{\cong}
\newcommand{\arr}{\rightarrow}

\newcommand{\Lin}{\operatorname{Lin}}
\newcommand{\Graphs}{\operatorname{Graphs}}

\newcommand{\C}{\mathbb{C}}
\newcommand{\Z}{\mathbb{Z}}

\newcommand{\Id}{\mathbbm{1}}

\newcommand{\mcA}{\mathcal{A}}
\newcommand{\mcB}{\mathcal{B}}

\newcommand{\mcF}{\mathcal{F}}
\newcommand{\mcG}{\mathcal{G}}
\newcommand{\mcH}{\mathcal{H}}
\newcommand{\mcI}{\mathcal{I}}

\newcommand{\mcO}{\mathcal{O}}
\newcommand{\mcP}{\mathcal{P}}

\newcommand{\mcS}{\mathcal{S}}

\newcommand{\mcV}{\mathcal{V}}
\newcommand{\mcW}{\mathcal{W}}

\newcommand{\mbF}{\mathbb{F}}

\newcommand{\msO}{\mathscr{O}}

\newcommand{\msX}{\mathscr{X}}
\newcommand{\msY}{\mathscr{Y}}

\newcommand{\mpP}{\mathsf{P}}

\newcommand{\Dih}{\textnormal{Dih}}

\newcommand{\tr}{\text{tr}}

\title{Arkhipov's theorem, graph minors, and linear system nonlocal games}
\author[C. Paddock]{Connor Paddock$^{1,2}$}
\author[V. Russo]{Vincent Russo$^{5, 6}$}
\author[T. Silverthorne]{Turner Silverthorne$^{4}$}
\author[W. Slofstra]{William Slofstra$^{1,3}$}

\address[1]{Institute for Quantum Computing, University of Waterloo}
\address[2]{Department of Combinatorics \& Optimization, University of Waterloo}
\address[3]{Department of Pure Mathematics, University of Waterloo}
\address[4]{Department of Mathematics, University of Toronto}
\address[5]{Modellicity, Inc.}
\address[6]{Unitary Fund}

\email{cpaulpad@uwaterloo.ca}
\email{vincent.russo@modellicity.com}
\email{turner.silverthorne@utoronto.ca}
\email{weslofst@uwaterloo.ca}

\begin{document}

\begin{abstract}
The perfect quantum strategies of a linear system game correspond to certain
representations of its solution group. We study the solution groups of graph
incidence games, which are linear system games in which the underlying linear
system is the incidence system of a (non-properly) two-coloured graph. While it
is undecidable to determine whether a general linear system game has a perfect
quantum strategy, for graph incidence games this problem is solved by
Arkhipov's theorem, which states that the graph incidence game of a connected
graph has a perfect quantum strategy if and only if it either has a perfect
classical strategy, or the graph is nonplanar. Arkhipov's criterion can be
rephrased as a forbidden minor condition on connected two-coloured graphs. We
extend Arkhipov's theorem by showing that, for graph incidence games
of connected two-coloured graphs, every quotient closed property of the solution
group has a forbidden minor characterization. We rederive Arkhipov's theorem
from the group theoretic point of view, and then find the forbidden minors for
two new properties: finiteness and abelianness. Our methods are entirely
combinatorial, and finding the forbidden minors for other quotient closed
properties seems to be an interesting combinatorial problem. 
\end{abstract}

\maketitle
\section{Introduction}\label{sec:introduction}

In quantum information, a nonlocal game is a type of cooperative game used to
demonstrate the power of entanglement. Linear system nonlocal games are a class
of nonlocal games in which the players try to prove to a referee that a linear
system over a finite field (in this paper we restrict to $\Z_2$) has a solution.
Specifically, if $Ax=b$ is a linear system over $\Z_2$, then the associated
nonlocal game has a perfect deterministic strategy if and only if the system
$Ax=b$ has a solution. It was first observed by Mermin and Peres that there are
linear systems $Ax=b$ which do not have a solution, but where the associated
game can be played perfectly if the players share an entangled quantum state
\cite{Mermin90, Peres91}.
In general, whether or not the game associated to $Ax=b$ has a perfect quantum strategy
is controlled by a finitely presented group called the \emph{solution group} of
$Ax=b$ \cite{CM14, CLS17}. This group has a distinguished central generator $J$
such that $J^2=1$, and the game has a perfect quantum strategy (resp. perfect
commuting-operator strategy\footnote{There are several different models of
quantum strategies.  \emph{Quantum strategies} often refers to the most
restrictive model of finite-dimensional strategies. \emph{Commuting-operator
strategies} refers to the most permissive model of infinite-dimensional
strategies.}) if and only if $J$ is non-trivial in a finite-dimensional
representation of the solution group (resp. is non-trivial in the solution
group). It is shown in \cite{Slof16} that any finitely-presented group can be
embedded in a solution group, so in this sense solution groups can be as
complicated as arbitrary finitely-presented groups.  In particular, it is
undecidable to determine whether $J$ is non-trivial in the solution group
\cite{Slof16}, and also undecidable to determine whether $J$ is non-trivial in
finite-dimensional representations of the solution group \cite{Slof17}.

In this paper, we study the subclass of linear system games in which every
variable of the linear system $Ax=b$ appears in exactly two equations, or
equivalently, in which each column of $A$ contains exactly two non-zero
entries. An $n \times m$ matrix $A$ satisfies this condition if and only if it
is the incidence matrix $\mcI(G)$ of a graph $G$ with $n$ vertices and $m$
edges. After making this identification, the vector $b \in \Z_2^n$ can be
regarded as a function from the vertices of $G$ to $\Z_2$, or in other words as a
(not necessarily proper) $\Z_2$-colouring of $G$.  Hence we refer to this subclass of
linear system games as \emph{graph incidence games}.  Given a $\Z_2$-coloured
graph $(G,b)$, we let $\mcG(G,b)$ be the associated graph incidence game, and
$\Gamma(G,b)$ be the associated solution group, which we now call the
\emph{graph incidence group}. We let $J_{G,b}$ denote the distinguished central
element of $\Gamma(G,b)$, although we refer to this element as $J$ if the
$\Z_2$-coloured graph $(G,b)$ is clear from context.

There is a simple criterion for whether a graph incidence game $\mcG(G,b)$ has
a perfect deterministic strategy: if $G$ is a connected graph, then the linear
system $\mcI(G) x = b$ has a solution if and only if $b$ has even parity, where
the parity of a colouring is the sum $\sum_{v \in V(G)} b(v)$ in $\Z_2$. If $G$
is not connected, then $\mcI(G) x = b$ has a solution if and only if the
restriction of $b$ to each connected component has even parity. If $G$ is
connected and $b$ has odd parity, then $\mcG(G,b)$ no longer has a perfect
deterministic strategy, but there are still graphs $G$ such that $\mcG(G,b)$
has a perfect quantum strategy. In fact, Mermin and Peres' original examples of
linear systems with perfect quantum strategies and no perfect deterministic
strategies---the \emph{magic square} and \emph{magic pentagram} games---are
examples of graph incidence games. The magic square game is $\mcG(K_{3,3},b)$,
where $K_{r,s}$ is the complete bipartite graph with $r$ vertices in one
partition and $s$ vertices in another, and the magic
pentagram is $\mcG(K_{5},b)$, where $K_r$ is the complete graph on $r$
vertices. Recall that Wagner's theorem famously states that a graph is
nonplanar if and only if it contains $K_{3,3}$ or $K_5$ as a graph minor
\cite{Wag37}. The following theorem of Arkhipov shows that this connection
between planarity and quantum strategies for $\mcG(G,b)$ is not a coincidence:
\begin{theorem}[Arkhipov's theorem \cite{Ark12}]\label{thm:Arkhipov}
    If $G$ is a connected graph, then the graph incidence game $\mcG(G,b)$ has
    a perfect quantum strategy if and only if either $b$ has even parity, or
    $b$ has odd parity and $G$ is nonplanar.
\end{theorem}
Another way to state Arkhipov's theorem is that if $G$ is connected, then
$\mcG(G,b)$ has a perfect quantum strategy and no perfect classical strategy if
and only if $b$ has odd parity and $G$ is nonplanar. The theorem also extends
easily to disconnected graphs; however, to avoid complicating theorem statements, we
focus on connected graphs in the introduction, and handle
disconnected graphs later. Also, although it is not stated in \cite{Ark12}, the
proof of Arkhipov's theorem implies that $\mcG(G,b)$ has a perfect quantum
strategy if and only if $\mcG(G,b)$ has a perfect commuting-operator strategy.

Since graph planarity can be tested in linear time~\cite{HT74}, Arkhipov's
theorem implies that it is easy to tell if $\mcG(G,b)$ has a perfect quantum
strategy (or equivalently if $J=1$ in the graph incidence group $\Gamma(G,b)$).
This suggests that while there are interesting examples of graph incidence
games, graph incidence groups are more tractable than the solution groups of
general linear systems. Arkhipov's theorem also suggests a connection between
graph incidence groups and graph minors. The purpose of this paper is to
develop these two points further. In particular we show that, for a natural
extension of the notion of graph minors to $\Z_2$-coloured graphs, there
is a strong connection between graph incidence groups and $\Z_2$-coloured graph
minors:
\begin{lemma}\label{lem:main}
    If $(H,c)$ is a $\Z_2$-coloured graph minor of a $\Z_2$-coloured
    graph $(G,b)$, then there is a surjective group homomorphism $\Gamma(G,b) \arr
    \Gamma(H,c)$ sending $J_{G,b}\mapsto J_{H,c}$.
\end{lemma}
A \emph{group over $\Z_2$} is a group $\Phi$ with a distinguished central
element $J_\Phi$ such that $J_\Phi^2=1$, and a \emph{morphism $\Phi \arr \Psi$ of groups over
$\Z_2$} is a group homomorphism $\Phi \arr \Psi$ sending $J_{\Phi} \arr J_{\Psi}$ (this
terminology comes from \cite{Slof17}). Although we won't use this statement,
the proof of Lemma \ref{lem:main} implies that $(G,b) \mapsto \Gamma(G,b)$ is a
functor from the category of $\Z_2$-coloured graphs with $\Z_2$-coloured minor
operations to the category of groups over $\Z_2$ with surjective homomorphisms.
It's also natural to define the graph incidence group $\Gamma(G)$ of an
ordinary graph $G$ without the $\Z_2$-colouring, and this gives a functor
from the category of graphs with the usual graph minor operations to the
category of groups with surjective homomorphisms.

A property $\mpP$ of groups is quotient closed if for every surjective
homomorphism $\Phi \arr \Psi$ between groups $\Phi$ and $\Psi$, if $\mpP$ holds
for $\Phi$ then it also holds for $\Psi$. Similarly, we say that a property
$\mpP$ of groups over $\Z_2$ is quotient closed if for every surjective
homomorphism $\Phi \arr \Psi$ of groups over $\Z_2$, if $\mpP$ holds for
$(\Phi,J_\Phi)$ then $\mpP$ holds for $(\Psi,J_\Psi)$. There are many
well-known quotient closed properties of groups, including abelianness,
finiteness, nilpotency, solvability, amenability, and Kazhdan's property (T). Any
property of groups is also a property of groups over $\Z_2$, but properties
of groups over $\Z_2$ can also refer to the distinguished element $J_{\Phi}$.
For instance, ``$J_\Phi=1$'' is a quotient closed property of groups over $\Z_2$,
but does not make sense as a group property. If $\mpP$ is a quotient closed property of groups
over $\Z_2$, then ``$\Gamma(G,b)$ has property $\mpP$'' is a $\Z_2$-coloured
minor-closed property of $\Z_2$-coloured graphs. The Robertson-Seymour theorem
states that if $\mpP'$ is a minor-closed property of graphs, then there is a
finite set $\mcF$ of graphs such that $G$ has $\mpP'$ if and only if $G$ avoids
$\mcF$, meaning that $G$ does not contain any graphs in $\mcF$ as a
minor~\cite{RS04}. Using the Robertson-Seymour theorem and Lemma
\ref{lem:main}, it is possible to prove:
\begin{cor}\label{cor:main}
    If $\mpP$ is a quotient closed property of groups over $\Z_2$, then
    there is a finite set $\mcF$ of $\Z_2$-coloured graphs such that if $G$ is
    connected, then $\Gamma(G,b)$ has $\mpP$ if and only if $(G,b)$
    avoids $\mcF$.
\end{cor}
In particular, if $\mpP$ is quotient closed, then it is possible to check
whether $\Gamma(G,b)$ has $\mpP$ in time polynomial in the size of $G$.
The set $\mcF$ of graphs is called a set of \emph{forbidden minors} for $\mpP$.

The restriction to connected graphs in Corollary \ref{cor:main} is necessary
when dealing with properties of groups over $\Z_2$, although for many natural
properties (including $J_\Phi=1$), characterizing the connected case is enough
to characterize the disconnected case. As with Lemma \ref{lem:main}, there is
also a version of
Corollary \ref{cor:main} for ordinary graphs $G$ and groups
$\Gamma(G)$, and for this version the restriction to connected graphs is
unnecessary. After giving more background on graph incidence games and
groups in Section \ref{sec:nonlocal}, the proofs of Lemma \ref{lem:main} and
Corollary \ref{cor:main}, along with the versions of these results for
uncoloured graphs, are given in Section~\ref{sec:minors}.

In the language of Corollary \ref{cor:main}, Arkhipov's theorem is equivalent
to the statement that for a connected graph $G$, $J_{G,b}=1$ if and only if
$(G,b)$ avoids $(K_{3,3},b)$ and $(K_5,b)$ with $b$ odd parity, and $(K_1,b)$
with $b$ even parity, where $K_1$ is the single vertex graph. Corollary
\ref{cor:main} explains why having a perfect quantum strategy can be
characterized by avoiding a finite number of graphs. However, it doesn't
explain why, in the odd parity case, the minors are exactly the minors for
planarity. An intuitive explanation for the connection with planarity is
provided by the fact that relations in a group can be captured by certain
planar graphs, called \emph{pictures}. We reprove Arkhipov's theorem using
Lemma \ref{lem:main} and pictures in Section~\ref{sec:pic}, and observe that
Arkhipov's theorem can be thought of as a stronger version of a result of
Archdeacon and Richter that a graph is planar if and only if it has an odd
planar cover \cite{AR90}. Although our proof of Arkhipov's theorem is phrased
in a different language, at its core our proof is quite similar to the original
proof, with the exception that our proof uses algebraic graph minors, while
Arkhipov's original proof uses topological graph minors.

Corollary \ref{cor:main} raises the question of whether we can find the
forbidden graphs for other quotient closed properties of graph incidence
groups. One particularly interesting property is finiteness. When
$\Gamma(G,b)$ is finite, perfect strategies for $\mcG(G,b)$ are direct sums of
irreducible representations of $\Gamma(G,b)$ (see Corollary
\ref{cor:finitestrats}). Using the Gowers-Hatami stability theorem \cite{GH17},
Coladangelo and Stark~\cite{CS17a} show that if $\Gamma(G,b)$ is finite then
$\mcG(G,b)$ is robust, in the sense that, after applying local isometries,
every almost-perfect strategy is close to a perfect strategy. Robustness is
important in the study of self-testing and device-independent protocols in
quantum information, see e.g.~\cite{McK12, NV16,AC17} for a small sample of
results. Using Lemma \ref{lem:main}, we give the following characterization of
finiteness:
\begin{theorem}\label{thm:finite}
    The graph incidence group $\Gamma(G,b)$ is finite if and only if $G$ avoids
    $C_2 \sqcup C_2$ and $K_{3,6}$.
\end{theorem}
Here $C_2$ is the $2$-cycle, i.e.~a multigraph with $2$ vertices and $2$ edges
between them, and $C_2 \sqcup C_2$ is the graph with two connected components
each isomorphic to $C_2$.
A graph contains $C_2 \sqcup C_2$ as a minor if and only if it has two
vertex disjoint cycles. The characterization in Theorem \ref{thm:finite} is in
terms of the usual minors of $G$, rather than $\Z_2$-coloured minors of
$(G,b)$, because finiteness of $\Gamma(G,b)$ turns out to be independent of $b$.

We can also characterize when $\Gamma(G,b)$ is abelian:
\begin{theorem}\label{thm:abelian}
    If $(G,b)$ is a $\Z_2$-coloured graph and $G$ is connected,
    then $\Gamma(G,b)$ is abelian if and only if $(G,b)$ avoids the graphs
    $(K_{3,3},b')$ with $b'$ odd parity, $(K_5,b')$ with $b'$ odd parity,
    $(K_{3,4},b')$ with $b'$ even parity, and $(C_2 \sqcup C_2,b')$ with $b'$
    any parity.
\end{theorem}
In terms of ordinary minors, if $G$ is connected then $\Gamma(G,b)$ is abelian
if and only if either $b$ is even and $G$ avoids $C_2 \sqcup C_2$ and
$K_{3,4}$, or $b$ is odd and $G$ avoids $K_{3,3}$, $K_5$, and $C_2 \sqcup C_2$
(i.e.~$G$ is planar and does not have two disjoint cycles). If $b$ is odd and
$G$ is planar, then $\mcG(G,b)$ does not have any perfect strategies. However,
when $b$ is even the game $\mcG(G,b)$ always has a perfect deterministic
strategy. In this case, the group $\Gamma(G,b)$ is abelian if and only if all
perfect quantum strategies are direct sums of deterministic strategies on the
support of their state (see Corollary \ref{cor:abelian}).

The key to the proofs of both Theorems \ref{thm:finite} and \ref{thm:abelian}
is that $\Gamma(C_2 \sqcup C_2) = \Z_2 \ast \Z_2$, and hence is infinite and
nonabelian. Thus, by Lemma \ref{lem:main}, if $G$ contains two disjoint cycles
then $\Gamma(G,b)$ must be infinite and nonabelian. We can then use a theorem
of Lovasz \cite{Lov65} which characterizes graphs which do not have two
disjoint cycles. The graph incidence groups of most of the graphs in this
characterization do not have interesting structure, but an exception is
the family of graphs $K_{3,n}$, which we analyze in Subsection \ref{SS:K3n}.
The games $\mcG(K_{m,n},b)$ have also recently been studied in \cite{AW20,AW22}
under the title of magic rectangle games, although our results do not
seem to overlap. The proofs of Theorems \ref{thm:finite} and \ref{thm:abelian}
are given in Section \ref{sec:2cycles}.  The implications of these results for
graph incidence games is explained further in Subsection \ref{SS:consequences}.

As mentioned above, there are many other quotient closed properties of groups,
and finding the minors for these properties seems to be an interesting problem.
We finish the paper in Section \ref{sec:open} with some remarks
about these open problems.

\subsection{Acknowledgements}

CP thanks Jim Geelen and Rose McCarthy for helpful conversations.  We are
grateful for computing resources provided by IQC, along with IT support
provided by Steve Weiss.  WS is supported by NSERC DG 2018-03968 and an Alfred
P. Sloan Research Fellowship.

\section{Graph incidence nonlocal games and groups}\label{sec:nonlocal}

In this section we define graph incidence nonlocal games, deterministic and
quantum strategies for such games, and graph incidence groups. Most of the
concepts in this section come from the theory of linear system games in
\cite{CM14, CLS17}, although we elaborate a bit in
Corollaries \ref{cor:finitestrats} and \ref{cor:abelian} to explain the
implications of finiteness and abelianness of the graph incidence group for
quantum strategies.

Let $G$ be a graph with vertex set $V(G)$ of size $n$, and edge set $E(G)$ of
size $m$. Throughout the paper, graphs are allowed to have multiple edges
between vertices, but loops are not allowed. We say that $G$ is simple if there
is at most one edge between every pair of vertices. For a vertex $v\in V(G)$,
we let $E(v)$ denote the subset of edges $e\in E(G)$ incident to $v$. Recall
that the incidence matrix $\mcI(G)$ of a graph $G$ is the $n\times m$ matrix
whose $(v,e)$th entry is $1$ whenever $e\in E(v)$ and $0$ otherwise. Let
$b:V(G) \arr \Z_2$ be a (not necessarily proper) vertex $\Z_2$-colouring of
$G$. A solution $\hat{x}$ of the linear system $\mcI(G) x=b$ is a function
$\hat{x} : E(G) \arr \Z_2$ assigning a label $\hat{x}(e) \in \Z_2$ to every $e
\in E(G)$, such that $\sum_{e \in E(v)} \hat{x}(e) = b(v)$ for every $v
\in V(G)$. A graph where $\mcI(G) x=b$ does not have a solution is shown in
Figure~\ref{fig:det_strat}.

\begin{figure}[h!]
    \begin{center}
        \includegraphics[scale=1]{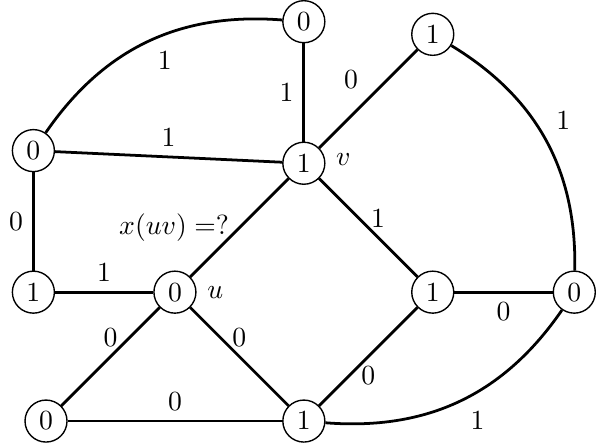}
        \caption{Assigning $0$ or $1$ to the edge $uv$ would be inconsistent
            with the colouring of $G$, so there is no way to complete the given
            edge assignment to a solution of $\mcI(G) x=b$. In fact, for this graph $\mcI(G) x=b$
            has no solution.}
        \label{fig:det_strat}
    \end{center}
\end{figure}

The graph incidence game $\mcG(G,b)$ associated to $(G,b)$ is a two-player
nonlocal game, in which the players try to demonstrate that they have a solution of
$\mcI(G) x=b$. In the game, the referee sends a vertex $v \in V(G)$ to the first player,
and another vertex $u \in V(G)$ to the second player. The players, who are unable
to communicate once the game starts, reply to the referee with functions
$f:E(v) \arr \Z_2$ and $g:E(u) \arr \Z_2$ respectively, which satisfy the
equations
\begin{equation}\label{eq:outputcondition}
    \sum_{e\in E(v)} f(e)=b(v)\quad \text{ and} \quad  \sum_{e\in E(u)}
        g(e)=b(u)\,.
\end{equation}
They win if and only if $f(e)=g(e)$ for every $e \in E(u) \cap E(v)$.  In other
words, the players assign a value from $\Z_2$ to each edge incident to their
given vertex, and they win if the assigned values agree on all edges incident
to both $u$ and $v$. If $u \neq v$, then the assigned values must agree on all
edges between $u$ and $v$, while if $u = v$ then to win the functions must
satisfy $f(e) = g(e)$ for all $e \in E(v)$.

\subsection{Perfect classical strategies for graph incidence games}\label{sec:det_space}
A strategy for a graph incidence game is said to be deterministic if each
player's answer depends only on the vertex they receive. Formally,
a deterministic strategy for a graph incidence game is specified by two
collections of functions $\{f_v:E(v)\arr \Z_2\}_{v\in V(G)}$ and $\{g_u:E(u)\arr
\Z_2\}_{u\in V(G)}$, such that $\sum_{e\in E(v)} f_v(e)= \sum_{e \in E(v)} g_v(e) = b(v)$
for all $v \in V(G)$. Such a strategy is \emph{perfect} if the players win on
every combination of inputs, i.e.~if for every $(v,u) \in V(G) \times V(G)$,
$f_v(e) = g_u(e)$ for all $e \in E(v) \cap E(u)$. In particular, if
$\left(\{f_v\}, \{g_u\}\right)$ is perfect, then $f_v = g_v$ for all $v \in V(G)$,
and furthermore, if $e$ is an edge with endpoints $v$ and $u$, then $f_v(e) =
g_u(e) = f_u(e)$. Thus, for any perfect deterministic strategy
$\left(\{f_v\},\{g_u\}\right)$, there is a function $\hat{x}:E(G)\arr \Z_2$ such
that $f_v=g_v = \hat{x}\big|_{E(v)}$ for all $v \in V(G)$.  Because $\sum_{e \in
E(v)} \hat{x}(e) = \sum_{e \in E(v)} f_v(e) = b(v)$ for every $v \in V(G)$, $\hat{x}$ is a solution to
the linear system $\mcI(G) x = b$. Conversely, if $\hat{x}$ is a solution to
this linear system, then the strategy $\left(\{f_v\},\{g_u\}\right)$ with $f_v
= g_v = \hat{x}|_{E(v)}$ is a perfect deterministic strategy, so there is a
one-to-one correspondence between perfect deterministic strategies of
$\mcG(G,b)$ and solutions of the linear system $\mcI(G) x=b$.

It is not hard to see that when $G$ is connected, $\mcI(G) x = b$ has a
solution if and only if the restriction of $b$ to each connected component of
$G$ has even parity. If $G$ is disconnected, then $\mcI(G)$ is the direct
sum of the incidence matrices for the connected components of $G$, and
hence $b$ is in the image of $\mcI(G)$ if and only if the restriction of
$b$ to every connected component of $G$ has even parity. We summarize
these facts in the following lemma:
\begin{lemma}\label{lem:graph_space}
    Let $(G,b)$ be a $\Z_2$-coloured graph, and let $G_1,\ldots,G_k$ be the
    connected components of $G$. The linear system $\mcI(G) x = b$ has a
    solution (or equivalently, $\mcG(G,b)$ has a perfect deterministic strategy) if
    and only if the restriction $b|_{V(G_i)}$ has even parity for all $1 \leq i \leq k$.
\end{lemma}
A proof of Lemma \ref{lem:graph_space} can be found in \cite{Ark12}.

Deterministic strategies belong to a larger class of strategies called
\emph{classical strategies}, which allow the players to use local and shared
randomness. Every classical strategy is a convex combination of deterministic
strategies, and hence a nonlocal game has a perfect classical strategy if and
only if it has a perfect deterministic strategy. Since we don't need the notion
of a classical strategy, we omit the definition here.

\subsection{Perfect quantum strategies and the graph incidence group}

To define quantum strategies for graph incidence games, we require the
following facts from quantum probability. A quantum state $\nu$ is a unit
vector in a complex Hilbert space $\mcH$. A $\{\pm 1\}$-valued observable
$\msO$ is a self-adjoint unitary operator on $\mcH$, and should be thought of
as a $\{\pm 1\}$-valued random variable. We let $\mcO(\mcH)$ denote the space
of $\{\pm 1\}$-valued observables on $\mcH$. Given a quantum state $\nu\in
\mcH$ and an observable $\msO\in \mcO(\mcH)$, the expected value of the
observable is given by the inner-product $\langle \nu| \msO\nu\rangle_{\mcH}$.
Two observables $\msO_A$ and $\msO_B$ are jointly measureable if they commute,
in which case the product $\msO_A \msO_B$ is the observable corresponding to
the product of the values.  Since $\Z_2$ is isomorphic to the multiplicative
group $\{\pm 1\}$, we can think of a $\{\pm 1\}$-valued observable as a
$\Z_2$-valued observable, with value $1$ corresponding to $0 \in \Z_2$, and
value $-1$ corresponding to $1 \in \Z_2$.

As mentioned in the introduction, there is more than one way to model quantum
strategies for nonlocal games. The finite-dimensional model is the most
restrictive, while the commuting-operator model is the most permissive. We
call strategies in the former model ``quantum strategies'', and strategies in the
latter model ``commuting-operator strategies''. Quantum strategies are
a subset of commuting-operator strategies.
\begin{definition}
    A \emph{commuting-operator strategy} (resp.~\emph{quantum strategy})
    for a graph incidence game $\mcG(G,b)$ consists of a Hilbert space $\mcH$
    (resp.~finite-dimensional Hilbert space $\mcH$), a quantum state $\nu \in
    \mcH$, and two subsets
    \begin{equation*}
        \{ \msX_{ve} : v \in V(G), e \in E(v)\} \text{ and }
        \{ \msY_{ve} : v \in V(G), e \in E(v)\}
    \end{equation*}
    of $\mcO(\mcH)$ such that the pairs
    \begin{itemize}[label={}]
        \item $\msX_{ve}, \msX_{vf}$, $v \in V(G)$, $e,f \in E(v)$,
        \item $\msY_{ve}, \msY_{vf}$, $v \in V(G)$, $e,f \in E(v)$, and
        \item $\msX_{ve}, \msY_{uf}$, $v,u \in V(G)$, $e \in E(v)$, $f \in E(u)$
    \end{itemize}
    are all jointly measureable, and such that
    \begin{equation}\label{eq:qoutput}
         \prod_{e\in E(v)}\msX_{ve}
            = \prod_{e \in E(v)} \msY_{ve}
            =(-1)^{b(v)}\Id
        \text{ for all } v \in V(G).
    \end{equation}
\end{definition}
In the graph incidence game, $\msX_{ve}$ (resp. $\msY_{ve}$) is the
observable corresponding to the first (resp. second) player's assignment $a$ to
edge $e$ upon receiving vertex $v$, where we identify the values $\{\pm 1\}$ of
the observables with $\Z_2$ as above. The condition that the players' outputs
satisfy Equation \ref{eq:outputcondition} is rewritten multiplicatively (due to the
identification of $\Z_2$ with $\{\pm 1\}$) in Equation \ref{eq:qoutput}.
In particular, if $(\{f_v\},\{g_u\})$ is a deterministic strategy for
$\mcG(G,b)$, then the observables $\msX_{ve} = (-1)^{f_v(e)}$ and
$\msY_{uf} = (-1)^{g_u(f)}$ on $\mcH = \C^1$ satisfy Equation
\ref{eq:qoutput}, and conversely any one-dimensional quantum strategy must come
from a deterministic strategy in this way.

To win $\mcG(G,b)$, the players' assignments to edge $e$ should be perfectly
correlated, and hence a commuting-operator strategy is \emph{perfect} if the
expectation $\langle \nu| \msX_{ve} \msY_{ue} \nu \rangle = 1$ for all
$v, u \in V(G)$ and $e \in E(u) \cap E(v)$.  Again, perfect one-dimensional
quantum strategies correspond to perfect deterministic strategies, and hence to
solutions of $\mcI(G) x = b$. Perfect commuting-operator strategies in higher dimensions
can be understood using the following group:
\begin{definition}\label{def:group}
    Let $(G,b)$ be a $\Z_2$-coloured graph.
    The \emph{graph incidence group $\Gamma(G,b)$} is the group generated by
    $\{x_e : e\in E(G)\} \cup \{J\}$ subject to the following relations:
    \begin{enumerate}
        \item $J^2=1$ and $x_e^2=1$ for all $ e \in E(G)$ (generators are involutions),
        \item $[x_e,J]=1$ for all $e \in E(G)$ ($J$ is central),
        \item $[x_e,x_{e'}]=1$ for all $v \in V(G)$ and $e,e' \in E(v)$
            (edges incident to a vertex commute), and
        \item $\prod_{e\in E(v)} x_e=J^{b(v)}$ for all $v \in V(G)$ (product of edges around
            a vertex is $J^{b(v)}$).
    \end{enumerate}
\end{definition}
Here $[x,y] \coloneqq xyx^{-1}y^{-1}$ is the commutator of $x$ and $y$.
Note that, after identifying $\Z_2$ with $\{\pm 1\}$, the one-dimensional
representations of $\Gamma(G,b)$ with $J \mapsto -1$ are the solutions of
$\mcI(G) x = b$.  Graph incidence groups are a special case of solution groups
of linear systems (specifically, $\Gamma(G,b)$ is the solution group of the
linear system $\mcI(G) x = b$).  Solution groups were formally introduced in
\cite{CLS17}, although the notion was essentially already present in
\cite{CM14}. Since it is not necessary to refer to the linear system $\mcI(G) x
= b$ when defining $\mcG(G,b)$ and $\Gamma(G,b)$, we prefer the term graph
incidence group in this context.

As mentioned in the introduction, the existence of perfect quantum strategies
for $\mcG(G,b)$ is connected to the graph incidence group in the following way:
\begin{thm}[\cite{CM14,CLS17}]\label{thm:CLS}
    The graph incidence game $\mcG(G,b)$ has a perfect commuting-operator
    strategy if and only if $J \neq 1$ in $\Gamma(G,b)$, and a perfect
    quantum strategy if and only if $J$ is non-trivial in a finite-dimensional
    representation of $\Gamma(G,b)$.
\end{thm}
Note that $J \neq 1$ in $\Gamma(G,b)$ if and only if $J$ is non-trivial in some
representation of $\Gamma(G,b)$. Also, since $J$ is central and $J^2=1$, if $J$
is non-trivial in a representation (resp. finite-dimensional representation) of
$\Gamma(G,b)$, then there is a representation (resp. finite-dimensional
representation) where $J \mapsto -\Id$.

\subsection{Properties of graph incidence groups}

For context with Theorem \ref{thm:CLS}, and for use in the next section, we
explain some basic properties of graph incidence groups.  First, we show that
the isomorphism type of $\Gamma(G,b)$ depends only on $G$ and the parity of
$b$.
\begin{lemma}\label{lem:parity}
    Let $b$ and $b'$ be $\Z_2$-colourings of a connected graph $G$. If $b$
    and $b'$ have the same parity, then there is an isomorphism
    $\Gamma(G,b) \iso \Gamma(G,b')$ sending $J_{G,b}\mapsto J_{G,b'}$.
\end{lemma}
\begin{proof}
    If $b$ and $b'$ have the same parity, then $\tilde{b}=b+b'$ has even
    parity, and the linear system $\mcI(G) x = \tilde{b}$ has a solution
    $\hat{x}$ by Lemma \ref{lem:graph_space}. It follows from Definition
    \ref{def:group} that there is an isomorphism
    $\Gamma(G,b) \arr \Gamma(G,b')$ sending $x_e\mapsto J^{\hat{x}(v)}x_e$ and
    $J_{G,b}\mapsto J_{G,b'}$.
\end{proof}
As stated in Lemma \ref{lem:graph_space}, whether $\mcG(G,b)$ has a perfect
deterministic strategy can be determined by looking at the connected components of $G$.
Whether $J=1$ in $\Gamma(G,b)$ can also be determined by looking at the
connected components of $G$.  Recall that the coproduct of a collection of
group homomorphisms $\psi_i : \Psi \arr \Phi_i$, $i=1,\ldots,k$, is the
quotient of the free product $\Phi_1 * \dotsm * \Phi_k$ by the normal subgroup
generated by the relations $\psi_i(g) = \psi_j(g)$ for all $1 \leq i,j \leq k$
and $g \in \Psi$. We denote this coproduct by
$\prescript{}{\Psi}\coprod_{i=1}^k \Phi_i$. If each $\psi_i$ is injective, so
that $\Psi$ is a subgroup of each $\Phi_i$, then
$\prescript{}{\Psi}\coprod_{i=1}^k \Phi_i$ is called the amalgamated free
product of the groups $\Phi_1,\ldots,\Phi_k$ over $\Psi$, and it is well-known
that the natural homomorphisms $\Phi_i \arr \prescript{}{\Psi}\coprod_{i=1}^k
\Phi_i$ are also injective.

\begin{lemma}\label{lem:disconnected}
    If a $\Z_2$-coloured graph $(G,b)$ has connected components $G_1,\ldots, G_k$,
    and $b_i$ is the restriction of $b$ to $G_i$, then
    \begin{equation*}
        \Gamma(G,b)= \prescript{}{\langle J \rangle} \coprod_{i=1}^k \Gamma(G_i,b_i),
    \end{equation*}
    where the coproduct is over the homomorphisms $\Z_2 = \langle J
    \rangle  \arr \Gamma(G_i,b_i)$ sending $J \mapsto J_{G_i,b_i}$.  In
    particular, $J_{G,b} \neq 1$ in $\Gamma(G,b)$ if and only if $J_{G_i,b_i} \neq 1$
    in $\Gamma(G_i,b_i)$ for all $i=1,\ldots,k$, and if $J_{G,b} \neq 1$ then the
    inclusions $\Gamma(G_i,b_i) \arr \Gamma(G,b)$ are injective.
\end{lemma}
\begin{proof}
    The relations for $\Gamma(G,b)$ can be grouped by component. The variable $J$
    is the only common generator between relations in different components, so the
    lemma follows directly from the presentation.
\end{proof}

The $\Z_2$-colouring $b$ is used in the definition of $\Gamma(G,b)$ to
determine when to include the generator $J$ in a relation, but if we replace
$J$ with the identity, then we get a similar group which only depends on the
uncoloured graph $G$:
\begin{defn}
    Let $G$ be a graph. The \emph{graph incidence group}
    $\Gamma(G)$ is the finitely presented group with generators $\{x_e : e \in
    E(G)\}$, and relations (1)-(4) from Definition \ref{def:group}, with $J$
    replaced by $1$.
\end{defn}
Note that $\Gamma(G) = \Gamma(G,b) / \langle J \rangle$ for any
$\Z_2$-colouring $b$ of $G$. For this reason, $\Gamma(G)$ is finite if and only
if $\Gamma(G,b)$ is finite for some (resp. any) $\Z_2$-colouring $b$. Also,
since $J$ does not appear in any of the relations for $\Gamma(G,0)$, we have
$\Gamma(G,0) = \Gamma(G) \times \Z_2$. By Lemmas \ref{lem:parity} and
\ref{lem:disconnected}, if the restriction of $b$ to every connected component
of $G$ has even parity, then we also have $\Gamma(G,b) \iso \Gamma(G) \times
\Z_2$.

\begin{lemma}\label{lem:disconnected2}
    If the graph $G$ has connected components $G_1,\ldots,G_k$, then
    \begin{equation*}
        \Gamma(G) = \Gamma(G_1) * \Gamma(G_2) * \cdots * \Gamma(G_k).
    \end{equation*}
\end{lemma}
Similarly to Lemma \ref{lem:disconnected}, Lemma \ref{lem:disconnected2} follows
immediately from the presentation of $\Gamma(G)$.

\subsection{Consequences of main results for quantum strategies}\label{SS:consequences}
To explain the implications of Theorems \ref{thm:finite} and \ref{thm:abelian}
for perfect strategies of $\mcG(G,b)$, we summarize the following points from
the proof of Theorem \ref{thm:CLS}.  Recall that a tracial state on a group $G$
is a function $\tau : G \arr \C$ such that $\tau(1) = 1$, $\tau(ab) = \tau(ba)$
for all $a,b \in G$, and $\tau$ is positive (meaning that $\tau$ extends to a
positive linear functional on the $C^*$-algebra of $G$).
The opposite group $\Phi^{op}$ of a group $(\Phi,\cdot)$ is the set $\Phi$ with
a new group operation $\circ$ defined by $a \circ b := b \cdot a$.

\begin{proposition}[\cite{CM14,CLS17}]\label{prop:CLS2}
    Let $(G,b)$ be a $\Z_2$-coloured graph.
    \begin{enumerate}[(a)]
        \item Suppose $\{\msX_{ve}\}, \{\msY_{uf}\}, \nu$ is a
            perfect strategy for $\mcG(G,b)$ on a Hilbert space $\mcH$.
            Let $\mcA$ and $\mcB$ be the subalgebras of $B(\mcH)$ generated
            by the observables $\{\msX_{ve}\}$ and $\{\msY_{ve}\}$, respectively.
            Then:
            \begin{itemize}
                \item If $\mcH_0 := \overline{\mcA \nu} \subset \mcH$, then
                    $\mcH_0$ is also equal to $\overline{\mcB \nu}$.
                \item If $e \in E(G)$ has endpoints $u,v$, then the observable
                    $X_e := \msX_{ve}|_{\mcH_0}$ on $\mcH_0$ is also equal to
                    $\msX_{ue}|_{\mcH_0}$, and the mapping
                    \begin{equation*}
                        x_e \mapsto X_e \text{ for } e \in E(G) \text{ and } J \mapsto -\Id
                    \end{equation*}
                    defines a representation $\phi$ of $\mcG(G,b)$ on $\mcH_0$.
                \item There is also a unique representation $\phi^R$ of the opposite
                    group $\Gamma(G,b)^{op}$ on $\mcH_0$ defined by
                    $\phi^R(z) w \nu = w \phi(z) \nu$ for all $z \in \Gamma(G,b)^{op}$
                    and $w \in \mcA$. Furthermore, $\msY_{ve} = \phi^R(x_e)$ for all
                    $v \in V(G)$ and $e \in E(v)$.
                \item The function
                    \begin{equation*}
                        \tau : \Gamma(G,b) \arr \C :  z \mapsto \langle \nu | \phi(z) \nu \rangle
                    \end{equation*}
                    defines a tracial state on $\Gamma(G,b)$.
            \end{itemize}
        \item Suppose $\tau$ is a tracial state on $\Gamma(G,b)$ with $\tau(J)
            = -1$. Then there is a Hilbert space $\mcH$, a unitary representation
            $\phi$ of $\Gamma(G,b)$ on $\mcH$, a unitary representation $\phi^R$ of the opposite group $\Gamma(G,b)^{op}$
            on $\mcH$, and a quantum state $\nu \in \mcH$ such that
            if $\msX_{ve} := \phi(x_e)$ and $\msY_{ve} := \phi^R(x_e)$, then
            $\{\msX_{ve}\}$, $\{\msY_{ve}\}$, $\nu$ is a perfect commuting-operator
            strategy for $\mcG(G,b)$, and furthermore $\tau(z) = \langle \nu |
            \phi(z) \nu \rangle$ for all $z \in \Gamma(G,b)$ (so $\tau$ is the tracial state of
            the strategy as in part (a) above).
        \item Let $\psi : \Gamma(G,b) \arr U(\mcH)$ be the left multiplication action of
            $\Gamma(G,b)$ on $\mcH = \ell^2 \Gamma(G,b)$, and let $\nu =
            \frac{1 - J}{\sqrt{2}} \in \mcH$. If $J \neq 1$ in $\Gamma(G,b)$, then
            the function
            \begin{equation*}
                \tau : \Gamma(G,b) \arr \C : z \mapsto \langle \nu | \psi(z) \nu \rangle
            \end{equation*}
            is a tracial state on $\Gamma(G,b)$ with $\tau(J) = -1$.

        \item Suppose $\psi$ is a finite-dimensional representation of $\Gamma(G,b)$ on
            $\C^d$ with $\psi(J)=-\Id$. Let $\tau$ be the tracial state
            $\tau(z) = \tr(\phi(x)) / d$ on $\Gamma(G,b)$.
            Then in part (b) we can take $\mcH = \C^d \otimes \C^d$, $\phi =
            \psi \otimes \Id$, $\phi^R = \Id \otimes \psi^T$
            (where $\psi^T$ refers to the transpose with respect to the standard basis
            on $\C^d$), and $\nu = \frac{1}{\sqrt{d}} \sum_{i=1}^d e_i \otimes e_i$, where
            $e_i$ is the $i$th standard basis element of $\C^d$.
    \end{enumerate}
\end{proposition}
\begin{proof}
    Part (a) is \cite[Lemma 8]{CLS17}. Part (d) is proved in \cite{CM14} for
    the more general class of binary constraint system games. Part (b) is slightly more
    general than what is shown in \cite{CLS17}, but if we take $\phi$ and
    $\phi^R$ to be the left and right GNS representations of $\tau$, and $\nu$
    to be the cyclic state for the GNS representation, then the rest of the proof is
    the same as for \cite[Theorem 4]{CLS17}. Part (c) is implicitly used in
    \cite{CLS17}, and follows immediately from the fact that $J \neq 1$ (so that
    $\nu$ is a unit vector) and $J$ is central.
\end{proof}
Every group is isomorphic to its opposite group via the map $\Phi \arr
\Phi^{op} : z \mapsto z^{-1}$. Thus we can also think of the representation
$\phi^R$ of $\Gamma(G,b)^{op}$ appearing in Proposition \ref{prop:CLS2} as a
representation of $\Gamma(G,b)$. We stick with $\Gamma(G,b)^{op}$ to better
distinguish the two representations.

Recall that perfect deterministic strategies correspond to solutions of
$\mcI(G) x = b$, and hence to one-dimensional representations of $\Gamma(G,b)$
with $J=-1$.  Proposition \ref{prop:CLS2} generalizes this by showing that
there is a correspondence between perfect commuting-operator representations
of $\Gamma(G,b)$ and tracial states on $\Gamma(G,b)$.  Thus, understanding the
structure of $\Gamma(G,b)$ allows us to understand the structure of perfect
strategies for $\mcG(G,b)$. For instance, part (d) of Proposition \ref{prop:CLS2}
implies that every finite-dimensional irreducible representation $\psi$ of
$\Gamma(G,b)$ with $\psi(J)=-\Id$ can be turned into a perfect quantum strategy for $\mcG(G,b)$.
If $\Gamma(G,b)$ is finite, then we can prove conversely that all perfect
strategies are direct sums of strategies of this form:
\begin{cor}\label{cor:finitestrats}
    Suppose $\Gamma(G,b)$ is a finite group, and let $\phi_i : \Gamma(G,b) \arr
    U(\mcV_i)$, $i=1,\ldots,k$ be the irreducible unitary representations of
    $\Gamma(G,b)$ with $\phi_i(J) = -\Id$. Choose an orthonormal basis
    $v_{i1},\ldots,v_{id_i}$ of $\mcV_i$, where $d_i = \dim \mcV_i$. Given $w \in
    \Gamma(G,b)$, let $\phi_i(w)^T$ denote the transpose of $\phi_i(w)$ with
    respect to the chosen basis, and let $\phi_i^T$ denote the corresponding
    representation of $\Gamma(G,b)^{op}$ on Hilbert space $\mcV_i$.

    If $\{\msX_{ve}\}$, $\{\msY_{ve}\}$, $\nu$ is a perfect commuting-operator
    strategy on a Hilbert space $\mcH$, then there is a finite-dimensional
    subspace $\mcH_0 \subseteq \mcH$ which contains $\nu$ and is invariant under
    $\msX_{ve}$ and $\msY_{ve}$ for all $v \in V(G)$, $e \in E(G)$, and an
    isometric isomorphism
    \begin{equation*}
        I : \mcH_0 \arr (\mcV_{i_1} \otimes \mcV_{i_1}) \oplus \cdots \oplus (\mcV_{i_m} \otimes \mcV_{i_m}),
    \end{equation*}
    where $1 \leq i_1 < \ldots < i_m \leq k$, such that
    \begin{equation*}
        I(\nu) = \sum_{j=1}^m \frac{\lambda_j}{\sqrt{d_{i_j}}} \sum_{\ell=1}^{d_{i_j}} v_{i_j \ell} \otimes v_{i_j \ell}
    \end{equation*}
    for some positive real numbers $\lambda_j$ with $\sum_j \lambda^2_j = 1$, and
    \begin{equation*}
        I \msX_{ve} I^{-1} = \sum_{j=1}^m \phi_{i_j}(x_e) \otimes \Id_{\mcV_i}, \quad\quad
        I \msY_{ve} I^{-1} = \sum_{j=1}^m \Id_{\mcV_i} \otimes \phi_{i_j}^T(x_e)
    \end{equation*}
    for all $v \in V(G)$ and $e \in E(v)$.
\end{cor}
\begin{proof}
    Let $\mcH_0$ be the subspace from part (a) of Proposition \ref{prop:CLS2}, and let
    $\phi$ and $\phi^R$ be the corresponding representations of $\Gamma(G,b)$ and
    $\Gamma(G,b)^{op}$ on $\mcH_0$. Since $\Gamma(G,b)$ is finite,
    \begin{equation*}
        \mcH_0 = \overline{\mcA \nu} = \overline{\phi(\C \Gamma(G,b)) \nu}
            = \phi(\C \Gamma(G,b)) \nu
    \end{equation*}
    is finite-dimensional. Since $\phi(J)=-\Id$, there is an isometric isomorphism
    \begin{equation*}
        I_0 : H_0 \arr \bigoplus_{i=1}^k \mcV_{i} \otimes \mcW_i
    \end{equation*}
    for some (possibly trivial) Hilbert spaces $\mcW_1,\ldots,\mcW_k$ with $I_0 \phi
    I_0^{-1} = \sum_{i=1}^k \phi_i \otimes \Id$. Let $I_0 \nu = \sum_{i=1}^k \nu_i$
    where $\nu_i \in \mcV_i \otimes \mcW_i$. If $P_i \in \C \Gamma(G,b)$ is the
    central projection for $\mcV_i$, then $I_0 \phi(P_i) I_0^{-1} = \phi_i(P_i) \otimes
    \Id = \Id_{\mcV_i} \otimes \Id_{\mcW_i}$ is the projection onto $\mcV_i \otimes \mcW_i$, and so $\nu_i = I_0 \phi(P_i) \nu$. Hence
    \begin{equation*}
        \langle \nu_i | (\phi_i(z) \otimes \Id) \nu_i \rangle
            = \langle \nu | \phi(z P_i) \nu \rangle
    \end{equation*}
    for all $z \in \Gamma(G,b)$,
    and since $z \mapsto \langle \nu | \phi(z) \nu \rangle$ is tracial,
    \begin{equation*}
        \tau_i : \Gamma(G,b) \arr \C : z \mapsto \langle \nu_i | (\phi_i(z) \otimes \Id) \nu_i
            \rangle
    \end{equation*}
    is a class function on $\Gamma(G,b)$, meaning that $\tau_i(zw) = \tau_i(wz)$ for all $z,w \in \Gamma(G,b)$.

    On the other hand, if we take the Schmidt decomposition
    \begin{equation*}
        \nu_i = \sum_{j=1}^{m_i} c_{ij} u_{ij} \otimes w_{ij}
    \end{equation*}
    for some integer $m_i \geq 0$, positive real numbers $c_{i1},\ldots,c_{im}$ and orthonormal subsets
    $u_{i1},\ldots,u_{im_i}$ and $w_{i1},\ldots,w_{im_i}$ of $\mcV_i$ and $\mcW_i$ respectively, then
    \begin{equation*}
        \langle \nu_i | (\phi_i(z) \otimes \Id) \nu_i \rangle
        = \sum_{j=1}^{m_i} c_{ij} \langle u_{ij} | \phi_i(z) u_{ij} \rangle.
    \end{equation*}
    Since $\mcV_i$ is irreducible, $\phi_i(\C \Gamma(G,b)) = \Lin(\mcV_i)$, the
    space of linear transformations from $\mcV_i$ to itself. Thus
    the only way for $\tau_i$ to be tracial is if $\nu_i=0$ (in which case $m_i=0$),
    or if $m_i=\dim \mcV_i=d_i$ and $c_{i1}=\ldots=c_{id_i}$. Since $\phi(\C \Gamma(G,b)) \nu
    = \mcH_0$, we must also have
    \begin{equation*}
        (\phi_i(\C \Gamma(G,b)) \otimes \Id) \nu_i =  I_0 \phi(\C \Gamma(G,b) P_i) \nu =
            I_0 \phi(P_i) \mcH_0 = \mcV_i \otimes \mcW_i.
    \end{equation*}
    and this is only possible if $m_i = \dim \mcW_i$ as well.

    Suppose that $\nu_i \neq 0$, and let
    \begin{equation*}
        \gamma_i = \frac{1}{\sqrt{d_i}} \sum_{j=1}^{d_i} v_{ij} \otimes v_{ij} \in \mcV_{i} \otimes \mcV_{i}.
    \end{equation*}
    Recall that if $A \in \Lin(\mcV_i)$, then $(A \otimes \Id) \gamma_i = (\Id
    \otimes A^T) \gamma_i$, where the transpose is taken with respect to the
    basis $v_{i1},\ldots,v_{id_i}$. Let $U_{i1} : \mcV_i \arr \mcV_i$ be the
    unitary transformation sending $v_{ij} \mapsto u_{ij}$, and let $U_{i2} : \mcV_i
    \arr \mcW_i$ be the unitary transformation sending $v_{ij} \mapsto w_{ij}$. Then
    \begin{equation*}
        \nu_i = \lambda_i (U_{i1} \otimes U_{i2}) \gamma_i = \lambda_i (\Id \otimes (U_{i2} U_{i1}^T)) \gamma_i,
    \end{equation*}
    where $\lambda_i := \sqrt{d_i} c_1$. Let
    $1 \leq i_1 < \ldots < i_m \leq k$ be the indices $i$ such that $\mcW_i \neq
    0$, and let
    \begin{equation*}
        I_1 = \sum_{j=1}^m \Id \otimes U_{i_j 2} U_{i_{j} 1}^T.
    \end{equation*}
    Since $U_{i1}^T$ is unitary for all $i$, $I_1$ is an isometric isomorphism. Hence
    \begin{equation*}
        I = I_1 I_0 : \mcH_0 \arr \bigoplus_{j=1}^m \mcV_{i_j} \otimes \mcV_{i_j}
    \end{equation*}
    is an isometric isomorphism with
    \begin{equation*}
        I \nu = \sum_{j=1}^m \lambda_{i_j} \gamma_{i_j} \text{ and } I \phi I^{-1} = \sum_{j=1}^m \phi_{i_j} \otimes \Id.
    \end{equation*}
    Since $\nu$ is a unit vector, $\sum_{j=1}^m \lambda_{i_j}^2 = 1$.
    Finally, by Proposition \ref{prop:CLS2}, part (a),
    \begin{align*}
        I \phi^R(z) I^{-1} (\phi_{i_j}(w) \otimes \Id) \lambda_{i_j} \gamma_{i_j} & = I \phi^R(z) \phi(w P_{i_j}) \nu
            = I \phi(w P_{i_j}) \phi(z) \nu \\
            & = I \phi(w z) I^{-1} \lambda_{i_j} \gamma_{i_j}
             = (\phi_{i_j}(w z)\otimes \Id) \lambda_{i_j} \gamma_{i_j} \\
            & = (\phi_{i_j}(w) \otimes \phi_{i_j}(z)^T) \lambda_{i_j} \gamma_{i_j}
    \end{align*}
    for all $z \in \Gamma(G,b)^{op}$ and $w \in \C \Gamma(G,b)$.
    Thus we have
    \begin{equation*}
        I \phi^R I^{-1} = \sum_{j=1}^m \Id \otimes \phi_{i_j}^T
    \end{equation*}
    as required.
\end{proof}
\begin{rmk}
    In Corollary \ref{cor:finitestrats}, we end up with a quantum strategy on a
    direct sum $\bigoplus_i \mcH_i \otimes \mcH_i$ of tensor products of
    Hilbert spaces, with the first players' observables $\msX_{ve}$ acting on
    the first tensor factor, and the second players' observables $\msY_{ve}$
    acting on the second tensor factor. It is well-known that every quantum
    strategy (but not every commuting-operator strategy) for any nonlocal game
    can be put in this form, and this tensor product decomposition is often
    used explicitly in the definition of quantum strategies (see, e.g.
    \cite{SW08}). To avoid stating two versions of all the results in this
    section, we've used a streamlined definition of quantum strategies that
    does not include an explicit tensor product decomposition. If we are given
    a strategy $\{\msX_{ve}\}$, $\{\msY_{ve}\}$, $\nu$ on a Hilbert space
    $\mcH$ with an explicit tensor product decomposition, so $\mcH = \mcH^A
    \otimes \mcH^B$
    and $\msX_{ve} = \tilde{\msX}_{ve} \otimes \Id$, $\msY_{ve} = \Id \otimes
    \tilde{\msY}_{ve}$ for all $v \in V(G)$, $e \in E(G)$, then it is possible
    to prove a slightly stronger version of Corollary \ref{cor:finitestrats} in
    which the subspace $\mcH_0$ and the isometry $I$ are local, meaning that
    $\mcH_0 = \bigoplus_i \mcH^A_i \otimes \mcH^B_i$ for some subspaces
    $\mcH^A_i$ and $\mcH^B_i$ of $\mcH^A$ and $\mcH^B$ respectively, and
    $I = \sum_{i} I_{i}^A \otimes I_{i}^B$ for isometries $I_i^A$ and $I_i^B$
    acting on $\mcH^A_i$ and $\mcH^B_i$ respectively. Other variants---for
    instance, in which $\mcH_0 = \mcH$ and $I$ is an isometry but not
    necessarily an isomorphism---are also possible in this setting. For brevity, we leave these
    variants for the interested reader.
\end{rmk}

\begin{example}\label{ex:K33}
    Let $b$ be an odd parity colouring of $G$, where $G = K_{3,3}$ or $K_{5}$.
    Then $\mcG(G,b)$ has no perfect deterministic strategy by Lemma
    \ref{lem:graph_space}. Since $G$ is nonplanar, Theorem \ref{thm:abelian}
    implies that $\Gamma(G,b)$ is nonabelian. In addition, $G$ does not contain two
    disjoint cycles or $K_{3,6}$ as a minor so by Theorem \ref{thm:finite}
    $\Gamma(G,b)$ is finite. Using the mathematical software system SageMath \cite{SageMath}, we
    compute the character table for $\Gamma(G,b)$ and find that both
    $\Gamma(K_{3,3},b)$ and $\Gamma(K_5,b)$ have a unique four-dimensional
    irreducible unitary representation with $J\mapsto -\Id$.

    For each corresponding nonlocal game, a perfect quantum strategy can be
    obtained from the irreducible representation by the method outlined in part
    (d) of Proposition \ref{prop:CLS2}. Since $\Gamma(G,b)$ is finite, the
    perfect quantum strategy for $\mcG(G,b)$ is unique by
    Corollary~\ref{cor:finitestrats}, and hence these are robust self-tests by
    \cite{CS17a}. This is well-known: as mentioned in the introduction,
    $\mcG(K_{3,3},b)$ and $\mcG(K_5,b)$ are the magic square and magic
    pentagram nonlocal games \cite{Wu16,Kal17,CS17a}.
\end{example}

\begin{example}\label{ex:K34odd}
    Let $b$ be an odd parity colouring of $K_{3,4}$, so that $\mcG(K_{3,4},b)$
    does not have a perfect deterministic strategy by Lemma
    \ref{lem:graph_space}. Since $K_{3,4}$ is nonplanar, Arkhipov's theorem
    implies that this game has a perfect quantum strategy. $K_{3,4}$ does not
    contain $K_{3,6}$ or two disjoint cycles as a minor, so $\Gamma(K_{3,4},b)$
    is finite by Theorem \ref{thm:finite}. Since $K_{3,4}$ contains $K_{3,3}$
    as a minor and $b$ is odd, Theorem \ref{thm:abelian} implies that
    $\Gamma(K_{3,4},b)$ is nonabelian.

    By direct computation, we
    find that $\Gamma(K_{3,4},b)$ has order $512$. By computing the character
    table, we see that $\Gamma(K_{3,4},b)$ has 16 irreducible representations
    in which $J\mapsto -\Id$, all of dimension $4$. By part (d) of Proposition \ref{prop:CLS2},
    there is a perfect quantum strategy corresponding to each of these irreducible
    representations. Corollary \ref{cor:finitestrats} shows that all perfect
    strategies are direct sums of these $16$ strategies, with some weights $\lambda_j$.
    Notably, this gives an example of a self-testing result in which there is
    no single ideal strategy. Examples of self-tests with similar behaviour are
    given in \cite{CMMN19,Kan19}.
\end{example}

If $b$ is an odd $\Z_2$-colouring of a connected nonplanar graph $G$, then
Theorem \ref{thm:abelian} implies (and it is also not hard to see from
Arkhipov's proof of Theorem \ref{thm:Arkhipov}) that $\Gamma(G,b)$ is
nonabelian. Thus when $b$ is odd, $\Gamma(G,b)$ must be nonabelian for
$\mcG(G,b)$ to have a perfect quantum strategy (for any graph $G$). On the other hand,
when $b$ is even, $\mcG(G,b)$ has a perfect deterministic strategy.
If $\Gamma(G,b)$ is abelian, then $\Gamma(G,b)$ is in fact finite, so
as a special case of Corollary \ref{cor:finitestrats}, every perfect
strategy for $\mcG(G,b)$ is a direct sum of deterministic strategies
on the support of the state. As mentioned in the introduction, this
property characterizes when $\Gamma(G,b)$ is abelian:
\begin{cor}\label{cor:abelian}
    Let $b$ be a $\Z_2$-colouring of a graph $G$, such that the restriction of
    $b$ to each connected component of $G$ has even parity. Let
    $(\{f_v^{(i)}\}, \{g_v^{(i)}\})$, $i=1,\ldots,k$ be a complete list of the
    deterministic strategies for $\mcG(G,b)$. Then $\Gamma(G,b)$ is abelian if
    and only if for every perfect commuting-operator strategy
    $\mcH$, $\nu$, $\{\msX_{ve}\}$, $\{\msY_{ve}\}$, there
    is a subspace $\mcH_0$ of $\mcH$ which contains $\nu$ and is
    invariant under $\msX_{ve}$ and $\msY_{ve}$ for all $v \in V(G)$,
    $e \in E(G)$, a sequence $1 \leq i_1 < \ldots < i_m \leq k$ and
    orthonormal basis $\nu_1,\ldots,\nu_m$ for $\mcH_0$ such that
    $\nu = \sum_{j=1}^m \lambda_j \nu_j$ for some positive real numbers
    $\lambda_j$, and
    \begin{equation*}
        \msX_{ve} |_{\mcH_0} = \sum_{j=1}^m (-1)^{f_v^{(i_j)}(e)} \nu_i \nu_i^*, \quad
        \msY_{ve} |_{\mcH_0} = \sum_{j=1}^m (-1)^{g_v^{(i_j)}(e)} \nu_i \nu_i^*, \quad
    \end{equation*}
    for all $v \in V(G)$, $e \in E(G)$.
\end{cor}
\begin{proof}
    When $\Gamma(G,b)$ is abelian, the irreducible representations of
    $\Gamma(G,b) = \Gamma(G,b)^{op}$ are exactly the one-dimensional
    representations. So the corollary follows immediately from Corollary
    \ref{cor:finitestrats} and the fact that one-dimensional representations of
    $\Gamma(G,b)^{op} = \Gamma(G,b)$ with $J=-1$ are the same as solutions of
    $\mcI(G) x = b$.

    Conversely, if $\mcH$, $\nu$, $\{\msX_{ve}\}$, $\{\msY_{ve}\}$ is
    a perfect commuting-operator strategy which is a direct sum of
    deterministic strategies as in the statement of the corollary,
    then the observables $\msX_{ve}|_{\mcH_0}$ and $\msX_{v'e'}|_{\mcH_0}$
    commute for all $v,v' \in V(G)$, $e,e' \in E(G)$. As a result, if
    $\tau$ is the tracial state of this strategy from part (a) of Proposition
    \ref{prop:CLS2}, then $\tau([x,y])=1$ for all $x,y \in \Gamma(G,b)$.

    Suppose $\Gamma(G,b)$ is nonabelian. Since the restriction of
    $b$ to each connected component of $G$ is even parity, Lemmas
    \ref{lem:parity} and \ref{lem:disconnected} imply that $\Gamma(G,b)
    \iso \Gamma(G) \times \Z_2$, where $J = J_{G,b}$ is mapped to the
    generator of the $\Z_2$ factor. Since $\Gamma(G,b)$ is
    nonabelian, $\Gamma(G)$ is nonabelian, and in this way we can
    find $x,y \in \Gamma(G,b)$ such that $[x,y] \not\in \{1,J\}$.
    Let $\psi$ be the left action of $\Gamma(G,b)$ on $\mcH = \ell^2
    \Gamma(G,b)$, let $\nu = \frac{1 - J}{\sqrt{2}} \in \mcH$, and let $\tau(g)
    = \langle \nu | g \nu \rangle$ be the tracial state on $\Gamma(G,b)$ from
    part (c) of Proposition \ref{prop:CLS2}. Using the definition of $\tau$,
    we see that $\tau(g) = 0$ if $g \not\in \{1, J\}$. In particular,
    $\tau([x,y]) = 0$. But by part (b) of Proposition \ref{prop:CLS2},
    there is a perfect commuting-operator strategy with tracial state
    $\tau$. By the paragraph above, it is not possible for every
    perfect commuting-operator strategy of $\mcG(G,b)$ to be a direct sum of
    deterministic strategies as in the statement of the corollary.
\end{proof}

\begin{example}
    Let $b$ be an even parity colouring of $K_{3,3}$. Since $K_{3,3}$
    does not contain $C_2 \sqcup C_2$ or $K_{3,4}$, Theorem \ref{thm:abelian}
    implies that $\Gamma(K_{3,3},b)$ is abelian. By direct computation
    with SageMath, we find that $\Gamma(K_{3,3}, b)$ has order $32$, and hence
    has $32$ distinct irreducible representations (all one-dimensional). Of these representations,
    $16$ send $J \mapsto -1$, so $\mcG(K_{3,3},b)$ has $16$ distinct deterministic
    strategies.  By Corollary \ref{cor:abelian}, every perfect strategy for
    $\mcG(K_{3,3}, b)$ is a weighted direct sum of these $16$ deterministic strategies.
\end{example}

\begin{example}
    Let $b$ be an even parity colouring of $K_{3,4}$.
    As in Example \ref{ex:K34odd}, $\Gamma(K_{3,4},b)$
    is finite, and by Theorem \ref{thm:abelian}, $\Gamma(K_{3,4},b)$ is not abelian.
    By Lemma \ref{lem:parity}, $\Gamma(K_{3,4}, b) \iso \Gamma(K_{3,4}) \times \Z_2$.
    By direct computation with SageMath, we see that $\Gamma(K_{3,4})$ has order $256$,
    $64$ irreducible representations of dimension $1$, and $12$ irreducible representations
    of dimension $4$. It follows that $\Gamma(K_{3,4},b)$ (which has order $512$) has
    $64$ irreducible representations of dimension $1$ and $12$ irreducible representations
    of dimension $4$ sending $J \mapsto -\Id$. By Corollary
    \ref{cor:finitestrats}, every perfect strategy of $\Gamma(K_{3,4},b)$ is a
    direct sum of these irreducible representations. While the
    one-dimensional representations give deterministic perfect strategies of
    $\mcG(K_{3,4},b)$, the four-dimensional representations yield perfect
    strategies which are not direct sums of deterministic strategies. Interestingly, the correlation
    matrices of these four-dimensional strategies are still classical. 
\end{example}

\section{Graph minor operations for $\Z_2$-coloured graphs}\label{sec:minors}

Although we've motivated the definition of graph incidence groups via the
connection with graph incidence games, graph incidence groups are fairly
natural from the point of view of graph theory. For instance, the linear
relations are a noncommutative generalization of the usual flow problem on the
graph.  The requirement that $x_e$ and $x_{e'}$ commute when $e$ and $e'$ are
incident to a common vertex, which comes from the fact that $x_e$ and $x_{e'}$
correspond to jointly measurable observables, is also natural from the graph
theory point of view, since there is no natural order on the vertices adjacent
to a given vertex. In the rest of the paper, we focus on graph incidence groups
from a combinatorial point of  view, starting in this section with the proof
of Lemma \ref{lem:main}.

Before proving Lemma \ref{lem:main}, we need to introduce a notion of minor
operations for $\Z_2$-coloured graphs. Recall that the standard minor
operations are
\begin{enumerate}[(1)]
    \item \emph{edge deletion}, which takes a graph $G$ and an edge $e$,
        and returns the graph $G\setminus e$ with vertex set $V(G)$ and edge set
        $E(G) \setminus \{e\}$;
    \item \emph{edge contraction}, which takes a graph $G$ and edge $e$,
        and returns the graph $G/e$ with vertex set $V(G) \setminus
        \{v_1,v_2\} \cup \{u\}$, where $v_1,v_2$ are the endpoints of $e$ and
        $u$ is a new vertex, and edge set $E(G) \setminus (E(v_1) \cap E(v_2))$,
        where any edge incident to $v_1$ or $v_2$ is now incident to $u$; and
    \item \emph{vertex deletion}, which takes a graph $G$ and vertex $v$,
        and returns the graph $G\setminus v$ with vertex set $V(G)\setminus \{v\}$ and edge set
        $E(G) \setminus E(v)$.
\end{enumerate}
A graph $H$ is said to be a graph minor of a graph $G$ if $H$ can be constructed
from $G$ by a sequence of graph minor operations.  Notice that for edge
contraction $G / e$, we remove all edges between the endpoints of $e$ to avoid
creating loops.

For $\Z_2$-coloured graphs, we start with the standard graph minor operations,
but add an additional operation, and place a restriction on vertex deletion:
\begin{defn}\label{def:graph_minors}
    The graph minor operations for $\Z_2$-coloured graphs are
    \begin{enumerate}[(1)]
        \item \emph{edge deletion}, which takes a $\Z_2$-coloured graph $(G,b)$
            and an edge $e$, and returns the graph $(G \setminus e, b)$;
        \item \emph{edge contraction}, which takes a $\Z_2$-coloured graph $(G,b)$
            and an edge $e$, and returns the graph $(G/e,b')$ with
            \begin{equation*}
                b'(v) = \begin{cases} b(v_1) + b(v_2) & v = u \\
                                      b(v) & v \neq u
                        \end{cases},
            \end{equation*}
            where as above $v_1$ and $v_2$ are the endpoints of $e$, and $u$ is the
            new vertex;
        \item \emph{vertex deletion}, which takes a $\Z_2$-coloured graph $(G,b)$
            and a vertex $v$ with $b(v)=0$, and returns the graph $\left(G\setminus v,
            b|_{V(G)\setminus\{v\}}\right)$; and
        \item \emph{edge toggling}, which takes a $\Z_2$-coloured graph $(G,b)$ and
            an edge $e$, and returns the graph $(G,b')$ with
            \begin{equation*}
                b'(v) = \begin{cases}
                            b(v) & e \not\in E(v) \\
                            b(v)+1 & e \in E(v)
                        \end{cases}.
            \end{equation*}
    \end{enumerate}
    A graph $(H,c)$ is a \emph{graph minor} of $(G,b)$ if it can be constructed
    from $(G,b)$ by a sequence of minor operations.
\end{defn}
The minor operations for coloured graphs are shown in
Figures~\ref{fig:edge_deletion}, \ref{fig:contract_edge},
\ref{fig:delete_isolated_vertex}, and \ref{fig:colour_swap}. Note that all of
the operations in Definition \ref{def:graph_minors} preserve the parity of the
colouring. In particular, vertices can only be deleted if they are labelled by
$0$, since otherwise deleting the vertex would change the parity. To delete a
non-isolated vertex $v$ coloured by $1$, we can toggle an edge in $E(v)$ to
change the colour of $v$ to $0$. However, isolated vertices labelled by $1$
cannot be deleted. Although this restriction is necessary for Lemma
\ref{lem:main}, it does cause a problem: if $(G,b)$ is a $\Z_2$-coloured graph
with two connected components $(G_0,b_0)$ and $(G_1,b_1)$, and $(H,c)$ is a
minor of $(G_0,b_0)$, then $(H,c)$ might not be a minor of $(G,b)$ if $b_1$ has
odd parity. To work around this problem, we often restrict to connected graphs
(see, for instance, Lemma \ref{lem:minor_colour}).

\begin{figure}[t]
    \begin{center}
        \includegraphics[scale=1]{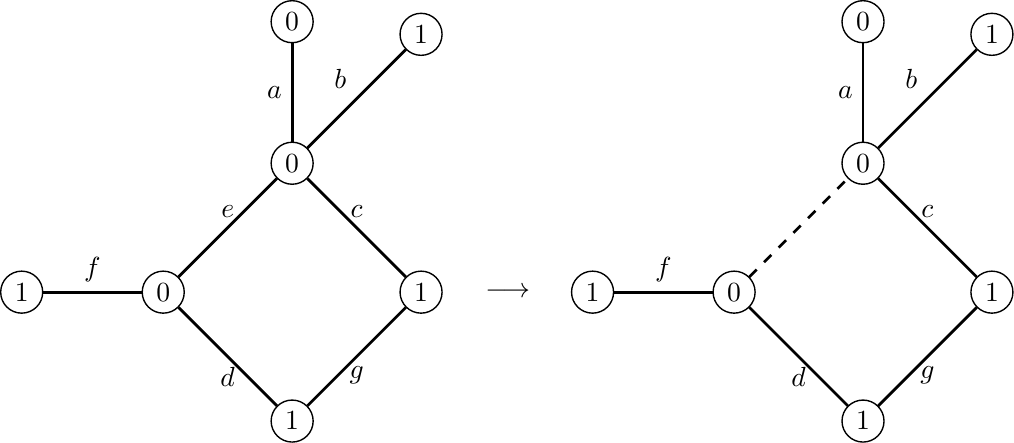}
    \end{center}
        \caption{Deleting an edge $e\in E(G)$ does not change the colour of vertices.}
\label{fig:edge_deletion}
\end{figure}

\begin{figure}[t]
    \begin{center}
        \includegraphics[scale=1]{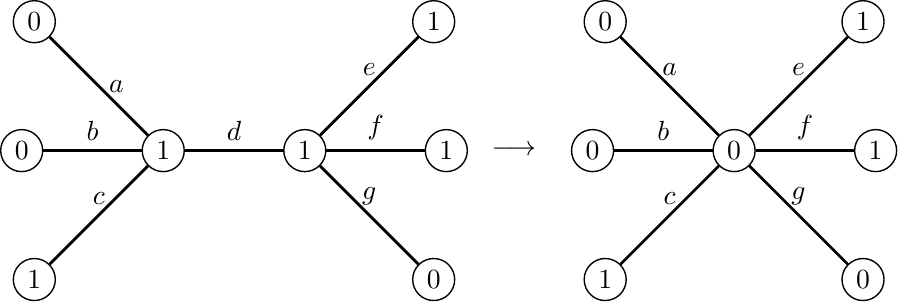}
    \end{center}
        \caption{Contraction of an edge $d \in E(G)$ with endpoints $v_1$ and $v_2$ results in a new vertex $u$ with incident edges $E(u)=\{E(v_1)\cup E(v_2)\}\setminus (E(v_1) \cap E(v_2))$, and colour $b(u)=b(v_1)+b(v_2)$.}
        \label{fig:contract_edge}
\end{figure}

\begin{figure}[t]
    \begin{center}
        \includegraphics[scale=1]{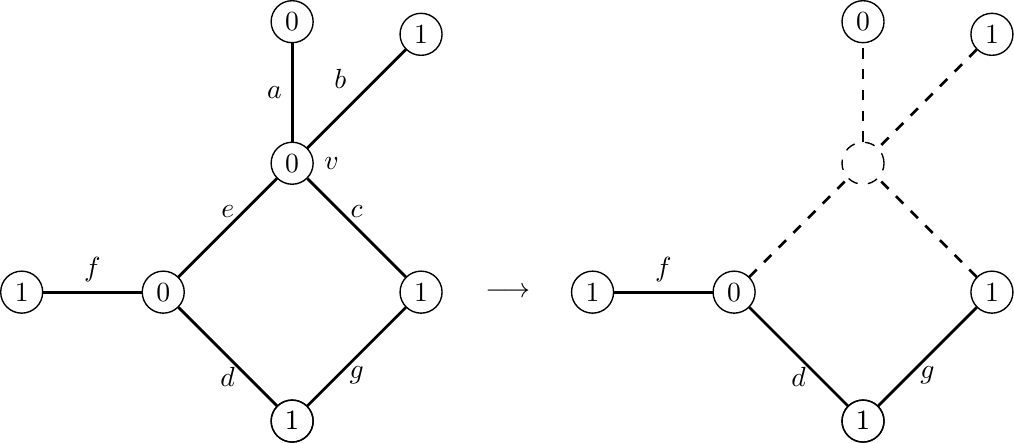}
    \end{center}
        \caption{Deleting a vertex removes the vertex and all incident edges.}
\label{fig:delete_isolated_vertex}
\end{figure}

\begin{figure}[t]
    \begin{center}
        \includegraphics[scale=1]{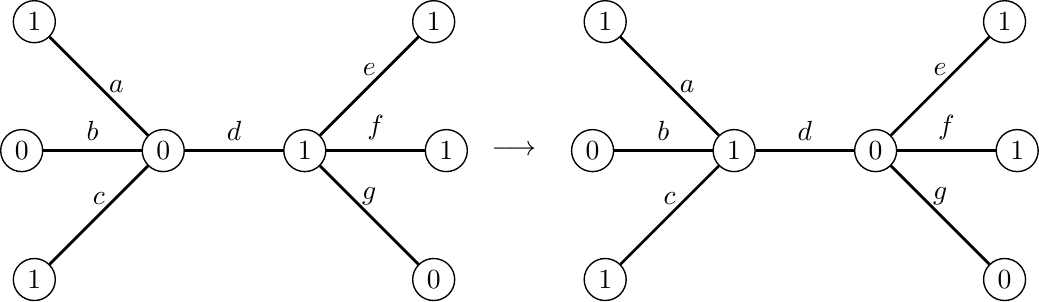}
    \end{center}
        \caption{The edge toggle minor operation on the edge $d$ toggles the $\Z_2$-colouring of the two endpoints.}
\label{fig:colour_swap}
\end{figure}

For any fixed graph $H$, it is possible to decide whether $H$ is a minor of $G$
in time polynomial in the size of $G$ \cite{RS95}.  Including edge toggling as
a minor operation allows us to efficiently determine when a
$\Z_2$-coloured graph $(H,c)$ is a minor of $(G,b)$.
\begin{lemma}\label{lem:minor_colour}
    Let $(G,b)$ be a connected $\Z_2$-coloured graph. Then $(H,c)$ is a minor
    of $(G,b)$ if and only if $H$ is a minor of $G$ and the parity of $c$ is
    equal to the parity of $b$.
\end{lemma}
\begin{proof}
    Clearly if $(H,c)$ is a minor of $(G,b)$ then $H$ is a minor of $G$ and $c$
    has the same parity as $b$. For the converse, we first show that if $b$
    and $b'$ are $\Z_2$-colourings of $G$ with the same parity $p$, then it
    is possible to change $b$ to $b'$ by edge toggling. Indeed, let $v$ be some
    vertex of $G$, and let $b''$ be the $\Z_2$-colouring with $b''(v)=p$
    and $b''(w)=0$ for $w \neq v$. Since $G$ is connected, for every $w \in V(G)$
    there is a path $P$ from $v$ to $w$. Toggling all the edges in $P$ adds the
    colour of $w$ to the colour of $v$, changes the colour of $w$ to $0$, and
    leaves all other colours unchanged, so it is possible to change $b$ to $b''$
    via edge toggling. Similarly, we can change $b'$ to $b''$. Since edge
    toggling is reversible, we can also change $b''$ to $b'$ via edge toggling,
    and hence we can change $b$ to $b'$.

    Note that if $H_1$ is the result of applying a minor operation to some
    graph $H_0$, then every vertex $v$ of $H_1$ comes from a vertex of $H_0$,
    in the sense that $v$ is either a vertex of $H_0$ and is unaffected by the
    minor operation, or $v$ is a new vertex added as a result of identifying
    two vertices of $H_0$ during edge contraction.

    Suppose that $H$ is a minor of $G$ and $c$ has the same parity as $b$, and
    fix a sequence
    \begin{equation*}
        G = G_0 \arr G_1 \arr \cdots \arr G_k = H
    \end{equation*}
    of minor operations sending $G$ to $H$. For each $v \in V(H)$, choose a
    vertex $f(v) \in V(G)$ such that this sequence of minor operations
    eventually sends $f(v)$ to $v$. Let $b'$ be the colouring of $G$ in which
    $b'(w)=c(v)$ if $w = f(v)$ for some $v \in V(H)$, and $b'(w)=0$ otherwise.
    Then $b'$ has the same parity as $c$ and $b$, and hence $(G,b')$ is a minor
    of $(G,b)$. If the minor operation $G_{i} \arr G_i+1$ deletes the vertex
    $w \in V(G_{i})$, and $w' \in V(G)$ maps to $w$ under the first $i$ minor
    operations, then $w'$ is not mapped to any vertex of $H$, so $w' \neq f(v)$
    for all $v \in V(H)$, and consequently $b'(w') = 0$. It follows that we
    can perform the $\Z_2$-coloured versions of each minor operation to
    turn $(G,b')$ into $(H,c)$, and hence $(H,c)$ is a minor of $(G,b)$.
\end{proof}
If $(H,c)$ is a minor of a disconnected graph $(G,b)$, then it must be possible
to decompose $(H,c)$ into a disconnected union of (not necessarily connected)
subgraphs $(H_1,c_1), \ldots, (H_k,c_k)$, such that each subgraph $(H_i,c_i)$
is a minor of a distinct connected component of $(G,b)$, and the remaining
connected components of $(G,b)$ have even parity. Thus we can determine
whether $(H,c)$ is a minor of $(G,b)$ by going through all possible functions
from connected components of $(H,c)$ to connected components of $(G,b)$. This
algorithm is polynomial in the size of $G$, although the exponent depends on
the number of connected components of $(H,c)$.

We now turn to the following proposition, which describes how $\Gamma(G,b)$
changes under minor operations.
\begin{proposition}\label{prop:minors}
    Let $(G,b)$ be a $\Z_2$-coloured graph.
    \begin{enumerate}[(i)]
        \item If $e \in E(G)$ then there is a surjective
            homomorphism $\phi:\Gamma(G,b)\arr\Gamma(G\setminus e,b)$ sending
            \begin{align*}
                J\mapsto J \text{ and } x_f\mapsto \begin{cases} x_f&\text{if}\;
                    f\neq e \\ 1 &\text{if}\; f=e \end{cases} \text{ for all $f\in E(G)$}.
            \end{align*}

        \item If $e \in E(G)$ is the only edge with endpoints $v_1, v_2$,
            then there is a surjective group homomorphism
            $\phi:\Gamma(G,b)\arr \Gamma(G/ e,b')$ sending
            \begin{align*}
                J\mapsto J\text{ and }x_f\mapsto
                        \begin{cases}x_f&\text{if}\; f\neq e \\
                                J^{b(v_1)} \prod\limits_{e'\in E(v_1)\setminus \{e\}} x_{e'} &\text{if}\; f=e
                        \end{cases} \text{ for all $f\in E(G)$},
            \end{align*}
            where $b'$ is the colouring of $G/e$ described in part (2) of
            Definition \ref{def:graph_minors}.

        \item If $v\in V(G)$ is an isolated vertex with $b(v)=0$, then there is an isomorphism
            $\phi:\Gamma(G,b)\arr \Gamma(G\setminus v,b|_{V(G\setminus \{v\})})$ sending
            \begin{align*}
                J\mapsto J\text{ and } x_f\mapsto x_f \text{ for all $f\in E(G)$}\,.
            \end{align*}

        \item If $e \in E(G)$ then there is an isomorphism
            $\phi:\Gamma(G,b)\arr \Gamma(G,b')$ sending
            \begin{align*}
                J\mapsto J\text{ and }x_f\mapsto \begin{cases} x_f&\text{if}\; f\neq e \\ J x_e&\text{if}\;f=e \end{cases} \text{for all $f\in E(G)$}\,,
            \end{align*}
            where $b'$ is the colouring described in part (4) of Definition \ref{def:graph_minors}.
\end{enumerate}
\end{proposition}
\begin{proof}
    In each case, we are given a homomorphism
    \begin{equation*}
        \widetilde{\phi} : \mbF(\{x_f : f \in E(G)\} \cup \{J\}) \arr \mbF(\{x_f : f \in E(G')\} \cup \{J\})
    \end{equation*}
    for some $\Z_2$-coloured graph $(G',b')$, where $\mbF(S)$ denotes the free
    group generated by $S$. For instance, in case (i), $G' = G\setminus e$, $b'
    = b$, and $\widetilde{\phi}$ sends $J \mapsto J$, $x_f \mapsto x_f$ if $f \neq
    e$, and $x_e \mapsto 1$. In each case, we want to show that $\widetilde{\phi}$
    descends to a homomorphism $\phi : \Gamma(G,b) \arr \Gamma(G',b')$. To do this,
    we need to show that if $r$ is a defining relation of $\Gamma(G,b)$ from Definition
    \ref{def:group}, then $\widetilde{\phi}(r)$ is in the normal subgroup generated
    by the defining relations of $\Gamma(G',b')$. In case (i), it is clear that
    if $r$ is a defining relation of $\Gamma(G,b)$ not containing $x_e$, then
    $\widetilde{\phi}(r)$ is also a defining relation of $\Gamma(G',b')$. The relations
    containing $x_e$ are $x_e^2$, $[x_e,x_f]$ for $f \in E(v_1) \cup E(v_2)$, and
    \begin{equation*}
        r_i = J^{b(v)} \prod_{f \in E_G(v_i)} x_f, \quad i=1,2
    \end{equation*}
    where $v_1$,$v_2$ are the endpoints of $e$. For the first two types of relations,
    $\widetilde{\phi}(x_e^2) = \widetilde{\phi}([x_e,x_f]) = 1$, while in the last case,
    $\widetilde{\phi}(r_1)$ and $\widetilde{\phi}(r_2)$ are again defining relations of
    $\Gamma(G,b)$. Hence $\phi$ is well-defined. Since the image of $\phi$
    contains $J$ and $x_f$ for all $f \in E(G')$, in this case $\phi$ is also
    surjective as required.

    For (ii), once again if $r$ is a defining relation of $\Gamma(G,b)$ not
    containing $x_e$, then $\widetilde{\phi}(r)$ is a defining relation of
    $\Gamma(G',b')$. The generators $x_{f}$, $f \in E(v_1) \cup E(v_2) \setminus \{e\}$,
    all commute in $\Gamma(G',b')$, so $\widetilde{\phi}([x_e,x_f]) = 1$ in $\Gamma(G',b')$.
    Also, since all these generators commute,
    \begin{equation*}
        \widetilde{\phi}\left( J^{b(v_1)} \prod_{f \in E(v_1)} x_f\right) =
            \left(J^{b(v_1)} \prod_{f \in E(v_1) \setminus \{e\}} x_f \right)^2 = 1
    \end{equation*}
    in $\Gamma(G',b')$, and $\widetilde{\phi}(x_e)^2=1$ in $\Gamma(G',b')$ for the same
    reason. Finally, if $u$ is the new vertex in $G'$, then
    \begin{align*}
        \widetilde{\phi}\left( J^{b(v_2)} \prod_{f \in E(v_2)} x_f\right) & =
            \widetilde{\phi}(x_e) \left(J^{b(v_2)} \prod_{f \in E(v_2) \setminus \{e\}} x_f \right)\\ & =
            J^{b(v_1)+b(v_2)} \prod_{f \in E(v_1)\cup E(v_2) \setminus \{e\}} x_f \\
            & = J^{b(u)} \prod_{f \in E(u)} x_f = 1
    \end{align*}
    in $\Gamma(G',b')$. So $\phi$ is well-defined, and again the image of $\phi$ contains
    $J$ and $x_f$ for all $f \in E(G')$, so $\phi$ is surjective.

    For (iii), an isolated vertex $v$ with $b(v)=0$ corresponds to a trivial relation in the
    presentation of $\Gamma(G,b)$. Removing the vertex just deletes this relation, so $\phi$
    is clearly an isomorphism.

    For (iv), toggling an edge does not change the parity of the
    $\Z_2$-colouring, and the map $\phi$ is exactly the isomorphism between $\Gamma(G,b)$
    and $\Gamma(G,b')$ defined in Lemma~\ref{lem:parity}.
\end{proof}
The homomorphisms given in part (i) and (ii) of Proposition \ref{prop:minors}
are not necessarily isomorphisms. In part (i), the kernel of $\phi :
\Gamma(G,b) \arr \Gamma(G\setminus e, b)$ is the normal subgroup of $\Gamma(G,b)$
generated by $x_e$. In part (ii), we have that
\begin{equation*}
    x_e = J^{b(v_1)} \prod_{e' \in E(v_1) \setminus\{e\}} x_{e'}
\end{equation*}
already in $\Gamma(G,b)$. However, if $f \in E(v_1)$ and $g \in E(v_2)$, then
$x_f$ and $x_g$ do not necessarily commute in $\Gamma(G,b)$, while $\phi(x_f)$ and
$\phi(x_g)$ commute in $\Gamma(G/e,b')$. Thus the kernel of $\phi : \Gamma(G,b) \arr \Gamma(G/e,b')$
is generated by the commutators $[x_f,x_g]$ for $f \in E(v_1) \setminus \{e\}$, $g \in E(v_2) \setminus \{e\}$. Note as well
that, although the homomorphism $\phi$ in part (ii) seems to depend on a choice of endpoint $v_1$
of $e$,
\begin{equation*}
    J^{b(v_1)} \prod_{e' \in E(v_1) \setminus\{e\}} x_{e'} =
    J^{b(v_2)} \prod_{e' \in E(v_2) \setminus\{e\}} x_{e'}
\end{equation*}
in $\Gamma(G/e,b')$, so $\phi$ is independent of the choice of endpoint.

The proof of Lemma \ref{lem:main} follows quickly from Proposition \ref{prop:minors}.
\begin{proof}[Proof of Lemma~\ref{lem:main}]
    Deleting a vertex $v$ of $(G,b)$ is equivalent to deleting all the edges in $E(v)$,
    and then deleting the now-isolated vertex $v$. Similarly, if $e$ is an edge of $(G,b)$
    with endpoints $v_1$, $v_2$, then contracting $e$ is equivalent to first deleting
    all the edges of $(E(v_1) \cap E(v_2)) \setminus \{e\}$, and then contracting $e$.
    Thus, if $(H,c)$ is a minor of $(G,b)$, then there is a sequence of minor operations
    \begin{equation*}
        (G,b) = (G_0,b_0) \arr (G_1,b_1) \arr \cdots \arr (G_k,b_k) = (H,c)
    \end{equation*}
    sending $(G,b)$ to $(H,c)$, where every operation is one of the operations listed
    in cases (i)-(iv) of Proposition \ref{prop:minors}. For these operations, there is a surjective homomorphism
    $\Gamma(G_{i-1},b_{i-1}) \arr \Gamma(G_i,b_i)$ sending $J_{G_{i-1},b_{i-1}}
    \mapsto J_{G_i,b_i}$ for all $1 \leq i \leq k$. By composing these
    homomorphisms we get a surjective homomorphism $\Gamma(G,b) \arr \Gamma(H,c)$ sending
    $J_{G,b} \mapsto J_{H,c}$.
\end{proof}
As mentioned in the introduction, we can also prove a version of Lemma \ref{lem:main} for
graph incidence groups of uncoloured graphs.
\begin{lemma}\label{lem:main_uncoloured}
    If $H$ is a minor of $G$, then there is a surjective morphism $\Gamma(G) \arr \Gamma(H)$.
\end{lemma}
\begin{proof}
    By Lemma \ref{lem:minor_colour}, $(H,0)$ is a minor of $(G,0)$ (this is true even if
    $G$ is disconnected, since every vertex is coloured $0$), so there is a surjective
    morphism $\Gamma(G,0) \arr \Gamma(H,0)$ sending $J_{G,0} \mapsto J_{H,0}$, and since
    $\Gamma(G) \iso \Gamma(G,0) / \langle J_{G,0} \rangle$ and $\Gamma(H) \iso \Gamma(H,0) / \langle
    J_{H,0}\rangle$, the lemma follows.
\end{proof}

Note that Lemmas \ref{lem:main} and \ref{lem:main_uncoloured} just require the
existence of a homomorphism $\Gamma(G,b) \arr \Gamma(H,c)$ (resp. $\Gamma(G)
\arr \Gamma(H)$), and this is all we will use in the remainder of the paper.
However, Proposition \ref{prop:minors} gives a recipe for finding this
homomorphism from a sequence of minor operations. Let $\Z_2$-$\Graphs$ be the
category of $\Z_2$-coloured graphs with morphisms freely generated by the minor
operations of edge deletion, edge contraction of edges which are the only edges
between their endpoints, vertex deletion of isolated vertices labelled by $0$,
and edge toggling. Then Proposition \ref{prop:minors} implies that there is a
functor $\Gamma$ from $\Z_2$-$\Graphs$ to the category of groups over $\Z_2$ with
surjective homomorphisms. Similarly, if $\Graphs$ is the category of graphs
with morphisms freely generated by edge deletion, edge contraction of edges which
are the only edges between their endpoints, and vertex deletion of isolated
vertices, then there is a functor $\Gamma$ from $\Graphs$ to the category of
groups with surjective homomorphisms. There is also a functor $F$ from $\Z_2$-$\Graphs$
to $\Graphs$ sending $(G,b) \mapsto G$, and another functor $F'$ from groups over
$\Z_2$ with surjective homomorphisms to groups with surjective homomorphisms which
sends $(\Phi,J_{\Phi})$ to $\Phi / \langle J_{\Phi} \rangle$, and these functors
commute with $\Gamma$, i.e.~$\Gamma \circ F = F' \circ \Gamma$.

\subsection{Forbidden minors for quotient closed properties}

A property $\mpP$ of graphs is minor-closed if $G$ has $\mpP$ and $H$ is a
minor of $G$, then $H$ also has $\mpP$. Recall that the Robertson-Seymour
theorem implies that if $\mpP$ is minor-closed, then there is a finite set of
graphs $\mcF$ which do not have $\mpP$, and such that for any graph $G$, $G$
has $\mpP$ if and only if $G$ does not contain any graph from $\mcF$ as a
minor. As an immediate corollary of this theorem and Lemma
\ref{lem:main_uncoloured}, we have:
\begin{cor}\label{cor:main_uncoloured}
    If $\mpP$ is a quotient closed property of groups, then there is a
    finite set $\mcF$ of graphs such that for any graph $G$, $\Gamma(G)$
    has $\mpP$ if and only if $G$ avoids $\mcF$.
\end{cor}
\begin{proof}
    By Lemma \ref{lem:main_uncoloured}, the graph property ``$\Gamma(G)$ has
    $\mpP$'' is minor-closed.
\end{proof}
For technical reasons, it turns out to be difficult to extend Corollary
\ref{cor:main_uncoloured} to all $\Z_2$-coloured graphs.  A quasi-order $\leq$
on a set $X$ is said to be a well-quasi-order if for any infinite sequence
$x_1,x_2,\ldots$ in $X$, there is a pair of indices $1 \leq i < j$ such that
$x_i \leq x_j$. If $\leq$ is a well-quasi-order and $x_1,x_2,\ldots$ is an
infinite sequence, then there must be a sequence $1 \leq i_1 < i_2 < \ldots$ of
indices such that $x_{i_1} \leq x_{i_2} \leq \ldots$. The full version of the
Robertson-Seymour theorem states that the set of graphs is well-quasi-ordered
under the minor containment relation (in which $H \leq G$ if $G$ contains $H$
as a minor).

Say that $(H,c) \leq (G,b)$ if $(H,b)$ is a minor of $(G,b)$. The following
example shows that $\leq$ is not a well-quasi-order on the set of
$\Z_2$-coloured graphs.
\begin{example}\label{ex:nonwqo}
    Let $G_n$ be the graph with $n$ vertices and no edges, and let $b_n :
    V(G_n) \arr \{0,1\}$ be the function sending every vertex to $1$. Then
    no minor operation in Definition \ref{def:graph_minors} can be applied
    to $(G_n,b_n)$, so $(G_n,b_n)$ is not a minor of $(G_k,b_k)$ for any
    $n$ and $k$.
\end{example}
In the presentation of $\Gamma(G,b)$, an isolated vertex $v$ with $b(v)=1$
corresponds to a relation $J=1$. In the presentation of $\Gamma(G_n,b_n)$, this
relation appears $n$ times, and deleting all but one of these relations does not
affect the group. This suggests adding an additional minor operation in which
we can identify these connected components.
\begin{defn}
    The minor operation \emph{component identification} takes a $\Z_2$-coloured
    graph $(G,b)$, and two connected components $G_i$, $i=1,2$, of
    $G$ which are isomorphic as $\Z_2$-coloured graphs, meaning that there is a
    graph isomorphism $\phi : G_2 \arr G_1$ such that $b(\phi(v)) = b(v)$ for
    all $v \in V(G_2)$. It returns the graph with vertex set $V(G) \setminus
    V(G_2)$, edge set $E(G) \setminus E(G_2)$, and vertex colouring $b|_{V(G)
    \setminus V(G_2)}$.

    We say that $(H,c) \preceq (G,b)$ if $(H,c)$ can be constructed from
    $(G,b)$ using the minor operations in Definition \ref{def:graph_minors},
    along with component identification.
\end{defn}
It is not hard to see that if $(H,c)$ is the result of identifying connected
components $G_1$ and $G_2$ of $(G,b)$, and $\phi : G_2 \arr G_1$ is an
isomorphism with $b(\phi(v)) = b(v)$ for all $v \in V(G_2)$, then there is a
surjective homomorphism
\begin{equation*}
    \Gamma(G,b) \mapsto \Gamma(H,c) : J \mapsto J \text{ and }
        x_e \mapsto \begin{cases} x_{\phi(e)} & e \in E(G_2) \\
                                  x_e & e \not\in E(G_2)
                    \end{cases}.
\end{equation*}
Thus if $(H,c) \preceq (G,b)$, then there is a surjective homomorphism
$\Gamma(G,b) \arr \Gamma(H,c)$. Also, if $(G_n,b_n)$ is the graph from Example
\ref{ex:nonwqo}, then $(G_n,b_n) \preceq (G_k,b_k)$ for all $n \leq k$.  Thus,
it seems possible that $\preceq$ is a well-quasi-order.
\begin{prop}\label{prop:wqo}
    The quasi-order $\preceq$ is a well-quasi-order on the set of $\Z_2$-coloured
    graphs if and only if the set of minor-closed properties of connected graphs
    are well-quasi-ordered under inclusion.
\end{prop}
\begin{proof}
    Suppose $(G,b)$ is a $\Z_2$-coloured graph, and let $G_1,\ldots,G_k$ be the
    connected components of $G$. Let $b_i = b|_{V(G_i)}$. For the purposes of
    this proof, we let $G_{even}$ (resp. $G_{odd}$) be the subgraph of $G$
    consisting of the connected components where $b_i$ is even (resp.  $b_i$ is
    odd). Using Lemma \ref{lem:minor_colour}, it is not hard to show that
    $(H,c) \preceq (G,b)$ if and only if $H_{even} \leq G_{even}$, $H_{odd}
    \leq G_{odd}$, and for every connected component $G'$ of $G_{odd}$, there
    is a connected component $H'$ of $H_{odd}$ with $H' \leq G'$.

    Suppose $\mpP$ is a minor-closed property of connected graphs. By the
    Robertson-Seymour theorem, there is a finite set $\mcF$ of connected
    graphs such that $G$ belongs to $\mpP$ if and only if $H \not\leq G$
    for all $H \in \mcF$. Conversely, if $\mcF$ is a finite set of connected
    graphs, then ``$H \not\leq G$ for all $H \in \mcF$'' is a
    minor-closed property of connected graphs. Suppose $\mpP$ and $\mpP'$ are
    two minor-closed properties, defined by finite sets of connected graphs
    $\mcF$ and $\mcF'$ respectively. Then $\mpP$ is contained in $\mpP'$ if
    and only if
    \begin{equation}\label{eq:inclusion}
        \text{ for every } G \in \mcF', \text{ there is }
        H \in \mcF \text{ such that } H \leq G.
    \end{equation}
    Indeed, $\mpP$ is contained in $\mpP'$ if and only if every graph not
    satisfying $\mpP'$ also does not satisfy $\mpP$. So if $\mpP$ is contained
    in $\mpP'$, and $G \in \mcF'$, then $G$ does not satisfy $\mpP'$ and
    hence does not satisfy $\mpP$. But this means that there is $H \in \mcF$
    with $H \leq G$. In the other direction, if $\mcF$ and $\mcF'$ satisfy
    Equation \eqref{eq:inclusion} and $G$ does not satisfy $\mpP'$, then
    there must be $G' \in \mcF'$ such that $G' \leq G$, and hence there is
    $H \in \mcF$ such that $H \leq G' \leq G$, so $G$ does not satisfy $\mpP$.

    Suppose that $\mpP_1, \mpP_2, \ldots$ is a sequence of minor-closed
    properties of connected graphs, and let $\mcF_1,\mcF_2,\ldots$ be the
    corresponding sequence of forbidden minors. Let $G_i$ be the graph with
    connected components $\mcF_i$, and let $b_i$ be a $\Z_2$-colouring
    of $G_i$ such that every connected component has odd parity. If $\preceq$
    is well-quasi-ordered, then there are indices $1 \leq i < j$ such that
    $(G_i,b_i) \preceq (G_j,b_j)$. But then for every connected component $G$ of
    $(G_j)_{odd} = G_j$, there must be a connected component $H$ of $(G_{i})_{odd}
    = G_i$ with $H \leq G$. So $\mcF_i$ and $\mcF_j$ satisfy the condition
    in Equation \eqref{eq:inclusion}, and $\mpP_i$ is contained in $\mpP_j$. We
    conclude that minor-closed properties of connected graphs are well-quasi-ordered
    under inclusion.

    On the other hand, suppose minor-closed properties of connected graphs are
    well-quasi-ordered under inclusion, and let $(G_1,b_1),(G_2,b_2),\ldots$ be
    a sequence of $\Z_2$-coloured graphs. Applying the fact that $\leq$ is
    well-quasi-ordered to the sequence $(G_1)_{even}, (G_2)_{even}, \ldots$,
    we see that there must be a sequence $1 < i_1 < i_2 < \ldots$ of
    indices such that $(G_{i_1})_{even} \leq (G_{i_2})_{even} \leq \ldots$.
    Applying the same reasoning to the sequence $(G_{i_1})_{odd}, (G_{i_2})_{odd},
    \ldots$, we see that there must be a sequence of indices $1 \leq j_1 < j_2 <
    \ldots$ such that $(G_{j_1})_{even} \leq (G_{j_2})_{even} \leq \ldots$
    and $(G_{j_1})_{odd} \leq (G_{j_2})_{odd} \leq \ldots$. Let $\mcF_{k}$
    be the connected components of $(G_{j_k})_{odd}$, and let $\mpP_k$ be
    the corresponding minor-closed property of connected graphs. Then there
    is $k < k'$ such that $\mpP_k$ is contained in $\mpP_{k'}$, so that for every
    $G \in \mcF_{k'}$, there is $H \in \mcF_{k}$ with $H \leq G$. But this means
    that $(G_{j_k},b_{j_k}) \preceq (G_{j_{k'}}, b_{j_{k'}})$. We conclude that
    $\preceq$ is a well-quasi-order.
\end{proof}
We are not aware of the inclusion order on minor-closed properties of connected
graphs being studied in the literature. However, whether or not the inclusion
order on minor-closed properties of all graphs is well-quasi-ordered seems to
be an open problem (see, for instance, \cite{DK05,BNW10}), and we expect the same is
true when looking at properties of connected graphs. Thus we do not know if
$\preceq$ is a well-quasi-order. There are also other graph operations which
are natural with respect to the functor $\Gamma(\cdot)$ (for instance, there is
a generalization of component identification which takes $G$ to $H$ whenever
$G$ is a cover of $H$). We leave it as an open problem to find a natural
category of graph minor operations for $\Z_2$-coloured graphs, such that Lemma
\ref{lem:main} holds, and such that graph minor containment is
well-quasi-ordered.

Fortunately, the above technical problems disappear if we restrict to connected
graphs.
\begin{cor}\label{cor:minor_colour}
    The graph minor containment relation $\leq$ is a well-quasi-order on the set
    of connected $\Z_2$-coloured graphs.
\end{cor}
\begin{proof}
    Let $(G_1,b_1), (G_2,b_2),\ldots$ be an infinite sequence of
    $\Z_2$-coloured connected graphs. Then there must be an infinite sequence
    $1 \leq i_1 < i_2 < \ldots$ such that either $b_{i_j}$ is odd for
    all $j$, or $b_{i_j}$ is even for all $j$. Since $\leq$ is a well-quasi-order
    on uncoloured graphs, there must be indices $j < j'$ such that
    $G_{i_j} \leq G_{i_{j'}}$, and by Lemma \ref{lem:minor_colour},
    $(G_{i_j},b_{i_j}) \leq (G_{i_{j'}}, b_{i_{j'}})$.
\end{proof}

Corollary \ref{cor:minor_colour} allows us to prove Corollary \ref{cor:main}.
\begin{proof}[Proof of Corollary \ref{cor:main}]
    By Lemma \ref{lem:main}, ``$\Gamma(G,b)$ satisfies $\mpP$'' is a
    minor-closed property $\mpP'$ of connected $\Z_2$-coloured graphs. Let
    $\mcS$ be the set of connected $\Z_2$-coloured graphs not satisfying
    $\mpP'$.  By Corollary \ref{cor:minor_colour}, there is a finite subset
    $\mcF$ of $\mcS$ such that every element of $\mcS$ contains an element of
    $\mcF$ as a minor. Conversely, since $\mpP'$ is minor-closed, any graph
    containing an element of $\mcF$ as a minor cannot satisfy $\mpP'$.
\end{proof}
\begin{rmk}
    The proof actually shows that there is a finite set $\mcF$ of connected
    $\Z_2$-coloured graphs such that $\Gamma(G,b)$ satisfies $\mpP$ if and
    only if $(G,b)$ avoids $\mcF$. However, sometimes it is convenient to use
    disconnected graphs when writing down forbidden minors for connected graphs.
    For instance, in Theorems \ref{thm:finite} and \ref{thm:abelian}, it is
    conceptually simpler to use $C_2 \sqcup C_2$ as a forbidden minor, although
    we could use a connected graph in its place.
\end{rmk}

\section{Arkhipov's theorem and pictures}\label{sec:pic}

Theorem \ref{thm:CLS} and Lemma \ref{lem:minor_colour} imply that Arkhipov's
theorem (Theorem \ref{thm:Arkhipov}) can be restated in the following way:
\begin{thm}\label{thm:Arkhipov2}
    Let $(G,b)$ be a connected $\Z_2$-coloured graph. Then the following
    are equivalent:
    \begin{enumerate}[(a)]
        \item $J_{G,b} = 1$ in $\Gamma(G,b)$.
        \item $J_{G,b}$ is trivial in finite-dimensional representations
            of $\Gamma(G,b)$.
        \item $(G,b)$ avoids $(K_{3,3},b')$ with $b'$ odd, $(K_5,b')$ with $b'$
            odd, and $(K_1,b')$ with $b'$ even.
    \end{enumerate}
\end{thm}
Suppose that $\phi : \Phi \arr \Psi$ is a homomorphism of groups over $\Z_2$,
so $\phi(J_{\Phi}) = J_{\Psi}$. If $J_{\Phi}=1$, then $J_{\Psi}=1$, so as
mentioned in the introduction, ``$J_{\Phi} = 1$'' is a quotient closed property
of groups over $\Z_2$. Similarly, if $\psi : \Psi \arr U(\C^n)$ is a
finite-dimensional representation of $\Psi$ for some $n \geq 1$ such that
$\psi(J_{\Psi}) \neq \Id$, then $\psi \circ \phi$ is a finite-dimensional
representation of $\Phi$ with $\psi \circ \phi(J_{\Phi}) \neq \Id$. So
``$J_{\Phi}$ is trivial in finite-dimensional representations of $\Phi$'' is
also a quotient closed property. As a result, Corollary \ref{cor:main} implies
that both properties can be characterized by forbidden minors. However,
Corollary \ref{cor:main} does not explain why these properties are equivalent,
or why they are related to planarity of $G$. In this section, we show
how to prove Theorem \ref{thm:Arkhipov2} in the group-theoretic language of
Lemma \ref{lem:main}, in a way that explains the equivalence of these two
properties, and the relation to planarity.

For this proof, we recall the notion of \emph{pictures} of groups. Pictures
provide a graphical representation of relations in a group, and are a standard
tool in combinatorial group theory~\cite{Sho07,Ol12} (although the planar duals
of pictures, called van Kampen diagrams, are more common). For
solution groups of linear systems, pictures are particularly nice, since it is
not necessary to keep track of the order of generators in the defining relations.
A detailed definition of pictures for solution groups can be found in
\cite[Definition 7.2]{Slof16}. For the convenience of the reader, we give a
streamlined version:
\begin{defn}
    A \emph{picture} over a graph $G$ is an embedded planar graph $\mcP$ with a
    distinguished vertex $v_b$, called the \emph{boundary vertex}, in the
    outside face, as well as labelling functions $h_E : E(\mcP) \arr E(G)$ and
    $h_V : V(\mcP) \setminus \{v_b\} \arr V(G)$, such that $h_E|_{E(v)}$ is a
    bijection between $E(v)$ and $E(h_V(v))$ for all $v \in V(\mcP) \setminus
    \{v_b\}$.

    A word $e_1 \cdots e_k$ over $E(G)$ is a \emph{boundary word} of a picture $\mcP$
    if the edges incident to the boundary vertex are labelled by $e_1,\ldots,e_k$
    when read in counterclockwise order from some starting edge. A picture is
    \emph{closed} if $E(v_b) = \emptyset$. The \emph{character} of a picture
    $\mcP$ is the vector $\chi(\mcP) \in \Z_2^{V(G)}$ with $\chi(\mcP)(v) =
    |h_V^{-1}(v)| \in \Z_2$.
\end{defn}
For simplicity, we use the same conventions for pictures as for graphs, in that
multiedges are allowed, but loops are not.  Note that if $e_1 \cdots e_k$ is a
boundary word of a picture $\mcP$, then every cyclic shift of this word is also
a boundary word, since we can choose any starting edge.

When drawing pictures, we usually blow up the boundary vertex to a disk, and then
think of the interior of this disk as the outside face. This gives a drawing of
the picture inside a closed disk, with the boundary of the disk corresponding
to the boundary vertex, as shown in Figures \ref{fig:k33_J_picture} and
\ref{fig:k5_J_picture}. Given such a drawing, we can shrink the boundary disk
down to a vertex to get a drawing of the picture with the boundary vertex in
the outside face, so these two ways of drawing a picture are equivalent.

Recall that a \emph{graph homomorphism} $\phi : G \arr H$ is a function $\phi_V : V(G)
\arr V(H)$ such that if $v, w \in V(G)$ are adjacent in $G$, then $\phi_V(v)$
and $\phi_V(w)$ are adjacent in $H$. In particular, $\phi_V(v) \neq \phi_V(w)$ if
$v$ and $w$ are adjacent in $G$. If $G$ and $H$ do not have multiple edges
between vertices, then given a graph homomorphism $\phi_V : G \arr H$, we can
define a function $\phi_E : E(G) \arr E(H)$ by sending $e \in E(G)$ with endpoints
$v, w \in V(G)$ to $e' \in E(H)$ with endpoints $\phi_V(v)$ and $\phi_V(w)$.
A graph homomorphism between graphs without multiple edges is a \emph{cover} if
$\phi_E|_{E(v)}$ is a bijection from $E(v)$ to $E(\phi_V(v))$ for all $v \in
V(G)$. To extend this concept to graphs with multiple edges between vertices,
we say that a \emph{cover} of a graph $H$ is a homomorphism $\phi_V : G \arr H$
along with a function $\phi_E : E(G) \arr E(H)$, such that if $e \in E(G)$ has
endpoints $v,w$ then $\phi_E(e)$ has endpoints $\phi_V(v),\phi_V(w)$, and such
that $\phi_E|_{E(v)}$ is a bijection from $E(v)$ to $E(\phi_V(v))$ for all $v
\in V(G)$.  A \emph{planar cover} of $H$ is a cover $G \arr H$ in which $G$ is
planar.  An \emph{embedded planar cover} of $H$ is a planar cover $G \arr H$
along with a choice of planar embedding of $G$. By thinking of $G$ as
embedded in a closed disk (or equivalently, by adding a boundary vertex
to the outer face), any embedded planar cover $G \arr H$ can be
regarded as a closed picture over $H$ with labelling functions $\phi_V$ and
$\phi_E$.

Conversely, if $\mcP$ is a closed picture over $H$, then all the data of $\mcP$
is contained in the embedded planar graph $\mcP \setminus v_b$.  If $e \in
E(\mcP)$ has endpoints $v,w \in V(\mcP)$, then $h_E(e)$ must be incident to
$h_V(v)$ and $h_V(w)$. However, this does not necessarily imply that $h_V(v)$
and $h_V(w)$ are adjacent, since $h_V(v)$ and $h_V(w)$ could be equal. As a
result, $h_V$ might not be a homomorphism, with the consequence that $\mcP
\setminus v_b$ is not necessarily an embedded planar cover of $H$. Thus we can
think of pictures as generalizations of planar covers, which preserve incidence
rather than adjacency. A closed picture $\mcP$ comes from an embedded planar
cover if and only if $h_V$ is a homomorphism, which happens if and only if
$h_V(v) \neq h_V(w)$ whenever vertices $v$ and $w$ are adjacent in $\mcP$.

The following lemma shows that boundary words of pictures give relations in the
graph incidence group, and that all relations in the group arise in this way.
This lemma is essentially due to van Kampen \cite{VK33}. A proof of this
version of the lemma can be found in \cite{Slof16}.
\begin{lemma}[van Kampen lemma]\label{lem:vk}
    Let $(G,b)$ be a $\Z_2$-coloured graph. Then $x_{e_1}\cdots x_{e_k}=J^a$ in
    $\Gamma(G,b)$ if and only if $e_1 \cdots e_k$ is the boundary word of a
    picture $\mcP$ over $G$ with $\chi(\mcP) \cdot b = a$. In
    particular, $J=1$ if and only if there is a closed picture $\mcP$ with
    $\chi(\mcP) \cdot b = 1$.
\end{lemma}
In this lemma, if $b,b' \in \Z_2^{V(G)}$, then $b' \cdot b := \sum_{v \in V(G)}
b'(v) b(v)$.

As an example of Lemma \ref{lem:vk}, we prove the following lemma:
\begin{lemma}\label{lem:commutatorpic}
    Let $G$ be $K_{3,3}$ or $K_5$, and let $b$ be a $\Z_2$-colouring of $G$ of
    parity $a$.  If $e$ and $f$ are two edges of $G$ which are not incident to
    a common vertex, then
    \begin{equation*}
        [x_e,x_f] = J^a
    \end{equation*}
    in $\Gamma(G,b)$.
\end{lemma}
\begin{proof}
    It is not hard to see that $G$ has a planar drawing with a single crossing
    between edges $e$ and $f$. If we replace the crossing point with a new
    boundary vertex $v_b$, we get a $G$-picture $\mcP$ with boundary word $e f
    e f$, as shown in Figures \ref{fig:k33_J_picture} and \ref{fig:k5_J_picture}.

    Because every vertex of $G$ occurs exactly once in the picture, the
    character $\chi(\mcP)$ is the vector of all $1$'s in $\Z_2^{V(G)}$.
    Thus $\chi(\mcP) \cdot b = \sum_{v \in V(G)} b(v) = a$, and the conclusion
    follows from Lemma \ref{lem:vk}.
\end{proof}

\begin{figure}[t]
\begin{center}
    \includegraphics[scale=1]{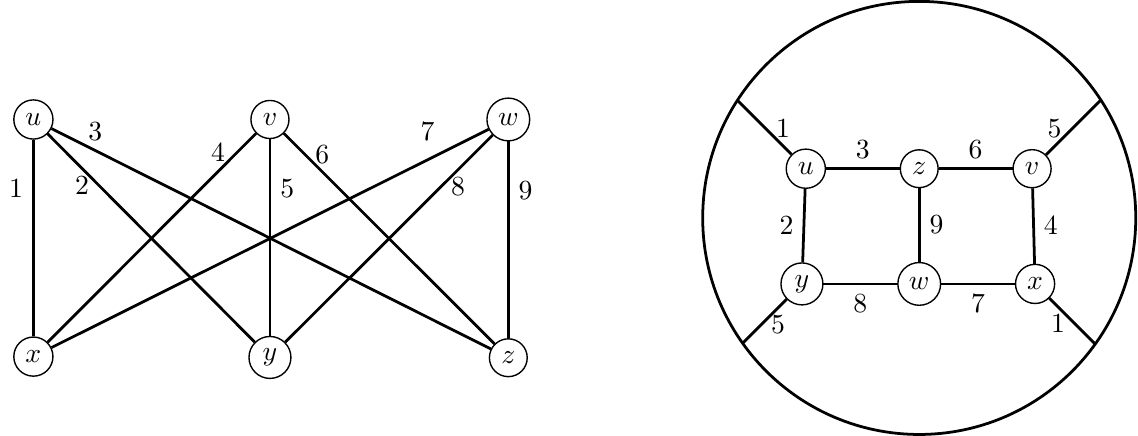}
\caption{On the left is $K_{3,3}$ with edges labelled from $1$ to $9$, and vertices
labelled from $u$ to $z$. On the right is a $K_{3,3}$-picture with boundary
word $1\ 5\ 1\ 5$.}
    \label{fig:k33_J_picture}
\end{center}
\end{figure}

\begin{figure}[t]
\begin{center}
    \includegraphics[scale=1]{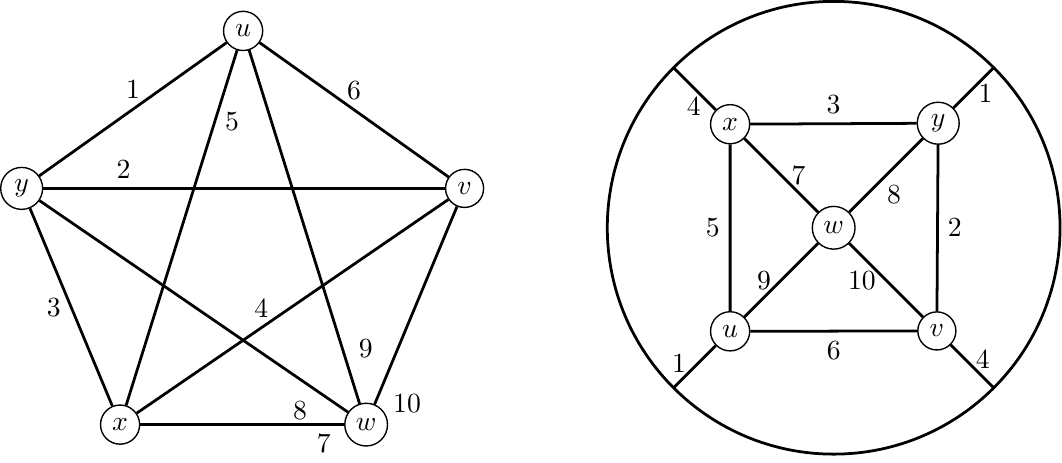}
\caption{On the left is $K_{5}$ with edges labelled from $1$ to $10$, and vertices
labelled from $u$ to $y$. On the right is a $K_{5}$-picture with boundary
word $1\ 4\ 1\ 4$.}
    \label{fig:k5_J_picture}
\end{center}
\end{figure}

To prove Theorem \ref{thm:Arkhipov2}, we need to know that if $G = K_{3,3}$ or
$K_5$, and $b$ is an odd parity colouring, then $\Gamma(G,b)$ is finite and
$J \neq 1$. This can be done directly on a computer, as in Example
\ref{ex:K33}. There is also a nice expression for these groups due to
\cite{CS17a}, which we now explain. Recall that the dihedral group $\Dih_n$ is
the group with presentation
\begin{equation}\label{eq:dihfin}
    \Dih_n=\left\langle z_1,z_2: z_1^2=z_2^2=(z_1z_2)^n=1\right\rangle\,.
\end{equation}
$\Dih_n$ is a finite group of order $2n$, and is nonabelian for $n \geq 3$. If
$n$ is even, then the center of $\Dih_n$ has a single non-trivial element $(z_1
z_2)^{n/2}$. When $n=4$, this central element has order $2$.

Let $\psi : Z(\Psi_1) \arr Z(\Psi_2)$ be an isomorphism between the centers of
two groups $\Psi_1$ and $\Psi_2$. The central product of $\Psi_1$ and $\Psi_2$
is the quotient of the product $\Psi_1 \times \Psi_2$ by the normal subgroup
generated by $(z,\psi^{-1}(z))$ for $z \in Z(\Psi_1)$. We denote the central
product by $\Psi_1 \odot_{\psi} \Psi_2$, or $\Psi_1 \odot \Psi_2$ if the
isomorphism $\psi$ is clear. The groups $\Psi_1$ and $\Psi_2$ are naturally
subgroups of $\Psi_1 \odot \Psi_2$, and the center of $\Psi_1 \odot \Psi_2$ is
$Z(\Psi_1) = Z(\Psi_2)$, considered as a subgroup of $\Psi_1 \odot \Psi_2$.
\begin{prop}[\cite{CS17a}]\label{prop:K33K5}\ 
    \begin{enumerate}[(a)]
        \item Let $b$ be an odd parity colouring of $K_{3,3}$. Then $$\Gamma(K_{3,3},b)
            \iso \Dih_4 \odot \Dih_4$$ via an isomorphism which sends $J \in \Gamma(K_{3,3},b)$
            to the unique non-trivial central element of $\Dih_n \odot \Dih_n$.
        \item Let $b$ be an odd parity colouring of $K_{5}$. Then $$\Gamma(K_5,
            b) \iso \Dih_4 \odot \Dih_4 \odot \Dih_4$$ via an isomorphism which
            sends $J \in \Gamma(K_5,b)$ to the unique non-trivial central
            element of $\Dih_4 \odot \Dih_4 \odot \Dih_4$.
    \end{enumerate}
\end{prop}
Although a proof of Proposition \ref{prop:K33K5} can be found in \cite{CS17a},
we provide a proof of this proposition for the convenience of the reader.
\begin{proof}
    For part (a), we use the vertex and edge labelling of $K_{3,3}$
    shown in Figure \ref{fig:k33_J_picture}. Let $b$ be the colouring with
    $b(w)=1$ and $b(t)=0$ for all other vertices $t \neq w$. Since the edges of
    $K_{3,3}$ are labelled from $1$ to $9$, the group $\Gamma(K_{3,3},b)$ is
    generated by $x_1,\ldots,x_9$.  Using the presentation of $\Dih_4$ from
    Equation~\eqref{eq:dihfin}, we see that
    \begin{align*}
        \Dih_4 \odot \Dih_4 = \langle z_{11}, z_{12}, z_{21}, z_{22} :\
            & z_{ij}^2 = 1 \text{ for all } i,j \in \{1,2\}, \\
            & (z_{i1} z_{i2})^4 = 1 \text{ for } i=1,2, \\
            & [z_{1i},z_{2j}]=1 \text{ for all } i,j \in \{1,2\}, \\
            & (z_{11} z_{12})^2 = (z_{21} z_{22})^2\ \rangle.
    \end{align*}
    Now we can define a homomorphism
    \begin{equation*}
        \phi : \Dih_4 \odot \Dih_4 \arr \Gamma(K_{3,3},b)
    \end{equation*}
    by setting $\phi(z_{11}) = x_1$, $\phi(z_{12}) = x_5$, $\phi(z_{21}) = x_2$, and
    $\phi(z_{22}) = x_4$. To see that $\phi$ is a homomorphism, observe that $x_i^2=1$
    for all $1 \leq i \leq 9$. Also, edges $2$ and $4$ both share common vertices with
    edges $1$ and $5$, so $x_2$ and $x_4$ both commute with $x_1$ and $x_5$. Since
    $1$ and $5$ are not incident with a common vertex, $(x_1 x_5)^2 = J$ by
    Lemma \ref{lem:commutatorpic}. As a result, $(x_1 x_5)^4 = 1$. Similarly, $(x_2 x_4)^2 = J$,
    so $(x_2 x_4)^4 = 1$, and $(x_1 x_5)^2 = (x_2 x_4)^2$. Thus $x_1$, $x_5$, $x_2$, and
    $x_4$ satisfy the defining relations for $\Dih_4 \odot \Dih_4$, and hence
    $\phi$ is well-defined.

    To see that $\phi$ is an isomorphism, we define
    \begin{equation*}
        \phi^{-1} : \Dih_4 \odot \Dih_4 \arr \Gamma(K_{3,3},b)
    \end{equation*}
    by setting
    \begin{equation}\label{Eq:D4oD4}
    \begin{split}
        & \phi^{-1}(x_1) = z_{11}, \quad
        \phi^{-1}(x_2) = z_{21}, \quad
        \phi^{-1}(x_3) = z_{11} z_{21}, \\
        & \phi^{-1}(x_4) = z_{22},  \quad
        \phi^{-1}(x_5) = z_{12}, \quad
        \phi^{-1}(x_6) = z_{12} z_{22}, \\
        &\phi^{-1}(x_7) = z_{11} z_{22}, \quad
        \phi^{-1}(x_8) = z_{12} z_{21},  \quad
         \phi^{-1}(x_9) = z_{11} z_{12} z_{21} z_{22}, \text{ and}\\
        & \phi^{-1}(J) = (z_{11} z_{12})^2 = (z_{21} z_{22})^2.
    \end{split}
    \end{equation}
    We can check that $\phi^{-1}(r)=1$ for all defining relations $r$ of
    $\Gamma(K_{3,3},b)$, so $\phi^{-1}$ is well-defined as a homomorphism.
    For instance,
    \begin{equation*}
        \phi^{-1}(x_7 x_8 x_9) = (z_{11} z_{12} z_{11} z_{12}) (z_{22} z_{21} z_{21} z_{22})
            = (z_{11} z_{12})^2 = \phi^{-1}(J),
    \end{equation*}
    while
    \begin{equation*}
        \phi^{-1}(x_3 x_6 x_9) = (z_{11} z_{12} z_{11} z_{12})(z_{21} z_{22} z_{21} z_{22}) = (z_{11} z_{12})^2 (z_{21} z_{22})^2 = 1,
    \end{equation*}
    matching the colouring $b(w)=1$ and $b(z)=0$. Thus $\phi$ is an isomorphism.
    By Lemma \ref{lem:parity}, part (a) is true for any other odd parity colouring
    of $K_{3,3}$.

    The proof of part (b) is similar. We use the vertex and edge labelling of $K_5$ from
    Figure \ref{fig:k5_J_picture}, and let $b$ be the $\Z_2$-colouring with $b(w)=1$
    and $b(t)=0$ for vertices $t \neq w$. The group $G = \Dih_4 \odot \Dih_4 \odot \Dih_4$
    has presentation
    \begin{align*}
        G = \langle z_{ij},\ i \in \{1,2,3\}, \ j\in \{1,2\}  : \
            & z_{ij}^2 = 1 \text{ for all } i \in \{1,2,3\}, j\in\{1,2\}, \\
            & (z_{i1} z_{i2})^4 = 1 \text{ for all } i \in \{1,2,3\}, \\
            & [z_{ij},z_{kl}]=1 \text{ for all } i \neq k \in \{1,2,3\}, j,l \in \{1,2\}, \\
            & (z_{11} z_{12})^2 = (z_{21} z_{22})^2 = (z_{31} z_{32})^2\ \rangle.
    \end{align*}
    To define an isomorphism $\psi : G \arr \Gamma(K_5,b)$, we set
    \begin{equation*}
        \psi(z_{11}) = x_1,\ \psi(z_{12}) = x_4,\ \psi(z_{21}) = x_2,\ \psi(z_{22}) = x_5,\ \psi(z_{31}) = x_3,\ \psi(z_{32})= x_6.
    \end{equation*}
    That $\phi$ is well-defined follows from the same arguments as in part (a); in particular, we once again
    use Lemma \ref{lem:commutatorpic} to see that $(x_1 x_4)^2 = (x_2 x_5)^2 = (x_3 x_6)^2 = J$.
    For the inverse, we define $\phi^{-1} : \Gamma(K_{5},b) \arr G)$ as the inverse of $\phi$ on
    $x_1,\ldots,x_6$, and setting
    \begin{align*}
        & \phi^{-1}(x_7) = \phi^{-1}(x_3 x_4 x_5) = z_{12} z_{22} z_{31},\quad
        \phi^{-1}(x_8) = \phi^{-1}(x_1 x_2 x_3) = z_{11} z_{21} z_{31}, \\
        & \phi^{-1}(x_9) = \phi^{-1}(x_1 x_5 x_6) = z_{11} z_{22} z_{32}, \quad
        \phi^{-1}(x_{10}) = \phi^{-1}(x_2 x_4 x_6) = z_{12} z_{21} z_{32}, \text{ and}\\
        & \phi^{-1}(J) = (z_{11} z_{12})^2 = (z_{21} z_{22})^2 = (z_{31} z_{32})^2.
    \end{align*}
    The defining relations of $\Gamma(K_5,b)$ for vertices $u$, $v$, $x$, and $y$ follow immediately
    from the definition, while for vertex $w$ we have
    \begin{equation*}
        \phi^{-1}(x_7 x_8 x_9 x_{10}) = (z_{12} z_{11} z_{11} z_{12}) (z_{22} z_{21} z_{22} z_{21})
            (z_{31} z_{31} z_{32} z_{32}) = (z_{21} z_{22})^2 = \phi^{-1}(J),
    \end{equation*}
    which matches the colouring $b(w)=1$. So $\phi^{-1}$ is well-defined as a homomorphism,
    and hence $\phi$ is an isomorphism.
\end{proof}

Proposition \ref{prop:K33K5} also allows us to determine $\Gamma(K_{3,3})$ and
$\Gamma(K_5)$. This will be used in the next section.
\begin{cor}\label{cor:K33K5}\ 
    \begin{enumerate}[(a)]
        \item $\Gamma(K_{3,3}) = \Z_2^4$, and if $b$ is an even parity colouring
            then $\Gamma(K_{3,3},b) = \Z_2^5$.  
        \item $\Gamma(K_5) = \Z_2^6$, and if $b$ is an even parity colouring
            then $\Gamma(K_5,b) = \Z_2^7$.
    \end{enumerate}
\end{cor}
\begin{proof}
    $\Gamma(K_{3,3}) = \Gamma(K_{3,3},b) / \langle J \rangle$ for any colouring $b$ of
    $K_{3,3}$.  Take $b$ to be an odd parity colouring, so that
    $\Gamma(K_{3,3},b) = \Dih_4 \odot \Dih_4$ has the presentation from
    Equation \eqref{Eq:D4oD4}. Setting $J = (z_{11} z_{12})^2 = 1$ in
    this presentation, we get that
    \begin{align*}
        \Gamma(K_{3,3}) = \langle z_{ij}, i,j \in \{1,2\} :\ & (z_{ij})^2 = 1 \text{ for all } i,j \in \{1,2\} \\
                                    & [z_{ij},z_{kl}] = 1 \text{ for all } i,j,k,l \in \{1,2\} \rangle = \Z_2^4.
    \end{align*}
    If $b$ is even parity, then by Lemma \ref{lem:parity}, $\Gamma(K_{3,3},b) =
    \Gamma(K_{3,3}) \times \Z_2 = \Z_2^5$.  The proof of (b) is similar.
\end{proof}

\begin{proof}[Proof of Theorem \ref{thm:Arkhipov2}]
    Let $(G,b)$ be a $\Z_2$-coloured connected graph. Suppose that $(G,b)$ satisfies
    (c), so that $(G,b)$ avoids $(K_{5},b')$ with $b$ odd, $(K_{3,3},b')$ with
    $b$ odd, and $(K_1,b')$ with $b'$ even. $G$ contains $K_1$ as a minor, so
    by Lemma \ref{lem:parity}, $b$ must be odd, and $G$ cannot contain $K_5$ or
    $K_{3,3}$ as a minor. But this implies that $G$ is planar. Choosing an
    embedding of $G$ in a closed disk, and setting $h_V(v) = v$ and $h_E(e) =
    e$, we get a closed picture $\mcP$ of $G$ with character $\chi(\mcP)$ equal
    to the vector of all $1$'s. Hence $\chi(\mcP) \cdot b = 1$, the parity of
    $b$, and by Lemma \ref{lem:vk}, we have $J=1$ in $\Gamma(G,b)$. Then
    $J$ is also trivial in finite-dimensional representations of $\Gamma(G,b)$,
    so $(G,b)$ satisfies (a) and (b).

    Suppose that $(G,b)$ does not satisfy (c). If $(G,b)$ contains $(K_1,b')$
    with $b'$ even, then by Lemma \ref{lem:minor_colour} the colouring $b$ must
    also have even parity. By Lemma \ref{lem:parity}, there is an isomorphism
    $\Gamma(G,b) \iso \Gamma(G,0) \iso \Gamma(G) \times \Z_2$ sending $J_{G,b}$
    to the generator of the $\Z_2$ factor. Composing with the projection
    $\Gamma(G) \times \Z_2 \arr \Z_2$ and identifying $\Z_2$ with the subgroup
    ${\pm 1} \subset U(\C^1)$, we see that $J$ is non-trivial in
    finite-dimensional representations of $\Gamma(G,b)$.

    If $(G,b)$ contains $(H,b')$ where $H = K_5$ or $K_{3,3}$ and $b'$ has
    odd parity, then by Lemma \ref{lem:main} there is a homomorphism $\phi :
    \Gamma(G,b) \arr \Gamma(H,b')$ with $\phi(J_{G,b}) = J_{H,b'}$. By
    Proposition \ref{prop:K33K5}, $\Gamma(H,b')$ is a finite group and
    $J_{H,b'} \neq 1$. Since $\Gamma(H,b')$ is finite, it has a faithful
    finite-dimensional representation, and composing with this representation,
    we see that $J_{G,b}$ is non-trivial in finite-dimensional representations
    of $\Gamma(G,b)$.

    In both cases, if $(G,b)$ does not satisfy (c), then $(G,b)$ does not
    satisfy (b). Since (a) implies (b), conditions (a), (b), and (c) are
    equivalent.
\end{proof}
Another equivalent formulation of Theorem \ref{thm:Arkhipov2} is that a graph
$G$ is planar if and only if there is a picture $\mcP$ with character
$\chi(\mcP) \neq 0$. Indeed, if there is a picture $\mcP$ over $G$ with
$\chi(\mcP)(v)=1$, then let $b$ be the $\Z_2$-colouring of $G$ with $b(v)=1$, and
$b(t)=0$ for all vertices $t \neq v$. Then $\chi(G)\cdot b = 1$, so $J=1$ in
$\Gamma(G,b)$, and $G$ must be planar by Theorem \ref{thm:Arkhipov2}.
Conversely, if $G$ is planar, then $G$ itself can be turned into a closed
picture $\mcP$ with $\chi(\mcP)$ equal to the vector of all $1$'s.

More generally, any planar cover $\mcP$ of $G$ can be turned into a closed
picture over $G$ as discussed above. If $G$ is connected and $\phi : \mcP \arr G$ is a planar
cover of $G$, then the function $|\phi_V^{-1}(v)|$ is a constant function of $v
\in V(G)$ (something that is not true of pictures in general).  This constant
is called the fold number. If this planar cover is made into a picture $\mcP$,
then $\chi(\mcP)$ is the zero vector if the fold number is even, and the vector
of all $1$'s if the fold number is odd. Thus from Theorem \ref{thm:Arkhipov2}
we recover a result of Archdeacon and Richter that a graph is planar if and
only if it has a planar cover with odd fold number \cite{AR90}. Arkhipov's
theorem can be thought of as a strengthening of Archdeacon and Richter's
result that includes arbitrary pictures, not just planar covers.

To finish the section, we observe that Theorem \ref{thm:Arkhipov2} can be
easily extended to the case that $G$ is disconnected:
\begin{cor}\label{cor:Arkhipov2}
    Let $(G,b)$ be a $\Z_2$-coloured graph. Then the following are equivalent:
    \begin{enumerate}[(a)]
        \item $J_{G,b}=1$ in $\Gamma(G,b)$. 
        \item $J_{G,b}$ is trivial in finite-dimensional representations of $\Gamma(G,b)$.
        \item There is some connected component $G'$ of $G$ such that $G'$ is planar
            and the restriction of $b$ to $G'$ is odd. 
    \end{enumerate}
\end{cor}
\begin{proof}
    Let $(G_1,b_1),\ldots,(G_k,b_k)$ be the connected components of $(G,b)$.
    By Lemma \ref{lem:disconnected}, $J_{G,b} = 1$ in $\Gamma(G,b)$ if and only
    if $J_{G_i,b_i} = 1$ in $\Gamma(G_i,b_i)$ for some $1 \leq i \leq k$.
    Clearly (a) implies (b). If $J_{G,b} \neq 1$, then by Theorem
    \ref{thm:Arkhipov2}, $J_{G_i,b_i}$ is non-trivial in finite-dimensional
    representations of $\Gamma(G_i,b_i)$ for all $1 \leq i \leq k$. As noted
    after Theorem \ref{thm:CLS}, this means that for each $i$, we can find
    a finite-dimensional representation $\psi_i$ of $\Gamma(G_i,b_i)$ on 
    $\C^{n_i}$ with $\psi_i(J_{G_i,b_i})=-\Id$. Let $m$ be the least common
    multiple of $n_1,\ldots,n_k$. Then $\psi_i^{\oplus m / n_i}$ is a
    representation of $\Gamma(G_i,b_i)$ on $\C^{m}$ sending $J_{G_i,b_i}\mapsto
    -\Id$. By Lemma \ref{lem:disconnected}, there is a representation $\psi$
    of $\Gamma(G,b)$ on $\C^{m}$ sending $J \mapsto -\Id$. So (a) and
    (b) are equivalent. 

    By Theorem \ref{thm:Arkhipov2}, $J_{G_i,b_i} = 1$ if and only if
    $b_i$ is odd, and $G_i$ is planar, so (a) and (c) are also equivalent. 
\end{proof}
While part (c) of Corollary \ref{cor:Arkhipov2} is a practical criterion for
testing $J_{G,b}=1$, it cannot be phrased as a pattern avoidance criterion with
a finite list of minors, for the same reason that $\leq$ is not a well-quasi-order
in Example \ref{ex:nonwqo}.

\section{Excluded $\Z_2$-graph minors for finiteness and abelianness}\label{sec:2cycles}

In this section we prove Theorems \ref{thm:finite} and \ref{thm:abelian} by
finding the the excluded $\Z_2$-graph minors for finiteness and abelianness of
graph incidence groups. As we will see, the proof reduces to the following
statements about the graph incidence groups of uncoloured graphs:
\begin{prop}\label{prop:finite}
    $\Gamma(G)$ is finite if and only if $G$ avoids $C_2 \sqcup C_2$ and $K_{3,6}$.
\end{prop}

\begin{prop}\label{prop:abelian}
    $\Gamma(G)$ is abelian if and only if $G$ avoids $C_2 \sqcup C_2$ and $K_{3,4}$.
\end{prop}

To explain the strategy of the proofs, we start with the following easy lemma.
Recall that a cycle is a connected graph where every vertex has degree $2$.
\begin{lemma}\label{lem:cycle}
    If $C$ is a cycle, then $\Gamma(C) \iso \Z_2$.
\end{lemma}
\begin{proof}
    Suppose $C$ has vertices $v_i$, $i \in \Z_n$, where $v_i$ is adjacent to
    $v_{i-1}$ and $v_{i+1}$ for all $i \in \Z_n$. For every $i \in \Z_n$, let
    $e_i$ be the edge connecting $v_i$ and $v_{i+1}$. Then $\Gamma(C)$ is
    generated by $x_{e_i}$ for $i \in \Z_n$, subject to the relations
    $x_{e_i}^2=1$ for all $i \in \Z_n$, $[x_{e_{i-1}},x_{e_i}]=1$ for
    all $i \in \Z_n$, and $x_{e_{i-1}} x_{e_{i}}=1$ for all $i \in \Z_n$.
    These last relations imply that $x_{e_i} = x_{e_j}$ for all $i,j$,
    so replacing all generators with a single generator $x$, we see that
    the defining presentation of $\Gamma(C)$ is equivalent to
    $\langle x : x^2=1\rangle = \Z_2$.
\end{proof}
Suppose $G$ contains two vertex disjoint cycles, or equivalently, that $G$
contains the disconnected union $C_2 \sqcup C_2$ of two-cycles $C_2$ as a
graph minor. By Lemmas \ref{lem:main_uncoloured} and \ref{lem:disconnected2},
there is a surjective homomorphism
\begin{equation*}
    \Gamma(G) \arr \Gamma(C_2 \sqcup C_2) = \Gamma(C_2) * \Gamma(C_2) = \Z_2 * \Z_2.
\end{equation*}
The group $\Z_2 * \Z_2$ is infinite and nonabelian, so we immediately see:
\begin{corollary}\label{cor:2cyc}
    If $G$ contains $C_2 \sqcup C_2$ as a graph minor, then $\Gamma(G)$ is infinite and nonabelian.
\end{corollary}
So if $\Gamma(G)$ is finite, then $G$ cannot have two vertex disjoint cycles.
Graphs without two disjoint cycles have been characterized by Lovasz
\cite{Lov65}. To state this characterization, observe that $G$ does not have
two disjoint cycles if it satisfies one of the following conditions:
\begin{enumerate}[(i)]\label{enum:lovasz}
    \item $G\setminus v$ is acyclic (and possibly empty) for some $v\in V(G)$,
    \item $G$ is a wheel whose spokes may be multiedges,
    \item $G$ is $K_5$, or
    \item $G$ is obtained from $K_{3,n}$ for some $n\geq 0$ by adding edges between vertices in the first partition.
        ($K_{3,0}$ refers to the graph with $3$ vertices and no edges.)
\end{enumerate}
Recall that \emph{edge subdivision} is a graph operation in which an edge $e$
is replaced by (or \emph{subdivided} into) two edges joined to a new vertex of degree two.
A graph $G_0$ is said to be a subdivision of $G$ if $G_0$ can be obtained from
$G$ by repeated edge subdivision. We consider $G$ to be a subdivision of
itself. If $G_0$ is a subdivision of $G$, then every cycle of $G_0$ is a
subdivision of a cycle of $G$, so if $G$ does not contain two disjoint
cycles, then $G_0$ also does not contain two disjoint cycles.

An acyclic graph, also called a forest, is a graph without cycles. We can
\emph{add a forest} to a graph $G$ by taking the disconnected union of $F$ and $G$,
and then adding edges between $F$ and $G$ such that there is at most one edge
between $G$ and every connected component of $F$. The only cycles in the
resulting graph $\widetilde{G}$ are the cycles of $G$, so if $G$ does not
contain two disjoint cycles, then neither does $\widetilde{G}$.

Starting from a graph without two disjoint cycles, edge subdivision and
adding a forest give two ways of constructing a new graph without two
disjoint cycles.  Lovasz's characterization states that all graphs without two
disjoint cycles arise in this way from one of the graphs satisfying conditions
(i)-(iv):
\begin{theorem}[\cite{Lov65}, see also \cite{Bol04}]\label{thm:lovasz}
    A graph $\widetilde{G}$ does not contain two vertex disjoint cycles if and
    only if $\widetilde{G}$ can be obtained from a graph $G$ satisfying one
    of the conditions (i)-(iv) by taking a subdivision $G_0$ of $G$, and then
    adding a (possibly empty) forest $F$ with at most one edge between $G_0$
    and each connected component of $F$.
\end{theorem}
We note that the families of graphs defined by conditions (i)-(iv) are not
disjoint, so Theorem \ref{thm:lovasz} does not give a unique way of
constructing every graph without two disjoint cycles. For a more precise
statement where the categories do not overlap, see \cite[Theorem
III.2.3]{Bol04}.

It's not hard to see that subdividing and adding forests to $G$
does not change $\Gamma(G)$:
\begin{lemma}\label{lem:subdivision}
    Let $G_0$ be a subdivision of $G$. Then $\Gamma(G_0) \iso \Gamma(G)$.
\end{lemma}
\begin{proof}
    Suppose $G_1$ is the result of subdividing an edge $e$ of $G$ into two new
    edges $e_0$ and $e_1$, joined by the new vertex $v$. In the presentation
    of $\Gamma(G_1)$, the relation for vertex $v$ implies that $x_{e_0} = x_{e_1}$.
    Replacing $x_{e_0}$ and $x_{e_1}$ with $x_e$, we see that the presentation
    of $\Gamma(G_1)$ is equivalent to the presentation of $\Gamma(G)$.
    Repeating this fact shows that $\Gamma(G_0) \iso \Gamma(G)$ for any subdivision
    $G_0$ of $G$.
\end{proof}

\begin{lemma}\label{lem:addingforest}
    Suppose $\widetilde{G}$ is obtained from a graph $G$ by adding a forest $F$
    such that there is at most one edge between $G$ and every connected component
    of $F$. Then $\Gamma(\widetilde{G}) \iso \Gamma(G)$.
\end{lemma}
\begin{proof}
    We can prove this by induction on the size of $F$. If $F$ is empty, then
    the lemma is clear. If $F$ has an isolated vertex $v$ which is also
    isolated in $\widetilde{G}$, then by Proposition \ref{prop:minors}, part (iii),
    $\Gamma(\widetilde{G}) \iso \Gamma(\widetilde{G} \setminus v)$.
    Suppose $F$ is non-empty, and does not have an isolated vertex which is
    also isolated in $\widetilde{G}$. If $F$ has an isolated vertex $v$,
    then since there is at most one edge from $v$ to $G$ in $\widetilde{G}$,
    $v$ must have degree one in $\widetilde{G}$. If $F$ does not have an
    isolated vertex, then every connected component of $F$ has at least two
    vertices of degree one, and since at most one of these vertices can be
    connected to $G$ in $\widetilde{G}$, at least one of these vertices has
    degree one in $\widetilde{G}$. Thus in both cases there is a vertex $v$ of
    $F$ such that $v$ has degree one in $\widetilde{G}$. Let $e \in
    E(\widetilde{G})$ be the edge incident to $v$. In the presentation of
    $\Gamma(\widetilde{G})$, the relation corresponding to $v$ is $x_e = 1$, so
    again $\Gamma(\widetilde{G}) \iso \Gamma(\widetilde{G} \setminus v)$.
    Now $\widetilde{G} \setminus v$ is the result of adding the forest $F
    \setminus v$ to $G$, and since $F \setminus v$ is smaller than $F$, the lemma
    follows by induction.
\end{proof}
Thus for the proofs of Propositions \ref{prop:finite} and \ref{prop:abelian},
we just need to analyze $\Gamma(G)$ for $G$ satisfying one of the conditions
(i)-(iv) from Theorem \ref{thm:lovasz}. The graph incidence group of $K_5$ has
already been determined in Proposition \ref{prop:K33K5}. We consider each other
family of graphs separately in the following subsections.

Before proceeding with the proofs of Propositions \ref{prop:finite} and
\ref{prop:abelian}, we note that Lemmas \ref{lem:cycle} and
\ref{lem:addingforest} give a characterization of when $\Gamma(G)$ is trivial:
\begin{prop}\label{prop:trivial}
    $\Gamma(G)$ is trivial if and only if $G$ is acyclic.
\end{prop}
\begin{proof}
    If $G$ is acyclic, then $G$ is the result of adding a forest to
    the empty graph $G'$. Hence $\Gamma(G) \iso \Gamma(G') = 1$.
    On the other hand, if $G$ contains a cycle $C$, then $C$ is a
    minor of $G$, and hence by Lemma \ref{lem:main_uncoloured} there is a
    surjective homomorphism $\Gamma(G) \arr \Gamma(C) = \Z_2$, so
    $\Gamma(G)$ is nontrivial.
\end{proof}

\subsection{Graphs where every cycle contains a common vertex}\label{subsection:commonvertex}

In this section, we consider the first graph family listed in Theorem
\ref{thm:lovasz}: graphs $G$ for which there is a vertex $v$ contained in all
cycles, or in other words, for which $G \setminus v$ is ayclic. An example
of such a graph is shown in Figure~\ref{fig:tree}.

\begin{figure}[h!]
\begin{center}
    \includegraphics[scale=1.1]{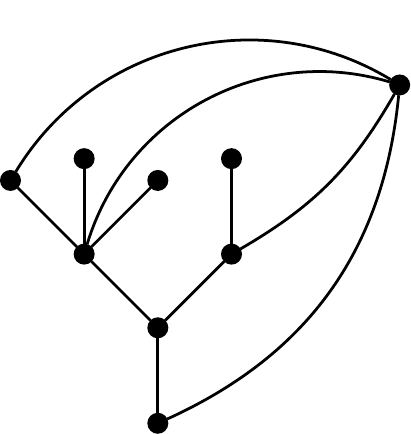}
    \caption{A graph where all the cycles share a common vertex.}
    \label{fig:tree}
\end{center}
\end{figure}

\begin{prop}\label{prop:commonvertex}
    Let $G$ be a graph for which there is a vertex $v$ such that $G \setminus v$
    is acyclic. Then $\Gamma(G)$ is abelian.
\end{prop}
\begin{proof}
    Fix some vertex $v$ such that $G \setminus v$ is acyclic. Note that $G
    \setminus v$ is simple, i.e.~there is at most one edge between every
    pair of vertices. Let $K$ be the subgroup of $\Gamma(G)$ generated by
    $x_f$ for $f \in E(v)$, so $K$ is abelian.  We claim that $x_e \in
    K$ for all $e \in E(G \setminus v)$, so that $\Gamma(G) = K$.

    To prove this claim, suppose $T$ is a connected component of $G \setminus v$. Pick some
    vertex $w_0$ of $T$ arbitrarily, and regard $T$ as a rooted tree with root
    $w_0$. Every vertex $w \neq w_0$ of $T$ has a unique path to $w_0$.
    Following the usual conventions for rooted trees, the vertex adjacent to
    $w$ in this path is called the parent of $w$, and the other vertices of $T$
    adjacent to $w$ are called the descendants of $w$. Suppose $w$ is some
    non-root vertex, and $e$ is the edge connecting $w$ to its parent. Let
    $D \subset E(w)$ be the edges connecting $w$ to its descendants. There might
    also be edges between $w$ and $v$ in $G$, so in $\Gamma(G)$ the defining
    relation corresponding to vertex $w$ states that
    \begin{equation*}
        x_e = \prod_{f \in D} x_f \cdot \prod_{f' \in E(v) \cap E(w)} x_{f'}.
    \end{equation*}
    If we assume that $x_f \in K$ for all $f \in D$, then $x_e \in K$. Thus we
    can use structural induction starting with the leaves of $T$ (the vertices
    without any descendants) to show that $x_e \in K$ for all $e \in E(T)$,
    proving the claim.
\end{proof}
Since $\Gamma(G)$ is finitely generated by elements of order $2$, if
$\Gamma(G)$ is abelian then it is a finite-dimensional $\Z_2$-vector space.
When $G \setminus v$ is acyclic for some vertex $v$, it is not hard to find
the dimension of $\Gamma(G)$. Suppose $e$ is an edge of the forest $G \setminus
v$. Since $G \setminus v$ is simple, the contraction $G / e$ is well-defined.
Let $\phi : \Gamma(G) \arr \Gamma(G / e)$ be the surjective homomorphism
defined in Proposition \ref{prop:minors}, part (ii). By the discussion after
Proposition \ref{prop:minors}, the kernel of this homomorphism is the subgroup
generated by $[x_f,x_g]$ for $f \in E(w_0)$ and $g \in E(w_1)$, where $w_0$
and $w_1$ are the endpoints of $e$. But since $\Gamma(G)$ is abelian, this
subgroup is trivial, so $\phi$ is an isomorphism.

Since $(G / e) \setminus v$ is also acyclic, we can continue contracting edges
until we get a graph $G'$ such that $G' \setminus v$ has no edges. Let
$\Z_2^{E(G')}$ denote the free abelian group generated by the set $\{x_e : e \in
E(G')\}$. Since $\Gamma(G')$ is abelian, $\Gamma(G')$ is the quotient of
$\Z_2^{E(G')}$ by the relations
\begin{equation}\label{eq:cyclescommonvertex}
    \prod_{e \in E(w)} x_e = 1
\end{equation}
for $w \in V(G')$. All edges of $G'$ are incident to $v$, so the sets
\begin{equation*}
    E(w),\ w \in V(G') \setminus \{v\}
\end{equation*}
partition $E(G')$. As a result, the relations \eqref{eq:cyclescommonvertex}
with $w \in V(G') \setminus \{v\}$ involve disjoint sets of variables,
and imply relation \eqref{eq:cyclescommonvertex} for $w = v$.
So $\Gamma(G) \iso \Gamma(G') = \Z_2^{m-k}$,
where $m = |E(G')|$, and $k$ is the number of non-isolated vertices in $V(G')
\setminus \{v\}$. The edges of $E(G')$ come from edges of $G$ incident to $v$,
and the vertices of $V(G') \setminus \{v\}$ correspond to connected components
of $G \setminus v$, so in terms of $G$, $m$ is the number $|E(v)|$ of edges incident
to $v$ in $G$, and $k$ is the number of connected components of $G \setminus v$
which are connected to $v$ in $G$.

\subsection{Wheel graphs with multispokes}

The simple wheel graph $W_n$ is the graph constructed by taking a simple cycle
on $n$ vertices, adding a central vertex (for a total of $n+1$ vertices), and
then adding an edge from each original vertex to the central vertex. The
original $n$ edges of the cycle are called the outer edges, and the $n$ added
edges connecting to the central vertex are called the spokes. The graph $W_8$
is shown on the left in Figure \ref{fig:8_wheel}. Clearly, simple wheels do not
contain two vertex disjoint cycles since each cycle either contains the central
vertex and at least two outer vertices, or is the original cycle containing
all outer vertices.

\begin{proposition}\label{prop:wheel}
    $\Gamma(W_n) = \Z_2^n$.
\end{proposition}

\begin{proof}
    Let $\{e_i : i \in \Z_n\}$ be the set of outer edges of $W_n$, and let
    $\{f_i : i \in \Z_n\}$ be the set of spokes, where we label the two edge
    sets so that $f_i$, $e_{i-1}$, and $e_i$ are incident to a common vertex $v_i$.
    The defining relation of $\Gamma(W_n)$ corresponding to $v_i$ states that
    \begin{equation*}
        x_{f_i} = x_{e_{i-1}} x_{e_i} = x_{e_{i}} x_{e_{i-1}}
    \end{equation*}
    for all $i \in \Z_n$. Because the spokes are all incident to the central vertex,
    \begin{equation*}
        [x_{f_i},x_{f_j}]=1
    \end{equation*}
    for all $i,j \in \Z_n$. It follows that
    \begin{equation}
        x_{e_i}x_{e_j}=x_{f_{i+1}}x_{f_{i+2}}\cdots
            x_{f_j}=x_{f_j}x_{f_{j-1}}\cdots x_{f_{i+1}}=x_{e_j}x_{e_i}
    \end{equation}
    for all $i,j \in \Z_n$. Since
    $\Gamma(W_n)$ is generated by $x_{e_i}$ for $i \in \Z_n$, $\Gamma(W_n)$ is abelian.

    Now $\Gamma(W_n)$ is abelian and generated by $n$ elements of order $2$,
    so $\Gamma(W_n)$ is a $\Z_2$-vector space of dimension at most $n$. Given $i
    \in \Z_n$, let $\overline{e}_i \in \Z_2^n$ denote the vector with $1$ in
    the $(i'+1)$th position and $0$'s in the other positions, where $i'$ is the
    representative of $i$ with $0 \leq i' < n$. Consider the surjective homomorphism
    $\Gamma(W_n) \arr \Z_2^n$ sending $x_{e_i} \mapsto \overline{e}_i$, and
    $x_{f_i} \mapsto \overline{e}_{i-1} + \overline{e}_i$. Since this homomorphism
    sends $x_{f_0} \cdots x_{f_{n-1}}$ to
    \begin{equation*}
        \sum_{i=0}^{n-1} \overline{e}_{i-1} + \overline{e}_i = 0
    \end{equation*}
    in $\Z_2$, this homomorphism is well-defined. So $\Gamma(W_n)$
    must have dimension $n$.
\end{proof}

\begin{figure}[h!]
\begin{center}
    \includegraphics[scale=1.4]{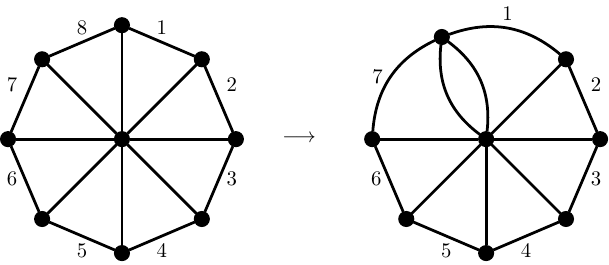}
    \caption{Contracting an outer edge of a simple wheel creates multiple spokes between the center and an outer vertex.}
    \label{fig:8_wheel}
\end{center}
\end{figure}

Condition (ii) from Theorem \ref{thm:lovasz} also allows wheel graphs where the
spokes may be multiedges. These graphs can be constructed by starting with
$W_n$ for some $n$, and adding additional edges between the outer vertices and
the central vertex. Adding these edges adds additional cycles, but all these
cycles still contain the central vertex and at least one outer vertex, so the
resulting graphs still do not contain two disjoint cycles. As shown in Figure
\ref{fig:8_wheel}, if $e$ is an outer edge of a simple wheel graph $W_{n+1}$,
then the contraction $W_{n+1} / e$ can also be obtained by adding a spoke to $W_{n}$.
In general, if $W$ is the result of adding $k$ edges between outer vertices and
the central vertex of $W_{n}$, then $W$ can be obtained by contracting $k$
outer edges of $W_{n+k}$. This can be used to determine $\Gamma(W)$:
\begin{cor}\label{cor:wheel}
    Let $W$ be a wheel graph where the spokes may be multiedges. Then $\Gamma(W) \iso \Z_2^m$,
    where $m$ is the number of spokes of $W$.
\end{cor}
\begin{proof}
    Suppose $W$ is the result of adding $k$ edges to $W_n$, so $W$ has $m=n+k$
    spokes. As discussed above, $W$ is the result of contracting $k$ edges in
    $W_{n+k}$. By Proposition \ref{prop:wheel}, $\Gamma(W_{n+k})$ is abelian.
    As in Subsection \ref{subsection:commonvertex}, if $\Gamma(G)$ is abelian, then the homomorphism
    $\Gamma(G) \arr \Gamma(G/e)$ from Proposition \ref{prop:minors} is an isomorphism.
    So $\Gamma(W) \iso \Gamma(W_{n+k}) \iso \Z_2^{n+k}$.
\end{proof}

\subsection{Complete bipartite graphs $K_{3,n}$}\label{SS:K3n}

In this section, we consider the last family of graphs in Theorem
\ref{thm:lovasz}, the graphs $G$ which can be obtained from $K_{3,n}$ for some
$n \geq 0$ by adding some number of edges (possibly zero) to the first
partition. For this section, we refer to the two partitions of $K_{3,n}$ as the
first and second partition, with the first partition referring to the partition
with $3$ vertices, and the second partition referring to the partition with $n$
vertices.

If $n=0$, then the second partition is empty, and $G$ can be any three vertex
graph. We start by determining $\Gamma(G)$ in this case.
\begin{prop}\label{prop:3vertex}
    Let $G$ be a three vertex graph. Then
    \begin{equation*}
        \Gamma(G) = \begin{cases} \Z_2^{|E(G)|-2} & G \text{ is connected} \\
                    \Z_2^{|E(G)|-1} & G \text{ is disconnected with at least one edge} \\
                    \langle 1 \rangle & |E(G)|=0
                    \end{cases}.
    \end{equation*}
\end{prop}
\begin{proof}
    When $G$ has three vertices, every pair of edges is incident to a common
    vertex, so $\Gamma(G)$ is abelian.  Thus we can think of the defining
    presentation of $\Gamma(G)$ as a linear system over $\Z_2$ with $|E(G)|$
    variables and three defining equations, one for each vertex. If $G$ is
    connected, then the equations have a single linear dependence, so $\Gamma(G) =
    \Z_2^{|E(G)|-2}$. If $G$ is disconnected and $E(G)$ is non-empty, then $G$ has
    a single isolated vertex, and $\Gamma(G) = \Z_2^{|E(G)|-1}$. If $E(G)$ is empty
    (so $G = K_{3,0}$) then $\Gamma(G)$ is the trivial group.  
\end{proof}

Moving on to the case $n \geq 1$, we show that edges added to the first partition
of $K_{3,n}$ end up in the centre of the solution group:
\begin{prop}\label{prop:add_edge}
    Let $G$ be a graph obtained from $K_{3,n}$ for some $n \geq 1$ by adding
    $m$ edges to the first partition. Then $\Gamma(G) \iso \Gamma(K_{3,n})
    \times \Z_2^m$. 
\end{prop}

\begin{proof}
    Suppose $e$ is one of the edges in $E(G) \setminus E(K_{3,n})$, so that the
    endpoints $u$ and $v$ of $e$ belong to the first partition of $K_{3,n}$. 
    Let $w$ be the other vertex in the first partition of $K_{3,n}$. We first
    show that $x_e$ is in the centre of $\Gamma(G)$. By definition, $x_e$
    commutes with $x_f$ for all edges $f$ incident to $u$ or $v$, so we just need
    to show that $x_e$ and $x_f$ commute when $f$ is not incident to $u$ or $v$.
    If $f$ isn't incident to $u$ and $v$, then $f$ must be incident to $w$ and
    another vertex $w'$ in the second partition of $K_{3,n}$. Since $w'$ is
    in the second partition, $w'$ has degree $3$ in $G$. The other two edges
    $f'$ and $f''$ incident to $w'$ are also incident to $u$ and $v$, so $x_e$
    commutes with $x_{f'}$ and $x_{f''}$. But $x_f = x_{f'} x_{f''}$, so
    $x_e$ commutes with $x_f$. Thus $x_e$ is in the centre of $\Gamma(G)$.

    By Proposition \ref{prop:minors}, there is a surjective homomorphism $\phi
    : \Gamma(G) \arr \Gamma(G \setminus e)$ with kernel $\langle x_e \rangle$.
    Since $x_e$ is central, $\Gamma(G)$ is a central extension of $\Gamma(G
    \setminus e)$ by $\langle x_e \rangle$. Pick some vertex $w'$ from the
    second partition of $K_{3,n}$, and write $E(w') = \{f,f',f''\}$, where $f$,
    $f'$, and $f''$ are incident to $w$, $u$, and $v$ respectively. 
    Define a homomorphism
    \begin{equation*}
        \psi : \Gamma(G \setminus e) \arr \Gamma(G) : x_{e'} \mapsto \begin{cases} x_{e'} x_e & e' \in \{f',f''\} \\
                                                                                x_{e'} & \text{ otherwise} \end{cases}.
    \end{equation*}
    Since $x_e$ is central,
    \begin{equation*}
        \psi\left(\prod_{e' \in E(w')} x_{e'}\right) = \psi(x_f x_{f'} x_{f''}) = x_f x_{f'} x_e x_{f''} x_e
            = x_f x_{f'} x_{f''} = 1,
    \end{equation*}
    while
    \begin{equation*}
        \psi\left(\prod_{e' \in E(u) \setminus \{e\}} x_{e'}\right) = \psi(x_{f'}) \cdot \psi\left( \prod_{e' \in E(u)\setminus\{e,f'\}}
            x_{e'}\right) = \prod_{e' \in E(u)} x_{e'} = 1,
    \end{equation*}
    and the other defining relations of $\Gamma(G\setminus e)$ can be checked
    similarly. So $\psi$ is well-defined. Using the formula for $\phi$ from
    Proposition \ref{prop:minors}, we see that $\phi \circ \psi =
    \Id_{\Gamma(G\setminus e)}$, so $\Gamma(G)$ is a split central extension of
    $\Gamma(G\setminus e)$, and hence $\Gamma(G) = \Gamma(G \setminus e) \times
    \langle x_e \rangle$.  

    Finally the subgraph $C$ of $G$ with vertices $u$, $v$, and $w'$, and edges
    $e$, $f'$, and $f''$ is a minor of $G$. The surjective homomorphism 
    $\Gamma(G) \arr \Gamma(C)$ from Proposition \ref{prop:minors} given by 
    deleting all other vertices and edges sends $x_e \mapsto x_e$. Since
    $x_e$ is non-trivial in $\Gamma(C) \iso \Z_2$, $x_e$ is non-trivial in
    $\Gamma(G)$, so $\langle x_e \rangle \iso \Z_2$. By successively deleting
    edges, we see that $\Gamma(G) \iso \Gamma(K_{3,n}) \times \Z_2^m$.
\end{proof}

Combining Propositions \ref{prop:3vertex} and \ref{prop:add_edge}, we get the
following corollary:
\begin{cor}\label{cor:K3nreduction}
    Let $G$ be a graph constructed by adding edges to the first partition of
    $K_{3,n}$, for some $n \geq 0$. Then $\Gamma(G)$ is finite (resp. abelian)
    if and only if $\Gamma(K_{3,n})$ is finite (resp. abelian).
\end{cor}
Thus we only need to look at the groups $K_{3,n}$ for $n \geq 0$.  By Corollary
\ref{cor:K33K5}, $\Gamma(K_{3,3})$ is abelian, and since $K_{3,n}$ is a minor
of $K_{3,3}$ for $n \leq 3$, $\Gamma(K_{3,n})$ is also abelian for $n \leq 3$.
To prove Proposition \ref{prop:abelian}, we need to show that $\Gamma(K_{3,4})$
is nonabelian, while for Proposition \ref{prop:finite}, we need to show that
$\Gamma(K_{3,n})$ is finite for $n \leq 5$ and infinite when $n=6$. It's
possible to determine the order of $\Gamma(K_{3,4})$ and $\Gamma(K_{3,5})$, as
well as the order of their abelianizations, using the GAP computer algebra
package. The orders of $\Gamma(K_{3,n})$ and $\Gamma(K_{3,n})^{ab}$ for $3 \leq
n \leq 6$ are shown in Table \ref{tab:orders}. The order of $\Gamma(G)$ for $G=
K_5$ and the wheel graphs $G = W_n$ are included for comparison. It follows
from this table that $\Gamma(K_{3,4})$ is nonabelian.

\begin{table}[t]
\centering
\begin{tabular}{ l  c  c }
$G$ &  $|\Gamma(G)|$ & $|\Gamma(G)^{Ab}|$  \\ \hline
$W_n$ & $2^n$ & $2^n$  \\ 
$K_{3,3}$ & $16$ & $16$   \\ 
$K_{5}$  & $64$ & $64$  \\ 
$K_{3,4}$ & $256$ & $64$ \\ 
$K_{3,5}$ & $8192$ & $256$  \\ 
$K_{3,6}$ & $\infty$ & $1024$  \\
\end{tabular}
\caption{Order of the incidence group and its abelianization for some small
    graphs not containing two vertex disjoint cycles.}
\label{tab:orders}
\end{table}

We still need to show that $\Gamma(K_{3,6})$ is infinite. Before doing this though,
we give a human-readable proof that $\Gamma(K_{3,n})$ is finite and nonabelian
for $n=4,5$. Although this isn't necessary to prove Propositions \ref{prop:finite}
and \ref{prop:abelian}, the proofs suggest that the structure of the groups
$\Gamma(K_{m,n})$ involve some interesting combinatorics. 

To work with the groups $\Gamma(K_{m,n})$, we order the vertices in each
partition, and label the edge from the $i$th vertex in the first partition to
the $j$th vertex in the second partition by $(j-1)m+i$, $1 \leq i \leq m$, $1
\leq j \leq n$. To visualize the presentation of $\Gamma(K_{m,n})$, we can draw
the $n \times m$ matrix with $x_{(j-1)m+i}$ in the $ji$th entry. Then
$\Gamma(K_{m,n})$ is the group generated by the entries of this matrix, such
that every entry of the matrix squares to the identity, any two entries in the
same row or column commute, and the product of entries in any row or column is
the identity. It is also possible to get $\Gamma(K_{m,n},b)$ from this picture,
by adding the central generator $J$, and setting the product of entries in any row
or column to be $J^{b(v)}$ for the corresponding vertex $v$, rather than the
identity. 

\begin{example}\label{Ex:magicsquare}
    Figure \ref{fig:k33_J_picture} shows the edge labelling above for $K_{3,3}$. 
    In this figure, the first vertex partition is $\{x,y,z\}$, and the second
    partition is $\{u,v,w\}$, both ordered as written. 
    Then $\Gamma(K_{3,3})$ is generated by the entries of
    \begin{equation*}
        \begin{tabular}{|c|c|c|}
            \hline 
            $x_1$ & $x_2$ & $x_3$ \\
            \hline 
            $x_4$ & $x_5$ & $x_6$ \\
            \hline 
            $x_7$ & $x_8$ & $x_9$ \\
            \hline 
        \end{tabular}
    \end{equation*}
    subject to the relations $x_i^2=1$ for all $1 \leq i \leq 9$,
    $[x_i,x_j]=1$ for $(i,j)$ equal to one of the pairs $(1,2)$, $(1,3)$, $(2,3)$
    , $(4,5)$, $(4,6)$, $(5,6)$, $(7,8)$, $(7,9)$, $(8,9)$, $(1,4)$, $(1,7)$,
    $(4,7)$, $(2,5)$, $(2,8)$, $(5,8)$, $(3,6)$, $(3,9)$, and $(6,9)$, and 
    \begin{equation*}
        x_1 x_2 x_3 = x_4 x_5 x_6 = x_7 x_8 x_9
            = x_1 x_4 x_7 = x_2 x_5 x_8 = x_3 x_6 x_9 = 1.
    \end{equation*} 
    If $b$ is the vertex labelling of $K_{3,3}$ with $b(x)=b(y)=b(z)=1$
    and $b(u)=b(v)=b(w)=0$, then we get $\Gamma(K_{3,3},b)$ by adding $J$ to the generators
    along with the relations $J^2 = [x_i,J]=1$ for all $1 \leq i \leq 9$, and modifying
    the last group of relations to
    \begin{equation*}
        x_1 x_2 x_3 = x_4 x_5 x_6 = x_7 x_8 x_9 = 1
    \end{equation*}
    and
    \begin{equation*}
        x_1 x_4 x_7 = x_2 x_5 x_8 = x_3 x_6 x_9 = J.
    \end{equation*}
\end{example}
Since $K_{m,n}$ has a large number of edges, this visual representation of
$\Gamma(K_{m,n})$ and $\Gamma(K_{m,n},b)$ is preferable to drawing $K_{m,n}$.
This representation is the reason that $\mcG(K_{3,3},b)$ is called the ``magic square''
game in \cite{Mermin90,Peres91}, and that the games $\mcG(K_{m,n},b)$ are called
``magic rectangle'' games in \cite{AW20,AW22}.

\begin{example}\label{ex:K30K31K32}
    Since $K_{3,0}$ has no edges, $\Gamma(K_{3,0})$ is trivial. For $\Gamma(K_{3,1})$
    and $\Gamma(K_{3,2})$, we can look at the matrices
    \begin{equation*}
        \begin{tabular}{|c|c|c|}
            \hline 
            $x_1$ & $x_2$ & $x_3$ \\
            \hline
        \end{tabular}
        \quad \text{ and } \quad
        \begin{tabular}{|c|c|c|}
            \hline 
            $x_1$ & $x_2$ & $x_3$ \\
            \hline 
            $x_4$ & $x_5$ & $x_6$ \\
            \hline
        \end{tabular}
    \end{equation*}
    respectively. $\Gamma(K_{3,1})$ is generated by $x_1$, $x_2$, and
    $x_3$, but since the product along any column is the identity, all these
    generators are trivial, and hence $\Gamma(K_{3,1})$ is trivial. 
    Similarly, $\Gamma(K_{3,2})$ is generated by $x_1, \ldots, x_6$, but
    since $x_1 = x_4$, $x_2 = x_5$, $x_3 = x_6$, and
    $x_1 x_2 x_3 = 1$, we see that $\Gamma(K_{3,2}) \iso \Z_2 \times \Z_2$.  
\end{example}
To study $\Gamma(K_{m,n})$, it's helpful to look at the
groups where we modify the matrix presentation of $\Gamma(K_{m,n})$ by leaving
out the relations stating that the product across columns is $1$.
\begin{defn}
    For any $m,n \geq 1$, let
    \begin{align*}
        H_{m,n} = \langle y_i, 1 \leq i \leq mn \ : \ 
            & y_i^2 = 1 \text{ for all } 1 \leq i \leq mn \\
            & [y_i,y_j] = 1 \text{ if } i \equiv j \mod m \\
            & [y_{(j-1)m+i}, y_{(j-1)m+k}]=1 \text{ for all } 1 \leq i,k \leq m,
                1 \leq j \leq n \\
            & y_{(j-1)m+1} y_{(j-1)m+2} \cdots y_{jm} = 1 \text{ for all } 1 \leq j \leq n \rangle. \\
    \end{align*}
\end{defn}
In other words, if we fill an $n \times m$ matrix with the indeterminates
$y_1,\ldots, y_{mn}$, reading from left to right and top to bottom, then
$H_{m,n}$ is generated by the entries of this matrix, subject to the relations
stating that all entries square to the identity, that entries in the same row
or column commute, and that the product of entries in any row of the matrix is
the identity. We draw the matrix for $H_{m,n}$ with a dashed bottom line,
to emphasize that the column products are not included. 
\begin{example}
    $H_{3,3}$ is generated by the entries of the matrix
    \begin{equation*}
        \begin{tabular}{|c|c|c|}
            \hline 
            $y_1$ & $y_2$ & $y_3$ \\
            \hline 
            $y_4$ & $y_5$ & $y_6$ \\
            \hline 
            $y_7$ & $y_8$ & $y_9$ \\
            \hdashline 
        \end{tabular}
    \end{equation*}
    subject to the relations $y_i^2=1$ for all $1 \leq i \leq 9$, $[y_i,y_j]=1$
    for $(i,j)$ equal to one of the pairs $(1,2)$, $(1,3)$, $(2,3)$ , $(4,5)$,
    $(4,6)$, $(5,6)$, $(7,8)$, $(7,9)$, $(8,9)$, $(1,4)$, $(1,7)$, $(4,7)$,
    $(2,5)$, $(2,8)$, $(5,8)$, $(3,6)$, $(3,9)$, and $(6,9)$, and 
    \begin{equation*}
        y_1 y_2 y_3 = y_4 y_5 y_6 = y_7 y_8 y_9 = 1.
    \end{equation*}
\end{example}
Something we can see very easily from the matrix presentations of
$\Gamma(K_{m,n})$ and $H_{m,n}$ is that swapping rows and columns of the
matrix gives an automorphism of the group. For instance, for $H_{3,3}$,
swapping the first two columns of the matrix gives an automorphism $H_{3,3}
\arr H_{3,3}$ sending $y_1 \mapsto y_2$, $y_2 \mapsto y_1$, $y_4 \mapsto y_5$,
$y_5 \mapsto y_4$, $y_7 \mapsto y_8$, $y_8 \mapsto y_7$, and $y_i \mapsto y_i$
for $i=3,6,9$. These automorphisms are very useful in analyzing these groups,
as we can see in the following key example:
\begin{example}\label{Ex:Hkey}
    $H_{3,2}$ arises from the matrix
    \begin{equation*}
        \begin{tabular}{|c|c|c|}
            \hline
            $y_1$ & $y_2$ & $y_3$ \\
            \hline
            $y_4$ & $y_5$ & $y_6$ \\
            \hdashline
        \end{tabular}
        \ .
    \end{equation*}
    Unlike $\Gamma(K_{3,2})$, the elements $y_1 y_4$, $y_2 y_5$, and $y_3 y_6$
    are not necessarily equal to the identity, but since elements in the same
    column commute, all three elements have order $2$. The commutator
    \begin{equation*}
        [y_1 y_4, y_2 y_5] = (y_1 y_4 y_5 y_2)^2 = (y_1 y_6 y_3 y_1)^2 = 1,
    \end{equation*}
    since $(y_3 y_6)^2 = 1$. Applying column swap automorphisms, we see that
    \begin{equation*}
        [y_1 y_4, y_3 y_6] = [y_2 y_5, y_3 y_6] = 1
    \end{equation*}
    as well. Consider the element
    \begin{equation*}
        w = (y_1 y_4) (y_2 y_5) (y_3 y_6).
    \end{equation*}
    Clearly $w$ is invariant under the automorphism which swaps the rows of the
    matrix, and since the elements $y_1 y_4$, $y_2 y_5$, and $y_3 y_6$ pairwise
    commute, $w$ is also invariant under column swaps. Now 
    \begin{equation*}
        w = (y_1 y_4) (y_5 y_2) (y_3 y_6) = y_1 (y_4 y_5) (y_2 y_3) y_6
            = y_1 y_6 y_1 y_6,
    \end{equation*}
    and since $y_1$ and $y_6$ commute with $y_3$ and $y_4$, 
    \begin{equation*}
        [w,y_3] = [w, y_4] = 1.
    \end{equation*}
    By applying row and column swaps, we see that $w = y_i y_j y_i y_j$ for
    any pair of generators $y_i$ and $y_j$ not in the same column or row. In particular,
    $[w,y_i] = 1$ for all $i$, so $w$ is central. 

    Consider the presentation of $\Gamma(K_{3,3},b)$ in Example
    \ref{Ex:magicsquare}, where $b$ is the vertex labelling of $K_{3,3}$ where
    all vertices are labelled by $0$ except the vertex corresponding to the bottom row. 
    Since the generators in the first two rows of the $3 \times 3$ matrix for
    $\Gamma(K_{3,3},b)$ satisfy the defining relations of $H_{3,2}$, there is a
    homomorphism $H_{3,2} \arr \Gamma(K_{3,3},b)$ sending $y_i \mapsto x_i$
    for all $1 \leq i \leq 6$. Going in the opposite direction, if we set $J=w$,
    then the entries of the $3 \times 3$ matrix
    \begin{equation*}
        \begin{tabular}{|c|c|c|}
            \hline
            $y_1$ & $y_2$ & $y_3$ \\
            \hline
            $y_4$ & $y_5$ & $y_6$ \\
            \hline
            $y_1 y_4$ & $y_2 y_5$ & $y_3 y_6$ \\
            \hline
        \end{tabular}
    \end{equation*}
    satisfy the defining relations of $\Gamma(K_{3,3},b)$. So we get a homomorphism
    $\Gamma(K_{3,3},b) \arr H_{3,2}$ sending $x_i \mapsto y_i$ for all
    $1 \leq i \leq 6$, $x_7 \mapsto y_1 y_4$, $x_8 \mapsto y_{2} y_5$,
    $x_9 \mapsto y_3 y_6$, and $J \mapsto w$. Clearly this homomorphism is
    an inverse to the homomorphism $H_{3,2} \arr \Gamma(K_{3,3},b)$, so 
    $H_{3,2} \iso \Gamma(K_{3,3},b) \iso D_4 \odot D_4$, where the central
    element of $D_4 \odot D_4$ corresponds to $w$. 
\end{example}
The reason Example \ref{Ex:Hkey} is so important is that in the matrix
presentation of $\Gamma(K_{3,n})$, the entries in any pair of rows satisfy
the defining relations of $H_{3,2}$, and thus there is a surjective homomorphism
from $H_{3,2}$ onto the subgroup generated by the entries of these rows. As a
result, the entries of these rows will satisfy the identities derived in
Example \ref{Ex:Hkey}. For example, if we take the $4 \times 3$ matrix
\begin{equation*}
    \begin{tabular}{|c|c|c|}
        \hline 
        $x_1$ & $x_2$ & $x_3$ \\
        \hline 
        $x_4$ & $x_5$ & $x_6$ \\
        \hline 
        $x_7$ & $x_8$ & $x_9$ \\
        \hline 
        $x_{10}$ & $x_{11}$ & $x_{12}$ \\
        \hline 
    \end{tabular}
\end{equation*}
for $\Gamma(K_{3,4})$, then we see that $x_1 x_6 x_1 x_6$ will
commute with $x_i$ for $1 \leq i \leq 6$. 
\begin{lemma}\label{lem:K34}
    $\Gamma(K_{3,4})$ is finite and nonabelian.
\end{lemma}
\begin{proof}
    As mentioned above $x_1 x_6 x_1 x_6$ commutes with
    all entries in the first two rows. In addition,
    \begin{equation*}
        x_1 x_6 x_1 x_6 = (x_1 x_4) (x_2 x_5)
            (x_3 x_6) = (x_7 x_{10}) (x_8 x_{11}) (x_9 x_{12}),
    \end{equation*}
    where the last identity comes from the fact that the product along any
    column is the identity. So $x_1 x_6 x_1 x_6$ also commutes
    with all entries in the last two rows, and hence is central. Applying row and
    column swap automorphisms, we see that
    all commutators $[x_i,x_j]$ are central. Let $N$ be the 
    central (hence normal) subgroup generated by these commutators. The quotient 
    $\Gamma(K_{3,4})/N$ is abelian, and since $N$ and $\Gamma(K_{3,4}) / N$ are
    both abelian groups finitely generated by elements of order $2$, both
    $N$ and $\Gamma(K_{3,4})/N$ are finite. It follows that $\Gamma(K_{3,4})$
    is finite.

    To see that $\Gamma(K_{3,4})$ is nonabelian, note that if $y_1,\ldots,y_6$ are
    the generators of $H_{3,2}$, then the entries of the matrix
    \begin{equation*}
        \begin{tabular}{|c|c|c|}
            \hline
            $y_1$ & $y_2$ & $y_3$ \\
            \hline
            $y_4$ & $y_5$ & $y_6$ \\
            \hline
            $y_1$ & $y_2$ & $y_3$ \\
            \hline
            $y_4$ & $y_5$ & $y_6$ \\
            \hline
        \end{tabular}
    \end{equation*}
    satisfy the defining relations for $\Gamma(K_{3,4})$. Thus there is a surjective
    homomorphism $\Gamma(K_{3,4}) \arr H_{3,2}$ sending
    \begin{equation*}
        x_i \mapsto \begin{cases} y_i & 1 \leq i \leq 6 \\
                                        y_{i-6} & 7 \leq i \leq 12
                        \end{cases}.
    \end{equation*}
    Since $H_{3,2}$ is nonabelian, so is $\Gamma(K_{3,4})$. 
\end{proof}

Since $K_{3,4}$ is a minor of $K_{3,n}$ for $n \geq 4$, Lemmas
\ref{lem:main_uncoloured} and \ref{lem:K34} imply that $\Gamma(K_{3,n})$ is
nonabelian for $n \geq 4$. We still need to prove:
\begin{lemma}\label{lem:K35}
    $\Gamma(K_{3,5})$ is finite.
\end{lemma}
\begin{proof}
    The proof is similar to the proof that $\Gamma(K_{3,4})$ is finite, but
    more involved. Consider the presentation of $\Gamma(K_{3,5})$ in terms
    of the matrix
    \begin{equation*}
        \begin{tabular}{|c|c|c|}
            \hline 
            $x_1$ & $x_2$ & $x_3$ \\
            \hline 
            $x_4$ & $x_5$ & $x_6$ \\
            \hline 
            $x_7$ & $x_8$ & $x_9$ \\
            \hline 
            $x_{10}$ & $x_{11}$ & $x_{12}$ \\
            \hline 
            $x_{13}$ & $x_{14}$ & $x_{15}$ \\
            \hline 
        \end{tabular}
        \ ,
    \end{equation*}
    and let
    \begin{equation*}
        w = (x_{10} x_{13}) (x_{11} x_{14}) (x_{12} x_{15})
          = (x_{1} x_{4} x_7) (x_{2} x_{5} x_8) (x_{3} x_{6} x_9).
    \end{equation*}
    From the first expression, we see that $w$ is invariant under swapping
    columns, while from the second expression, we see that $w$ is invariant
    under swapping any of the first three rows. Now from the second expression, we get
    \begin{align*}
        w & = x_1 (x_4 x_7 x_5 x_8) (x_2 x_3) x_6 x_9 \\
            & = x_1 (x_4 x_7 x_5 x_8 x_6 x_9) (x_6 x_9) x_1 (x_4 x_5) x_9 \\
            & = x_1 (x_5 x_9 x_5 x_9) (x_6 x_9 x_4 (x_1 x_5 x_9) \\
            & = (x_1 x_5 x_9) (x_5 x_6 x_4) x_1 x_5 x_9 \\
            & = x_1 x_5 x_9 x_1 x_5 x_9, \\
    \end{align*}
    where for the third identity, we use Example \ref{Ex:Hkey} to conclude $x_4 x_7 x_5 x_8 x_6 x_9
    = x_5 x_9 x_5 x_9$.  But since $w$
    is invariant under swapping the first and second row, and also swapping the
    first and second column, we get
    \begin{equation*}
        x_1 x_5 x_9 x_1 x_5 x_9 = w = x_5 x_1 x_9 x_5 x_1 x_9.
    \end{equation*}
    Cancelling $x_9$ on the right and rearranging the remaining terms, we see that 
    \begin{equation*}
        x_1 x_5 x_1 x_5 x_9 = x_9 x_5 x_1 x_5 x_1 = x_9 x_1 x_5 x_1 x_5,
    \end{equation*}
    where we use Example \ref{Ex:Hkey} again for the identity $x_5 x_1 x_5 x_1 = x_1 x_5 x_1 x_5$. Thus $u =
    x_1 x_5 x_1 x_5$ commutes with $x_9$. By permuting rows
    and columns, we see that $u$ commutes with $x_i$ for $7 \leq i \leq 15$.
    We also know that $u$ commutes with $x_i$ for $1 \leq i \leq 6$, so 
    $u$ is central. Ultimately we conclude from symmetry that commutators
    of the form $[x_i,x_j]$ are central for all $1 \leq i,j \leq 15$. 
    The rest of the proof is as in Lemma \ref{lem:K34}. 
\end{proof}

We finish by showing that $\Gamma(K_{3,6})$ is infinite. Recall (from, e.g.
\cite{BN98}) that a \emph{rewriting system over a finite set $S$} is a finite
subset $W \subseteq S^* \times S^*$, where $S^*$ is the set of words over $S$.
We write $a \to_W b$ for words $a,b \in S^*$ if there are words $c,d, \ell, r
\in S^*$ such that $a=c\ell d$, $b = c r d$, and $(\ell,r) \in W$, and $a
\to_W^* b$ if there is a sequence $a = a_0, a_1, \ldots, a_k = b \in S^*$ with
$a_{i-1} \to_W a_i$ for all $1 \leq i \leq k$. We say that $\ell \in S^*$ is a
subword of $a \in S^*$ if $a = b \ell c$ for some $b,c \in S^*$. A word $a \in
S^*$ is a \emph{normal form with respect to $W$} if $\ell$ is not a subword of
$a$ for all pairs $(\ell,r) \in W$. If $a \to_W^* b$ and $b$ is a normal form,
then $b$ is said to be a \emph{normal form of $a$}. A rewriting system $W$ is
\emph{terminating} if there is no infinite sequence $a_1 \to_W a_2 \to_W a_3
\to_W \cdots$, \emph{confluent} if for every $a,b,c \in S^*$ with $a \to_W^* b$
and $a \to_W^* c$, there is $d \in S^*$ such that $b \to_W^* d$ and $c \to_W^*
d$, and \emph{locally confluent} if for every $a,b,c \in S^*$ with $a \to_W b$
and $a \to_W c$, there is $d \in S^*$ such that $b \to_W^* d$ and $c \to_W^*
d$. If $W$ is terminating, then every word $a$ has a normal form with respect
to $W$.  If $W$ is terminating and confluent, then every word has a unique
normal form with respect to $W$.  Newman's lemma states that if $W$ is
terminating and locally confluent, then $W$ is confluent. 

For a finite set $S$, an order $\leq$ on $S^*$ is a \emph{reduction order} if
it is a well-order (i.e.~every subset of $S^*$ has a least element) and $a \leq
b$ implies $cad \leq cbd$ for all $a,b,c,d \in S^*$. If $\leq$ is a reduction
order, and $W$ is a rewriting system such that $\ell > r$ for all $(\ell,r) \in
W$, then $W$ is terminating. Suppose $S = \{s_1,\ldots,s_n\}$, and let $S_k =
\{s_1,\ldots,s_k\}$. Define an order $\leq_k$ on $S_k$ inductively as follows:
Let $a <_1 b$ for $a,b \in S_1^*$ if and only if $b$ is longer than $a$.  For
$k > 1$, let $a <_k b$ if either
\begin{itemize}
    \item $s_k$ appears more often in $b$ than in $a$, or
    \item $s_k$ appears $m$ times in both $a$ and $b$, and when we write
        $a = a_0 s_1 a_1 s_1 \cdots s_1 a_m$, $b = b_0 s_1 b_1 s_1
        \cdots s_1 b_m$ for words $a_0,\ldots,a_m$, $b_0,\ldots,b_m \in S_{k-1}^*$,
        there is an index $0 \leq i \leq m$ such that $a_{j} = b_j$ for $j < i$
        and $a_i <_{k-1} b_i$.
\end{itemize}
The resulting order $\leq\ :=\ \leq_k$ on $S$ is called the \emph{wreath product
order for the sequence $s_1,\ldots,s_n$}. Wreath product orders are reduction orders.

Returning to groups, a \emph{complete rewriting system} for a finitely
presented group $\langle S : R \rangle$ is a terminating and confluent
rewriting system over $S \cup S^{-1}$, such that the empty word $1$ is the
normal form of all $r \in R$, and is also the normal form of $s s^{-1}$ and
$s^{-1} s$ for all $s \in S$.  It is well-known that if $W$ is a complete
rewriting system for $\langle S : R \rangle$, then the elements of $\langle S :
R \rangle$ are in bijection with the set of normal forms with respect to
$W$.\footnote{For the reader interested in proving this themselves, note that
if $r_1 r_2^{-1} \in R$, then $r_1 r_2^{-1} r_2 \to_W^* r_2$ and $r_1 r_2^{-1}
r_2 \to_W^* r_1$, so $r_1$ and $r_2$ must have the same normal forms.} If every
element of $S$ has order two, then it is more convenient to work with rewriting
systems over $S$ rather than $S \cup S^{-1}$. Specifically, if $R$ is a set of
words over a set $S$, then we say that a rewriting system $W$ over $S$ is a
complete rewriting system for the group
\begin{equation*}
    \Gamma = \langle S : R \cup \{s^2 : s \in S\} \rangle
\end{equation*}
if $W$ is terminating and confluent, $r$ has normal form $1$ for all relations
$r \in R$, and $s^2$ has normal form $1$ for all $s \in S$. If $W$ is such a
rewriting system, then once again the elements of $\Gamma$ are in bijection
with the normal forms with respect to $W$. 

The Knuth-Bendix algorithm is a procedure for constructing a complete rewriting
system for a finitely presented group, given the group presentation and a
reduction order as input. It is not guaranteed to halt, and the success and
running time of the procedure is often highly dependent on the specified order.
Using the implementation of Knuth-Bendix in the KBMAG package \cite{KBMAG1995},
we were able to find a complete rewriting system for the group $H_{3,3}$. This
rewriting system is shown in Figure \ref{fig:h33_rws}. Finding this rewriting
system involved a lengthy automated search through reduction orders until we 
found one for which the Knuth-Bendix procedure would halt. However, it is much
easier to verify that this rewriting system is complete once we've found it:

\begin{figure}[tp]
    \begin{center}
\begin{tabular}{ l }
$y_4y_4 \longrightarrow 1$\\
$y_7y_7 \longrightarrow 1$\\
$y_2y_2 \longrightarrow 1$\\
$y_8y_8 \longrightarrow 1$\\
$y_3y_3 \longrightarrow 1$\\
$y_6y_6 \longrightarrow 1$\\
$y_4y_7 \longrightarrow y_7y_4$\\
$y_2y_8 \longrightarrow y_8y_2$\\
$y_6y_3 \longrightarrow y_3y_6$\\
$y_2y_3 \longrightarrow y_3y_2$\\
$y_4y_6 \longrightarrow y_6y_4$\\
$y_7y_8 \longrightarrow y_8y_7$\\
$y_9 \longrightarrow y_8y_7$\\
$y_1 \longrightarrow y_3y_2$\\
$y_5 \longrightarrow y_6y_4$\\
$y_8y_3 \longrightarrow y_3y_8y_7y_2y_7y_2$\\
$y_6y_8 \longrightarrow y_8y_7y_6y_7$\\
$y_4y_3 \longrightarrow y_3y_2y_4y_2$\\
$y_7y_3 \longrightarrow y_3y_2y_7y_2$\\
$y_6y_4y_2 \longrightarrow y_2y_6y_4$\\
$y_6y_2 \longrightarrow y_4y_2y_6y_4$\\
$y_4y_8 \longrightarrow y_8y_7y_6y_7y_6y_4$\\
$y_6y_7y_6y_7 \longrightarrow y_7y_6y_7y_6$\\
$y_2y_7y_2y_7 \longrightarrow y_7y_2y_7y_2$\\
$y_4y_2y_4y_2 \longrightarrow y_2y_4y_2y_4$\\
$y_2y_7y_4y_2y_6y_7 \longrightarrow y_6y_7y_2y_6y_7y_2y_6y_4$\\
$y_2y_6y_7y_2y_6y_7 \longrightarrow y_7y_4y_2y_6y_7y_2y_6y_4$\\
$y_2y_7y_6y_7y_2y_6y_7 \longrightarrow y_7y_2y_7y_2y_4y_2y_6y_7y_2y_6y_4$\\
$y_2y_4y_2y_7y_2y_6y_7 \longrightarrow y_4y_2y_6y_7y_4y_2y_6y_7y_2y_6y_4$\\
$y_2y_4y_2y_7y_6y_7y_4y_2y_6y_7 \longrightarrow y_4y_2y_7y_2y_6y_7y_2y_4y_2y_6y_7y_4y_2y_4$\\
$y_2y_4y_2y_6y_7y_4y_2y_6y_7 \longrightarrow y_4y_2y_7y_2y_6y_7y_2y_6y_4$\\
$y_4y_2y_6y_7y_2y_7y_2y_6y_7 \longrightarrow y_2y_7y_2y_6y_7y_2y_7y_6y_4$\\
$y_2y_4y_2y_7y_4y_2y_7 \longrightarrow y_4y_2y_7y_4y_2y_7y_2$\\
$y_2y_7y_4y_2y_7y_6y_7 \longrightarrow y_6y_7y_2y_6y_7y_2y_7y_6y_4$\\
$y_2y_6y_7y_2y_7y_6y_7 \longrightarrow y_7y_4y_2y_6y_7y_2y_7y_6y_4$\\
$y_2y_4y_2y_6y_7y_2y_7 \longrightarrow y_4y_2y_6y_7y_2y_7y_2$\\
$y_2y_6y_7y_4y_2y_7 \longrightarrow y_6y_7y_4y_2y_7y_2$\\
$y_2y_7y_6y_7y_2y_7 \longrightarrow y_6y_7y_2y_6y_7y_2y_4y_2y_6y_7y_4y_2$\\
$y_4y_2y_6y_7y_2y_4y_2y_6y_7 \longrightarrow y_2y_7y_2y_4y_2y_6y_7y_2y_4y_2y_6$\\
$y_2y_4y_2y_7y_6y_7y_4y_2y_7 \longrightarrow y_4y_2y_6y_7y_2y_7y_2y_4y_2y_6y_7y_2y_6$\\
$y_4y_2y_7y_2y_6y_7y_2y_7 \longrightarrow y_2y_6y_7y_2y_7y_2y_6y_7y_6y_4$\\
$y_4y_2y_7y_4y_2y_7y_2y_6y_7 \longrightarrow y_2y_6y_7y_4y_2y_6y_7y_2y_7y_6y_4$\\
$y_4y_2y_6y_7y_4y_2y_6y_7y_2y_7 \longrightarrow y_2y_7y_4y_2y_7y_2y_6y_7y_6y_4$\\
$y_4y_2y_7y_2y_4y_2y_6y_7 \longrightarrow y_2y_6y_7y_2y_4y_2y_6y_7y_2y_4y_2y_6$\\
$y_4y_2y_6y_7y_2y_4y_2y_7y_6y_7 \longrightarrow y_2y_7y_2y_4y_2y_6y_7y_2y_4y_2y_7y_6$\\
$y_4y_2y_7y_2y_4y_2y_7y_6y_7 \longrightarrow y_2y_6y_7y_2y_4y_2y_6y_7y_2y_4y_2y_7y_6$
\end{tabular}
\end{center}

    \caption{A complete rewriting system for the group $H_{3,3}$.}
    \label{fig:h33_rws}
\end{figure}

\begin{lemma}\label{lem:H33confluent}
    The rewriting system $W$ in Figure \ref{fig:h33_rws} is a complete rewriting
    system for $H_{3,3}$.
\end{lemma}
\begin{proof}
    Let $S = \{y_1,\ldots,y_9\}$. 
    To show that $W$ is terminating, we can check either by hand or on a
    computer that if $(\ell, r) \in W$, then $\ell > r$ in the wreath product
    ordering for the sequence $O = (y_4, y_6, y_2, y_7, y_5, y_8, y_3, y_9, y_1)$.
    For example, $y_1 > y_3 y_2$ because $y_1$ occurs after $y_2$ and $y_3$ in
    the sequence $O$, and $y_1$ appears fewer times in $y_3 y_2$ than in $y_1$. In another
    example, consider the pair 
    \begin{equation*} 
        (\ell, r) = (y_4y_2y_7y_4y_2y_7y_2y_6y_7, y_2y_6y_7y_4y_2y_6y_7y_2y_7y_6y_4).
    \end{equation*}
    The generator $y_7$ occurs three times in both $\ell$ and $r$, so to
    compare these two we look at the words $y_4 y_2$ and $y_2 y_6$. Since $y_4 >
    1$, we see that $y_4 y_2 > y_2 y_6$, 
    and hence $\ell > r$. The other pairs can be checked similarly. Since wreath
    product orders are reduction orders, $W$ is terminating.

    For local confluence, it suffices to check two conditions:
    \begin{enumerate}[(i)]
        \item If $(ab,r_1), (bc,r_2) \in W$ for $a,b,c,r_1,r_2 \in S^*$, so that $abc \to_W a r_2$ and
            $abc \to_W r_1 c$, then there is a word $d$ such that $a r_2 \to_W^*
            d$ and $r_1 c \to_W^* d$.
        \item If $(abc,r_1), (b,r_2) \in W$, so that $abc \to_W r_1$ and $abc
            \to_W a r_2 c$, then there is a word $d$ such that $r_1 \to_W^* d$
            and $a r_2 c \to_W^* d$.
    \end{enumerate}
    This is time-consuming to check by hand for $W$, but can easily be checked
    on a computer. 

    Since $W$ is both terminating and locally confluent, it is confluent. Clearly
    $y_i^2 \to_W 1$ for all $1 \leq i \leq 9$. We can check either by hand or on
    a computer that $r \to_W^* 1$ for all defining relations $r$ of $H_{3,3}$, so
    $W$ is a complete rewriting system.
\end{proof}
Short computer programs for performing the calculations in Lemma
\ref{lem:H33confluent} can be found at \cite{PRSS22}.
The rewriting system from Figure \ref{fig:h33_rws} can be used to show:
\begin{lemma}\label{lem:k36}
    $H_{3,3}$ is infinite. 
\end{lemma}
\begin{proof}
    By inspection of Figure \ref{fig:h33_rws}, we see that $(y_4 y_2 y_7)^n$ is
    a normal form for all $n \geq 1$. By Lemma \ref{lem:H33confluent}, these
    elements are all distinct, so $H_{3,3}$ is infinite.
\end{proof}

\begin{prop}\label{prop:K36infinite}
    $\Gamma(K_{3,6})$ is infinite.
\end{prop}
\begin{proof}
    If $y_1,\ldots,y_9$ are
    the generators of $H_{3,3}$, then the entries of the matrix
    \begin{equation*}
        \begin{tabular}{|c|c|c|}
            \hline
            $y_1$ & $y_2$ & $y_3$ \\
            \hline
            $y_4$ & $y_5$ & $y_6$ \\
            \hline
            $y_7$ & $y_8$ & $y_9$ \\
            \hline
            $y_1$ & $y_2$ & $y_3$ \\
            \hline
            $y_4$ & $y_5$ & $y_6$ \\
            \hline
            $y_7$ & $y_8$ & $y_9$ \\
            \hline
        \end{tabular}
    \end{equation*}
    satisfy the defining relations of $\Gamma(K_{3,6})$. Thus there is a surjective
    homomorphism $\Gamma(K_{3,6}) \arr H_{3,3}$ sending
    \begin{equation*}
        x_i \mapsto \begin{cases} y_i & 1 \leq i \leq 9 \\
                                        y_{i-9} & 10 \leq i \leq 18 
                        \end{cases}.
    \end{equation*}
    Since $H_{3,3}$ is infinite, so is $\Gamma(K_{3,6})$. 
\end{proof}

\subsection{Proofs of characterizations for abelianness and finiteness}

We can now finish the proofs of Propositions \ref{prop:finite} and \ref{prop:abelian}.
\begin{proof}[Proof of Proposition~\ref{prop:finite}]
    If $G$ contains $C_2 \sqcup C_2$ or $K_{3,6}$ as a minor, then $\Gamma(G)$
    is infinite by Lemma \ref{lem:main_uncoloured}, Corollary \ref{cor:2cyc},
    and Proposition \ref{prop:K36infinite}. 

    Suppose $G$ avoids $C_2 \sqcup C_2$ and $K_{3,6}$. Then
    $G$ can be obtained from a graph $G'$ satisfying conditions (i)-(iv) from
    Theorem \ref{thm:lovasz}, by taking a subdivision of $G'$ and adding a
    forest.  Since $G'$ is a minor of $G$, $G'$ also avoids $K_{3,6}$. 
    By Proposition \ref{prop:commonvertex}, Corollary \ref{cor:wheel} and Corollary \ref{cor:K33K5},
    if $G'$ satisfies one of conditions (i)-(iii) from Theorem \ref{thm:lovasz},
    then $\Gamma(G')$ is abelian (and hence finite). Suppose $G'$ satisfies condition (iv) from
    Theorem \ref{thm:lovasz}, so there is $n \geq 0$ such that $G'$ can be obtained from $K_{3,n}$
    by adding edges to the first partition. Since $G'$ does not contain $K_{3,6}$,
    we must have $n < 6$, so $\Gamma(K_{3,n})$ is finite by 
    Example \ref{ex:K30K31K32}, Corollary \ref{cor:K33K5}, and Lemmas
    \ref{lem:K34} and \ref{lem:K35}.  By Corollary \ref{cor:K3nreduction},
    $\Gamma(G')$ is also finite. Since $\Gamma(G) \iso \Gamma(G')$ by Lemmas
    \ref{lem:subdivision} and \ref{lem:addingforest}, we conclude in all four
    cases that $\Gamma(G)$ is finite.
\end{proof}

The proof of Proposition \ref{prop:abelian} is very similar:
\begin{proof}[Proof of Proposition \ref{prop:abelian}]
    If $G$ contains $C_2 \sqcup C_2$ or $K_{3,4}$ as a minor, then $\Gamma(G)$
    is nonabelian by Lemma \ref{lem:main_uncoloured}, Corollary \ref{cor:2cyc},
    and Proposition \ref{lem:K34}. 

    Suppose $G$ avoids $C_2 \sqcup C_2$ and $K_{3,4}$. Then
    $G$ can be obtained from a graph $G'$ satisfying conditions (i)-(iv) from
    Theorem \ref{thm:lovasz}, by taking a subdivision of $G'$ and adding a
    forest.  Since $G'$ is a minor of $G$, $G'$ also avoids $K_{3,4}$. 
    By Proposition \ref{prop:commonvertex}, Corollary \ref{cor:wheel} and Corollary \ref{cor:K33K5},
    if $G'$ satisfies one of conditions (i)-(iii) from Theorem \ref{thm:lovasz},
    then $\Gamma(G')$ is abelian. Suppose $G'$ satisfies condition (iv) from
    Theorem \ref{thm:lovasz}, so there is $n \geq 0$ such that $G'$ can be obtained from $K_{3,n}$
    by adding edges to the first partition. Since $G'$ does not contain $K_{3,4}$,
    we must have $n < 4$, so $\Gamma(K_{3,n})$ is abelian by Example \ref{ex:K30K31K32} and
    Corollary \ref{cor:K33K5}. By Corollary \ref{cor:K3nreduction}, $\Gamma(G')$ is also abelian.
    Since $\Gamma(G) \iso \Gamma(G')$ by Lemmas \ref{lem:subdivision} and
    \ref{lem:addingforest}, we conclude in all four cases that $\Gamma(G)$
    is abelian.
\end{proof}

With these propositions, we can prove Theorems \ref{thm:finite} and \ref{thm:abelian}.
\begin{proof}[Proof of Theorem~\ref{thm:finite}]
    If $b$ is a $\Z_2$-colouring of $G$, then $\Gamma(G,b)$ is finite if and only if $\Gamma(G)$
    is finite. So the theorem follows immediately from Proposition \ref{prop:finite}.
\end{proof}

\begin{proof}[Proof of Theorem~\ref{thm:abelian}]
    Let $(G,b)$ be a connected $\Z_2$-coloured graph, and let
    \begin{align*}
        \mcF = 
            & \{ (K_{3,3},b') : b' \text{ odd parity}\} \cup 
            \{ (K_{5},b') : b' \text{ odd parity}\} \\ \cup 
            & \{ (K_{3,4},b') : b' \text{ even parity}\} \cup 
            \{ (C_2 \sqcup C_2,b') : b' \text{ any parity}\}.
    \end{align*}
    By Lemma \ref{lem:minor_colour}, $(G,b)$ avoids $\mcF$ if and only if either
    $b$ has even parity and $G$ avoids $K_{3,4}$ and $C_2 \sqcup C_2$, or $b$
    has odd parity and $G$ avoids $K_{3,3}$, $K_5$, and $C_2 \sqcup C_2$.

    If $b$ has even parity, then $\Gamma(G,b) \iso
    \Gamma(G,0) \iso \Gamma(G) \times \Z_2$ by Lemma \ref{lem:parity}. So
    $\Gamma(G,b)$ is abelian if and only if $\Gamma(G)$ is abelian. By
    Proposition \ref{prop:abelian}, this occurs if and only if $G$ avoids
    $K_{3,4}$ and $C_2 \sqcup C_2$.

    Suppose $b$ has odd parity. The groups $\Gamma(K_{3,3},b')$ and $\Gamma(K_5,b')$
    are nonabelian when $b'$ is odd by Proposition \ref{prop:K33K5}. Since 
    $\Gamma( C_2 \sqcup C_2, b') / \langle J \rangle = \Gamma( C_2 \sqcup C_2)$ is
    nonabelian, $\Gamma(C_2 \sqcup C_2,b')$ is nonabelian for any $b'$. So if
    $G$ contains $K_{3,3}$, $K_5$, or $C_2 \sqcup C_2$, then $\Gamma(G,b)$ is
    nonabelian by Lemma \ref{lem:main}. Suppose $G$ avoids $K_{3,3}$, $K_5$, and
    $C_2 \sqcup C_2$. Then $G$ is planar, so $J=1$ in $\Gamma(G,b)$ by Theorem 
    \ref{thm:Arkhipov2}. Hence $\Gamma(G,b) \iso \Gamma(G,b) / \langle J \rangle
    = \Gamma(G)$. Since $G$ avoids $K_{3,3}$, it also avoids $K_{3,4}$, and hence
    $\Gamma(G,b)$ is abelian by Proposition \ref{prop:abelian}.
    
    We conclude that $\Gamma(G,b)$ is abelian if and only if $(G,b)$ avoids $\mcF$.
\end{proof}

\section{Open problems}\label{sec:open}

In Theorems \ref{thm:finite} and \ref{thm:abelian}, we characterize when
$\Gamma(G,b)$ is finite or abelian.  As mentioned in the introduction, it is
also interesting to ask for the forbidden minors for other quotient closed
properties. Since it's connected with group stability and finite-dimensional
approximations of groups (and hence with near-perfect strategies for games),
amenability is a particularly interesting property to ask about:
\begin{problem}
    Find the forbidden minors for amenability of $\Gamma(G,b)$ and
    $\Gamma(G)$.
\end{problem}
Since amenability is closed under extensions, and $\Z_2$ and $\Z$ are both
amenable, the group $\Z_2 \ast \Z_2 \iso \Z_2 \ltimes \Z$ is amenable. So
Lovasz's characterization of graphs that avoid two disjoint cycles does not
help with this problem.  However, $\Z_2 \ast \Z_2 \ast \Z_2 = \Gamma(C_2 \sqcup
C_2 \sqcup C_2)$ is not amenable, since it contains $\Z \ast \Z$ as a subgroup.
Thus a starting point for this problem might be to look at graphs that avoid
three disjoint cycles.  We note that, like planarity testing, deciding whether
a graph contains $k$-disjoint cycles can be done in linear time in the size of
the graph \cite{Bod94}. 

Another property that comes up in the study of group stability is property (T). 
\begin{problem}
    Find the forbidden graph minors characterizing property (T) for
    $\Gamma(G,b)$ and $\Gamma(G)$.
\end{problem}

The only groups which are both amenable and have property (T) are the finite
groups, so $\Gamma(C_2 \sqcup C_2) = \Z_2 \ast \Z_2$ does not have property
(T).  Hence if $\Gamma(G)$ has property (T), then $G$ does not contain two
disjoint cycles. The groups (i)-(iii) in Theorem \ref{thm:lovasz} are all finite
and hence have property (T). However, we do not know whether $\Gamma(K_{3,6})$
has property (T). If it does not, then $\Gamma(G,b)$ and $\Gamma(G)$ would have
property (T) if and only if they are finite. 

As mentioned in the introduction, while testing whether $J=1$ is easy for graph
incidence groups, it is undecidable for solution groups. It would be
interesting to know whether the word problem for graph incidence groups is
decidable in general. If a group has a complete rewriting system, then its word
problem is decidable, so it would be also interesting to know:
\begin{problem}
    Is there a graph incidence group which does not have a complete rewriting system.
\end{problem}
Doing some initial computer exploration with the KBMAG package for the GAP
computer algebra system, we were able to find complete rewriting systems for
the graph incidence groups $\Gamma(G)$ of all 30 cubic graphs on at most 10
vertices, with one exception: the Petersen graph, shown in Figure
\ref{fig:pet_graph}.
\begin{figure}[!htpb]
	\begin{center}
		\includegraphics[scale=0.8]{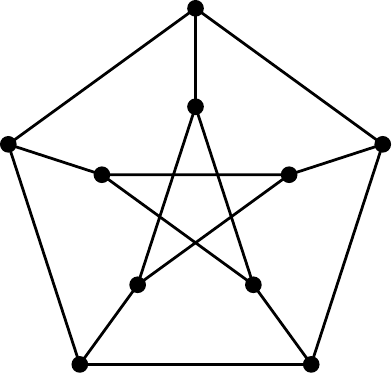}
	\end{center}
		\caption{We were not able to find a complete rewriting system for the Petersen graph.}
\label{fig:pet_graph}
\end{figure}
We were also able to find complete rewriting systems for the Petersen graph
with an edge contracted or deleted. In all these cases, the Knuth-Bendix
algorithm in KBMAG finished within a few seconds, and returned rewriting
systems with less than 50 rules. Thus the Petersen graph might be a good
candidate for a graph that does not have a complete rewriting system.  Having a
decidable word problem or a complete rewriting system is not a quotient
property, so we do not expect these properties to be characterizable by
forbidden minors.

\bibliographystyle{alpha}
\bibliography{bibfile}

\begin{thebibliography}{CMMN20}

\bibitem[AR90]{AR90}
Dan Archdeacon and Bruce Richter.
\newblock On the parity of planar covers.
\newblock {\em Journal of {G}raph {T}heory}, 14(2):199--204, 1990.

\bibitem[Ark12]{Ark12}
Alex Arkhipov.
\newblock Extending and {C}haracterizing {Q}uantum {M}agic {G}ames.
\newblock {\em arXiv:1209.3819}, 2012.

\bibitem[AW20]{AW20}
Sean~A. Adamson and Petros Wallden.
\newblock Quantum magic rectangles: {C}haracterization and application to
  certified randomness expansion.
\newblock {\em Physical Review Research}, 2:043317, 2020.

\bibitem[AW22]{AW22}
Sean~A. Adamson and Petros Wallden.
\newblock Practical parallel self-testing of {B}ell states via magic
  rectangles.
\newblock {\em Physical Review A}, 105:032456, 2022.

\bibitem[BN98]{BN98}
Franz Baader and Tobias Nipkow.
\newblock {\em Term Rewriting and All That}.
\newblock Cambridge University Press, 1998.

\bibitem[BNW10]{BNW10}
Olivier Bernardi, Marc Noy, and Dominic Welsh.
\newblock Growth constants of minor-closed classes of graphs.
\newblock {\em Journal of Combinatorial Theory, Series B}, 100(5):468--484,
  2010.

\bibitem[Bod94]{Bod94}
Hans~L. Bodlaender.
\newblock On disjoint cycles.
\newblock {\em International Journal of Foundations of Computer Science},
  5(01):59--68, 1994.

\bibitem[Bol04]{Bol04}
B{\'e}la Bollob{\'a}s.
\newblock {\em Extremal graph theory}.
\newblock Courier Corporation, 2004.

\bibitem[BRS07]{Sho07}
Noel Brady, Tim Riley, and Hamish Short.
\newblock {\em The {G}eometry of the {W}ord {P}roblem for {F}initely
  {G}enerated {G}roups}.
\newblock Springer Science \& Business Media, 2007.

\bibitem[CGS17]{AC17}
Andrea Coladangelo, Koon~Tong Goh, and Valerio Scarani.
\newblock All pure bipartite entangled states can be self-tested.
\newblock {\em Nature Communications}, 8:15485, 2017.

\bibitem[CLS17]{CLS17}
Richard Cleve, Li~Liu, and William Slofstra.
\newblock Perfect commuting-operator strategies for linear system games.
\newblock {\em Journal of Mathematical Physics}, 58(01):012202, 2017.

\bibitem[CM14]{CM14}
Richard Cleve and Rajat Mittal.
\newblock Characterization of binary constraint system games.
\newblock In {\em International Colloquium on Automata, Languages, and
  Programming}, pages 320--331. Springer, 2014.

\bibitem[CMMN20]{CMMN19}
David Cui, Arthur Mehta, Hamoon Mousavi, and Seyed~Sajjad Nezhadi.
\newblock A generalization of {CHSH} and the algebraic structure of optimal
  strategies.
\newblock {\em Quantum}, 4:346, 2020.

\bibitem[CS17]{CS17a}
Andrea Coladangelo and Jalex Stark.
\newblock Robust self-testing for linear constraint system games.
\newblock {\em arXiv:1709.09267}, 2017.

\bibitem[DK05]{DK05}
Reinhard Diestel and Daniela K{\"u}hn.
\newblock Graph minor hierarchies.
\newblock {\em Discrete Applied Mathematics}, 145(2):167--182, 2005.

\bibitem[GH17]{GH17}
William~Timothy Gowers and Omid Hatami.
\newblock Inverse and stability theorems for approximate representations of
  finite groups.
\newblock {\em Sbornik: Mathematics}, 208(12):1784, 2017.

\bibitem[Hol]{KBMAG1995}
Derek Holt.
\newblock {KBMAG}---{K}nuth-{B}endix in monoids and automatic groups, software
  package.
\newblock Available at http://homepages.warwick.ac.uk/\~{}mareg/kbmag/.

\bibitem[HT74]{HT74}
John Hopcroft and Robert Tarjan.
\newblock Efficient planarity testing.
\newblock {\em Journal of the ACM}, 21(4):549--568, 1974.

\bibitem[Kan20]{Kan19}
J{\k{e}}drzej Kaniewski.
\newblock Weak form of self-testing.
\newblock {\em Physical Review Research}, 2:033420, 2020.

\bibitem[KM18]{Kal17}
Amir Kalev and Carl~A. Miller.
\newblock Rigidity of the magic pentagram game.
\newblock {\em Quantum Science and Technology}, 3(1):015002, 2018.

\bibitem[Lov65]{Lov65}
L{\'a}szl{\'o} Lov{\'a}sz.
\newblock On graphs not containing independent circuits.
\newblock {\em Mat. Lapok.}, 16:289--299, 1965.

\bibitem[Mer90]{Mermin90}
David Mermin.
\newblock Simple unified form for the major no-hidden-variables theorems.
\newblock {\em Physical Review Letters}, 65(27):3373--3376, 1990.

\bibitem[MYS12]{McK12}
Matthew McKague, Tzyh~Haur Yang, and Valerio Scarani.
\newblock Robust self-testing of the singlet.
\newblock {\em Journal of Physics A: Mathematical and Theoretical},
  45(45):455304, 2012.

\bibitem[NV17]{NV16}
Anand Natarajan and Thomas Vidick.
\newblock A quantum linearity test for robustly verifying entanglement.
\newblock {\em Proceedings of the 49th Annual ACM SIGACT Symposium on Theory of
  Computing (STOC 2017)}, 2017.

\bibitem[Ol'12]{Ol12}
A.~Yu. Ol'shanskii.
\newblock {\em Geometry of {D}efining {R}elations in {G}roups}, volume~70.
\newblock Springer Science \& Business Media, 2012.

\bibitem[Per91]{Peres91}
Asher Peres.
\newblock Two simple proofs of the {K}ochen-{S}pecker theorem.
\newblock {\em Journal of Physics A: Mathematical and General}, 24(4):L175,
  1991.

\bibitem[PRSS22]{PRSS22}
Connor Paddock, Vincent Russo, Turner Silverthorne, and William Slofstra.
\newblock Supplementary software.
\newblock https://github.com/vprusso/graph\_incidence\_nonlocal\_games, 2022.

\bibitem[RS95]{RS95}
Neil Robertson and Paul~D. Seymour.
\newblock Graph minors. {XIII}. {T}he disjoint paths problem.
\newblock {\em Journal of Combinatorial Theory, Series B}, 63(1):65--110, 1995.

\bibitem[RS04]{RS04}
Neil Robertson and Paul~D. Seymour.
\newblock Graph minors. {XX}. {W}agner's conjecture.
\newblock {\em Journal of Combinatorial Theory, Series B}, 92(2):325--357,
  2004.

\bibitem[{Sag}17]{SageMath}
{Sage Developers}.
\newblock {\em {S}ageMath, the {S}age {M}athematics {S}oftware {S}ystem
  ({V}ersion 8.1)}, 2017.
\newblock {\tt https://www.sagemath.org}.

\bibitem[Slo19a]{Slof17}
William Slofstra.
\newblock The set of quantum correlations is not closed.
\newblock {\em Forum of Mathematics, Pi}, 7(e1), 2019.

\bibitem[Slo19b]{Slof16}
William Slofstra.
\newblock Tsirelson's problem and an embedding theorem for groups arising from
  non-local games.
\newblock {\em J. Amer. Math. Soc.}, 2019.

\bibitem[SW08]{SW08}
Volkher~B. Scholz and Reinhard~F. Werner.
\newblock Tsirelson's problem.
\newblock {\em arXiv:0812.4305}, 2008.

\bibitem[VK33]{VK33}
Egbert~R. Van~Kampen.
\newblock On {S}ome {L}emmas in the {T}heory of {G}roups.
\newblock {\em American Journal of Mathematics}, 55(1):268--273, 1933.

\bibitem[Wag37]{Wag37}
Klaus Wagner.
\newblock {\"U}ber eine {E}igenschaft der ebenen {K}omplexe.
\newblock {\em Mathematische Annalen}, 114(1):570--590, 1937.

\bibitem[WBMS16]{Wu16}
Xingyao Wu, Jean-Daniel Bancal, Matthew McKague, and Valerio Scarani.
\newblock Device-independent parallel self-testing of two singlets.
\newblock {\em Physical Review A}, 93(6):062121, 2016.

\end{thebibliography}

\end{document}